\documentclass[10pt]{amsart}

\usepackage{amsfonts}
\usepackage{amssymb}
\usepackage{amsmath}
\usepackage{amsthm}
\usepackage{float}
\usepackage{hyperref}
\usepackage{wasysym}
\usepackage[all]{xy} \newdir{ >}{{}*!/-10pt/@{>}}
\usepackage{mathrsfs}
\usepackage{dsfont}
\usepackage{enumerate}
\usepackage{tikz}\usetikzlibrary{decorations.markings}
\usetikzlibrary{positioning}
\usepackage[top=3cm,bottom=3cm,left=3cm,right=3cm]{geometry}
\usepackage{lmodern}
\usepackage[T1]{fontenc}
\usepackage{mathtools}
\usepackage{caption}
\usepackage[url=false,isbn=false,doi=false,style=alphabetic, maxbibnames=10,maxcitenames=5,backend=bibtex8]{biblatex}
\bibliography{../bibliography/library}
\appto{\bibsetup}{\emergencystretch=0.5em}

\newcommand{\N}{\mathds N}
\newcommand{\Z}{\mathds Z}

\newcommand{\C}{\mathds C}

\newcommand{\id}{\mathds{1}}

\newcommand{\kk}{\mathds{k}}
\newcommand{\End}{\operatorname{End}}
\let\oldbigwedge\bigwedge
\renewcommand{\bigwedge}{\oldbigwedge\nolimits}

\newcommand{\chuse}[2]{\left[\begin{array}{c} #1 \\ #2\end{array}\right]}

\newtheorem{thm}{Theorem}[section]
\newtheorem{lem}[thm]{Lemma}
\newtheorem{cor}[thm]{Corollary}
\theoremstyle{definition}\newtheorem*{definition}{Definition}\newtheorem{remark}[thm]{Remark}
\newtheorem{example}[thm]{Example}
\newtheorem*{acknowledgement}{Acknowledgement}
\newtheorem*{thm*}{Theorem}
\newtheorem*{conjecture}{Conjecture}

\numberwithin{equation}{subsection}

\begin{document}
\title[A categorification of the skew Howe action on $U_q(\mathfrak{gl}(m|n))$]{A categorification of the skew Howe action on a representation category of $U_q(\mathfrak{gl}(m|n))$}
\author{Jonathan Grant}
\address{Dept of Mathematical Sciences\\ Durham University\\ Science Laboratories\\ South Rd.\\ Durham\\ DH1 3LE\\ UK}
\email{jonathan.grant@durham.ac.uk}

\begin{abstract}
Using quantum skew-Howe duality, we study the category $\operatorname{Rep}(\mathfrak{gl}(m|n))$ of tensor products of exterior powers of the standard representation of $U_q(\mathfrak{gl}(m|n))$, and prove that it is equivalent to a category of ladder diagrams modulo one extra family of relations. We then construct a categorification of $\operatorname{Rep}(\mathfrak{gl}(m|n))$ using the theory of foams. In the case of $n=0$, we show that we can recover $\mathfrak{sl}(m)$ foams introduced by Queffelec and Rose to define Khovanov-Rozansky $\mathfrak{sl}(m)$ link homology. We also define a categorification of the monoidal category of symmetric powers of the standard representation of $U_q(\mathfrak{gl}(n))$, since this category can be identified with $\operatorname{Rep}(\mathfrak{gl}(0|n))$. The relations on our foams are non-local, since the number of dots that can appear on a facet depends on the position of the dot in the foam, rather than just on its colouring. This may be related to Gilmore's non-local relations in Heegaard Floer knot homology.
\end{abstract}

\maketitle

\section{Introduction}
Categorified quantum groups were first defined by Lauda for $\mathfrak{sl}(2)$, and then by Khovanov and Lauda \cite{Khovanov2009,Khovanov2011,Khovanov2010a} and independently by Rouquier \cite{Rouquier2008} for general simple Lie algebras $\mathfrak{g}$. While these have independent algebraic interest, they are also interesting to topologists due to their connections with certain knot homology theories. Webster \cite{Webster2013} used categorified quantum groups to define, for any simple Lie algebra $\mathfrak{g}$, a homology theory for any link coloured with finite-dimensional representations of $\mathfrak{g}$. Moreover, Lauda, Queffelec and Rose \cite{Lauda2012} showed that Khovanov homology could be defined from the categorified quantum groups associated to $\mathfrak{sl}(m)$ over all $m$. Queffelec and Rose \cite{Queffelec} showed in fact that Khovanov-Rozansky homology \cite{Khovanov2004} could be defined from the same categorified quantum groups in a similar way.

One fundamental tool for the work of Lauda, Queffelec and Rose is skew Howe duality, which defines an action of $U_q(\mathfrak{sl}(m))$ on a category of representations of $U_q(\mathfrak{sl}(n))$. This version of skew Howe duality was first proved by Cautis, Kamnitzer and Morrison \cite{Cautis2012}, who used to it show that every morphism between tensor products of exterior powers of the fundamental representation of $U_q(\mathfrak{sl}(n))$ could be described by the graphical calculus of MOY diagrams. In fact, MOY diagrams are shown to arise from ladder diagrams, which describe morphisms in $\dot U_q(\mathfrak{sl}(m))$. The action of $U_q(\mathfrak{sl}(m))$ even defines the braiding on this category of $U_q(\mathfrak{sl}(n))$ representations, making it a very natural setting for studying the knot polynomial associated to exterior powers of the fundamental representation of $\mathfrak{sl}(n)$.

One possible categorification of MOY diagrams is the 2-category of foams. This is a 2-category in which objects are sequences of numbers referring to colourings of the strands, 1-morphisms are certain MOY diagrams, and 2-morphisms are seamed surfaces with a top boundary and bottom boundary that are both MOY diagrams. The foams are required to satisfy certain relations so that they can induce isomorphisms between MOY diagrams categorifying relations between MOY diagrams. Relations between foams for $U_q(\mathfrak{sl}(n))$ were first defined by Bar-Natan \cite{Bar-Natan2004} for $n=2$, Khovanov \cite{Khovanov2003} for $n=3$, and some relations were determined by Mackaay, Sto\v{s}i\`{c} and Vaz \cite{Mackaay2007} using the Kapustin-Li formula for $n\geq 4$. It was not known whether the relations they defined were complete, but they were enough to give an foam-based definition of Khovanov-Rozansky homology.

The main observation of Queffelec and Rose \cite{Queffelec} is that they can define relations on foams by categorifying the action of $\dot U_q(\mathfrak{sl}(m))$ on the category of representations. So foams are defined in such a way that they satisfy the same relations as the 2-morphisms of categorified $U_q(\mathfrak{sl}(m))$. This allows them to define a complete set of relations, and recover Khovanov-Rozansky homology for links and Wu's generalisation  \cite{Wu2009}.

Categorified skew Howe duality also enabled Mackaay and Webster \cite{Mackaay2015} to compare the different constructions of $\mathfrak{sl}(n)$ homology, and determine that they all give rise to isomorphic homology theories.

In this paper, we define an action of $U_q(\mathfrak{gl}(p))$ on the monoidal category $\operatorname{Rep}(m|n)$ of exterior powers of standard representations of the superalgebra $U_q(\mathfrak{gl}(m|n))$. We establish the following

\begin{thm}\label{thm:introthm1}
There is a full functor
\[ \dot U_q(\mathfrak{gl}(p)) \to \operatorname{Rep}(\mathfrak{gl}(m|n)) \]
for each $p\geq 2$.
\end{thm}

We use $\mathfrak{gl}(p)$ here instead of $\mathfrak{sl}(p)$ for notational reasons, the difference is not significant. Since we have a functor for each $p$, it makes sense to combine them all into a single functor from $\dot{U}_q(\mathfrak{gl}(\infty))$. The functor also factors through a quotient category $\dot U_q^\infty(\mathfrak{gl}(\infty)$ giving a full functor
\[ \dot U_q^\infty(\mathfrak{gl}(\infty)) \to \operatorname{Rep}(\mathfrak{gl}(m|n)).\]
The algebra $\dot U_q^\infty(\mathfrak{gl}(\infty))$ has a decomposition $\sum_{\mu}\operatorname{End}_{fr}(V_\infty(\mu))$ where $\mu$ ranges over all dominant weights, $V_\infty(\mu)$ is the irreducible highest-weight representation of $\dot U_q^\infty(\mathfrak{gl}(\infty))$ generated by $\mu$, and $\operatorname{End}_{fr}(V_\infty(\mu))$ is the non-unital algebra of finite-rank endomorphisms of $V_\infty(\mu)$. It turns out the kernel of the functor in Theorem \ref{thm:introthm1} is exactly the sum of $\operatorname{End}_{fr}(V_\infty(\mu))$ with $\mu_{n+1}\not\leq m$. Thus if we define $\dot U_q^{(m|n)}(\mathfrak{gl}(\infty))$ to be the quotient of $\dot U_q^\infty(\mathfrak{gl}(\infty))$ by the kernel, we have

\begin{thm}\label{thm:introthm2}
There is an equivalence of categories
\[\dot U_q^{(m|n)}(\mathfrak{gl}(\infty)) \cong \operatorname{Rep}(\mathfrak{gl}(m|n)). \]
\end{thm}

This Theorem generalises the main theorem of Cautis, Kamnitzer and Morrison \cite{Cautis2012} who investigated the case of $\mathfrak{gl}(m|0)$ and the author \cite{Grant2014} which covers the case of $\mathfrak{gl}(1|1)$. The quotient can be described as adding a single extra family of local relations to ladder diagrams.

It is possible to categorify the above. Highest-weight modules over $U_q(\mathfrak{gl}(p))$ are categorified by the category of projective modules over cyclotomic KLR algebras, as proposed by Khovanov and Lauda \cite{Khovanov2009}, and proved in full generality by Kang and Kashiwara \cite{Kang2012} and independently by Webster \cite{Webster2013}. We generalise this construction to categorify modules $V_\infty(\lambda)$. We define an algebra $R^\lambda$ for each $\lambda\in \Z^\infty$.

\begin{thm}\label{thm:introthm3}
The category $\operatorname{p-mod}R^\lambda$ of finite-dimensional projective left $R^\lambda$-modules categorifies $V_\infty(\lambda)$. In particular, it satisfies
\[ K_0(\operatorname{p-mod}R^\lambda)\otimes_{\Z[q,q^{-1}]} \C(q) \cong V_\infty(\lambda). \]
\end{thm}

Thus we can categorify $\operatorname{End_{fr}}(V_\infty(\lambda))$ by taking the quotient of $\dot{\mathscr{U}}_Q(\mathfrak{gl}(\infty))$ by 2-morphisms that act as $0$ on $\operatorname{p-mod}R^\lambda$. We call the resulting 2-category $\mathcal{E}(V_\infty(\lambda))$. We can then define a 2-category $\mathscr{R}(\mathfrak{gl}(m|n))$ as $\sum_{\mu\in H} \mathcal{E}(V_\infty(\mu))$ where $\sum$ is the coproduct of 2-categories, and $H$ is the set of dominant weights with $\mu_{n+1}\leq m$.

\begin{thm}There is an isomorphism of algebras
\[ K_0(\mathscr{R}(\mathfrak{gl}(m|n)))\otimes_{\Z[q,q^{-1}]} \C(q) \cong \operatorname{Rep}(\mathfrak{gl}(m|n)) \]
which is also an equivalence of categories.
\end{thm}

Since $\dot{\mathscr{U}}^\infty_Q(\mathfrak{gl}(\infty))$ has a graphical description by foams, it follows that the 2-category $\mathscr{R}(\mathfrak{gl}(m|n))$ has a description by foams. Foams in $\mathcal{E}(V_\infty(\mu))$ have an extra relation according to the cyclotomic relation on $R^\mu$. In the case of $\mu=(n,n,\ldots,n,0,\ldots)$, this extra relation is exactly the relation that a foam with $n$ dots on a 1-coloured facet is $0$. In Section \ref{sec:foamssln}, we show how this can be used to recover the foam definition of Khovanov-Rozansky homology as described by Queffelec and Rose \cite{Queffelec}.

We note that in general, cyclotomic relations on foams are non-local, and certainly do not depend only on the colour of the facet of the foam. This makes it harder to have a purely topological description of the foams, since the rules for dot migration will have to be more complicated. This non-locality is discussed in Section \ref{sec:non-local}.

One important motivation for this paper is categorifying the Alexander polynomial. The Alexander polynomial of a knot $K$ arises from the Reshetikhin-Turaev construction applied to the standard representation of $\mathfrak{gl}(1|1)$, as well as from more topological definitions via $\pi_1(S^3\setminus K)$. The first categorification of the Alexander polynomial is due to Ozsv\'{a}th and Szab\'{o} \cite{Ozsvath2004} and (independently) Rasmussen \cite{Rasmussen2003}. This is a Floer-theoretic categorification, so does not \textit{a priori} fit naturally into the plethora of combinatorial quantum knot homologies defined since Khovanov's seminal work \cite{Khovanov1999}. Gilmore \cite{Gilmore2010} gave a cube of resolutions description of Heegaard Floer knot homology, and this description necessarily involved non-local relations. Such non-local relations are somewhat unexpected for those working with quantum knot homologies, but they occur quite naturally in our setting. This may mean there is a connection between Heegaard Floer knot homology and quantum knot invariants.

An alternative categorification of the category of tensor powers of the standard representation of $U_q(\mathfrak{gl}(1|1))$ was given by Sartori in \cite{Sartori2013a} using methods coming from BGG category $\mathcal{O}$. It is not yet clear how his categorification is related to our one.

An obstruction to defining link homologies for $U_q(\mathfrak{gl}(m|n))$ is that the category $\operatorname{Rep}(m|n)$ does not contain duals unless $n=0$. This means it is not possible to describe a closed diagram using only ladder diagrams. Hence it will be necessary to enlarge $\operatorname{Rep}(m|n)$ to include duals if we wish to categorify the link polynomials associated with $U_q(\mathfrak{gl}(m|n))$. Queffelec and Sartori \cite{Queffelec2014,Queffelec2015} have described a doubled Schur algebra, which allows them to include duals in their category of representations, and so they are able to define the link polynomials associated to $U_q(\mathfrak{gl}(m|n))$ and also the HOMFLY polynomial purely from their doubled Schur algebra.

Another interesting special case to study is $\operatorname{Rep}(\mathfrak{gl}(0|n))$. We have $\mathfrak{gl}(n|0)\cong \mathfrak{gl}(0|n)$, and exterior powers of the standard representation $\C^{0|n}_q$ of $\mathfrak{gl}(0|n)$ are symmetric powers of the standard representation $\C^n_q$ of $\mathfrak{gl}(n|0)$. Hence we also obtain a diagram calculus for the monoidal category $\operatorname{Sym}(\mathfrak{gl}(n))$ of symmetric powers of the standard representation of $\mathfrak{gl}(n)$, and we have a categorification of this given by $\mathscr{R}(\mathfrak{gl}(0|n))$. See section \ref{sec:symmetric} for details. An exploration of diagrams for mixed symmetric and exterior powers of the fundamental representation of $\mathfrak{gl}(n)$ was begun by Rose and Tubbenhauer for $n=2$ \cite{Rose2015}, and Tubbenhauer, Vaz and Wedrich for all $n$ \cite{Tubbenhauer2015}. Mixing symmetric and exterior powers simplifies the diagram calculus somewhat and in the case of $n=2$ allows a description of closed diagrams due to cups and caps coming from the Temperley-Lieb algebra.

\subsection{Plan of the paper}
In Section \ref{sec:skewhowe} we recall some basic definitions and prove a version of quantum skew Howe duality for $U_q(\mathfrak{gl}(p))$ acting on the category $\operatorname{Rep}(\mathfrak{gl}(m|n))$ of exterior powers of the standard representation of $U_q(\mathfrak{gl}(m|n))$.

In Section \ref{sec:ladderstop} we define ladders and give a diagrammatic description of morphisms in $\operatorname{Rep}(\mathfrak{gl}(m|n))$ by ladders. We also describe the extra ladder relation that appears in $\operatorname{Rep}(\mathfrak{gl}(m|n))$ and show that is a local relation. For the case of $n=1$, we compute this relation explicitly.

In Section \ref{sec:catquantumgroup}, we recall the definitions of the categorified quantum group for $\mathfrak{gl}(p)$ and 2-representations. We also generalise this to $\mathfrak{gl}(\infty)$.

In Section \ref{sec:foams}, we define the foam 2-category by an equivalence with $\mathscr{U}^\infty_Q(\mathfrak{gl}(\infty))$.

In Section \ref{sec:catirreducibleweightmodules} we recall cyclotomic KLR algebras $R^\lambda$, and generalise these to the case of weights of $\mathfrak{gl}(\infty)$. We prove that the category $\operatorname{p-mod}R^\lambda$ is a categorification of $V_\infty(\lambda)$.

In Section \ref{sec:catofrep}, we use the above to define extra relations on foams in order to categorify $\operatorname{Rep}(\mathfrak{gl}(m|n))$, and give examples.

Finally in Section \ref{sec:specialcases}, we discuss the special cases $\mathfrak{gl}(m|0)$, $\mathfrak{gl}(0|m)$ and $\mathfrak{gl}(1|1)$. The former leads to Khovanov-Rozansky knot homology as defined by Queffelec and Rose, the case of $\mathfrak{gl}(0|m)$ gives a categorification of the category of symmetric powers of the standard representation, and the case of $\mathfrak{gl}(1|1)$ has conjectural connections with Heegaard Floer knot homology.

\begin{acknowledgement}
I would like to thank Daniel Tubbenhauer for helpful discussions at an early stage of this project, and the hospitality of the QGM in hosting me. I would also like to thank Paul Wedrich and Jake Rasmussen for probing questions and comments, and Pedro Vaz for interesting discussions and comments. Thanks also to Antonio Sartori for comments on an earlier version of this paper. This work was supported by an EPSRC doctoral training grant.
\end{acknowledgement}

\section{Skew Howe Duality for \texorpdfstring{$\mathfrak{gl}(m|n)$}{gl(m|n)}}\label{sec:skewhowe}
\subsection{Definitions}\label{sec:definitions}
We are mostly concerned with representations of $U_q(\mathfrak{gl}(m|n))$, so in this section we recall the definition of $U_q(\mathfrak{gl}(m|n))$ and its properties that allow us to form the category $\operatorname{Rep}(\mathfrak{gl}(m|n))$.
\begin{definition}Define $U_q(\mathfrak{gl}(m|n))$ to be the $\Z/2\Z$-graded $\C(q)$-algebra generated by $E_1,\ldots,E_{n+m-1}$, $F_1,\ldots, F_{n+m-1}$, $L_1^{\pm 1},\ldots,L_{m+n}^{\pm 1}$ with $\deg E_m=\deg F_m =1$ and all other generators even. Let $K_i=L_iL_{i+1}^{-1}$ for all $i$. For $i\in \{1,\ldots,m+n\}$ let $\{i\}=1$ if $i\leq m$ and $\{i\}=-1$ if $i>m$. Then the relations on this algebra are
\[ L_iL_j=L_jL_i, \quad \text{for all } \, i,j \]
\[ L_i E_j= q^{\{j\}(\delta_{i,j}-\delta_{i,j+1})}E_j L_i  \]
\[ L_i F_j= q^{\{j\}(\delta_{i,j+1}-\delta_{i,j})}F_j L_i \]
\[ E_iF_i - F_iE_i = \frac{ K_i - K_i^{-1}}{q^{\{i\}}- q^{-\{i\}}}, \,\quad i\neq m \]
\[ E_mF_m + F_m E_m = \frac{ K_m-K_m^{-1}}{q-q^{-1}} \]
\[ E^2_m =F_m^2=0 \] 
\[ E_iE_j = E_jE_i, \,\quad |i-j|>1 \]
\[ F_iF_j=F_jF_i, \,\quad |i-j|>1 \]
\[ E_i^2E_{i\pm 1}-[2]E_iE_{i\pm1}E_i + E_{i\pm 1}E_i^2 = 0,\,\quad i\neq m \]
\[ F_i^2F_{i\pm 1}-[2]F_iF_{i\pm1}F_i + F_{i\pm 1}F_i^2 = 0,\,\quad i\neq m \]
\[ [2]E_mE_{m+1}E_{m-1}E_{m}= E_{m+1}E_{m}E_{m-1}E_{m} + E_{m}E_{m+1}E_{m}E_{m-1} +  E_mE_{m-1}E_mE_{m+1} + E_{m-1}E_mE_{m+1}E_{m} \]
\[[2]F_mF_{m+1}F_{m-1}F_{m}= F_{m+1}F_{m}F_{m-1}F_{m} + F_{m}F_{m+1}F_{m}F_{m-1} + F_mF_{m-1}F_mF_{m+1} + F_{m-1}F_mF_{m+1}F_{m} \]
\end{definition}

Algebras with $\Z/2\Z$-gradings where the grading introduces signs into the defining relations are often referred to as \emph{superalgebras}. Thus the above definition is a quantisation of the universal enveloping superalgebra of the Lie superalgebra $\mathfrak{gl}(m|n)$.

A representation of $U_q(\mathfrak{gl}(m|n))$ is the same thing as a module over $U_q(\mathfrak{gl}(m|n))$, and both terms are freely switched between in the literature.

The algebra $U_q(\mathfrak{gl}(m|n))$ has the structure of a Hopf superalgebra. We choose the coproduct
\[ \Delta(E_i)=E_i\otimes K_i + 1\otimes E_i, \, \Delta(F_i)=F_i\otimes 1 + K_i^{-1}\otimes F_i,\, \Delta(L_i)=L_i\otimes L_i. \]
Using this, given two representations $V,W$ of $U_q(\mathfrak{gl}(m|n))$, we can define a new representation $V\otimes W$ by $X\cdot (v\otimes w) = \Delta(X)v\otimes w$ with the understanding that $A\otimes B(v\otimes w)=(-1)^{\deg B\deg v}Av\otimes Bw$ on homogeneous elements.

The standard (or vector) representation of $U_q(\mathfrak{gl}(m|n))$ is defined to be the $\Z/2\Z$-graded $\C(q)$-vector space $\C^{m|n}_q= \langle v_1,\ldots,v_m,v_{m+1},\ldots v_{m+n}\rangle$, where $\deg v_i=0$ for $i\leq m$ and $\deg v_i=1$ for $i>m$. With
\[ F_i v_i=v_{i+1} \]
\[ E_i v_{i+1}=v_i\]
\[ L_i v_i = qv_i  \]
where $E_i$, $F_i$ act as $0$ otherwise, and $L_k$ acts as identity otherwise.

Due to the fact that the coproduct is not cocommutative, there does not generally exist a map of representations $V\otimes W\to W\otimes V$ that simply flips the tensor factors. However, as usual with quantised universal enveloping algebras, it is possible to modify the flip map to produce an isomorphism
\[ R:V\otimes W\to W\otimes V \]
that commutes with the action of $U_q(\mathfrak{gl}(m|n))$. An important feature of this map is that, unlike the flip map, we have $R^2\neq \id$, which means $R\neq R^{-1}$.

The standard way to introduce the exterior algebra is as a quotient that involves an ideal of symmetric tensors. In other words, the symmetric tensors are the $(+1)$-eigenspace of the flip map $V\otimes V\to V\otimes V$, and the antisymmetric tensors are the $(-1)$-eigenspace. The appearance of the flip map suggests the modification in \cite{Berenstein2008}, where we instead define the symmetric tensors as the union of $(q^r)$-eigenspaces of the map $R:V\otimes V\to V\otimes V$ for all $r\in \Z$. This is due to the fact that all eigenvalues of $R$ are either $q^r$ or $-q^r$. Thus we have the following:
\[\operatorname{Sq}^2(\C^{m|n}_q)=\langle v_i\otimes v_j+(-1)^{\deg (v_i)\deg (v_j)}q v_j\otimes v_i, v_k\otimes v_k\rangle\]
with $i\leq j$, and $k\leq m$. Note that $v_{m+l}\otimes v_{m+l}\notin \operatorname{Sq}^2(\C^{m|n}_q)$ for $l>0$. Forming the two-sided ideal $(\operatorname{Sq}^2(\C^{m|n}_q))$ generated by $\operatorname{Sq}^2(\C^{m|n}_q)$ inside the tensor algebra $T\C^{m|n}_q$, we can define \[\bigwedge_q (\C^{m|n}_q):= T\C^{m|n}_q/\operatorname{Sq}^2(\C^{m|n}_q).\]
By using the $R$-matrix instead of the flip map to define $\operatorname{Sq}^2(\C^{m|n}_q)$, we have ensured that $\bigwedge_q(\C^{m|n}_q)$ is a well-defined representation of $U_q(\mathfrak{gl}(m|n))$.

Throughout the paper we shall use the notation $[n]=\frac{q^n-q^{-n}}{q-q^{-1}}$ for quantum numbers, as well as $[n]!=[n][n-1]\cdots[2]$ and $\chuse{n}{k}=\frac{[n]!}{[n-k]![k]!}$.

An important role is often played by the divided powers, $F_i^{(k)}=\frac{F_i^k}{[k]!}$ and similarly $E_i^{(k)}=\frac{E_i^k}{[k]!}$.

\subsection{Highest-weight representations}
Given a weight $\lambda\in \Z^{m+n}$, the highest-weight representation $V_{m|n}(\lambda)$ is defined by a vector $v^\lambda$ with $L_iv^\lambda=q^{\lambda_i}v^\lambda$ and $E_iv^\lambda=0$ for all $i$, and the relation that $F_i^{(\lambda_i-\lambda_{i+1}+1)}v^\lambda=0$ for $i\leq m-1$, and $F_i^{(-\lambda_i+\lambda_{i+1}+1)}v^\lambda=0$. Note that in the defining relations of $U_q(\mathfrak{gl}(m|n))$, we in fact have $F_m^2=0$, so $F_m^2v^\lambda=0$ regardless of the value of $\lambda_m$ and $\lambda_{m+1}$.

The standard representation $\C^{m|n}_q$ above is the highest-weight representation $V_{m|n}((1,0,\ldots,0))$.

The case we shall mostly be interested in are the highest-weight modules of $U_q(\mathfrak{gl}(m|0))=U_q(\mathfrak{gl}(m))$. These are finite dimensional if and only if $\lambda \in \Z^m$ satisfies $\lambda_1\geq \lambda_2\geq \cdots\geq \lambda_m\geq 0$, that is, if and only if $\lambda$ is a partition of the natural number $\sum_i \lambda_i$. In this case, $\lambda$ is said to be \emph{dominant}.

\subsection{Skew Howe Duality}
An extremely powerful tool for studying the representations $\bigwedge_q(\C^{m|n}_q)$ is skew Howe duality. The quantum version of this was first introduced by Cautis, Kamnitzer and Morrison \cite{Cautis2012}. We will place a particular importance on the algebra $\mathfrak{gl}(p|0)$, which we simply denote $\mathfrak{gl}(p)$. This is the classical Lie algebra, with odd-degree part $0$. In this subsection, we show there is an action of $U_q(\mathfrak{gl}(p))$ on $\left(\bigwedge_q \C^{m|n}_q \right)^{\otimes p}$ that generates the commutant of the $U_q(\mathfrak{gl}(m|n))$ action. This is a special case of a theorem proved by Queffelec and Sartori \cite[Theorem 4.2]{Queffelec2015}.

\begin{lem}\label{lm:flatmodule}
The $q=1$ specialisation of $\bigwedge_q(\C^{m|n}_q\otimes \C^p_q)$ is isomorphic to $\bigwedge(\C^{m|n}\otimes \C^p)$ as modules over $U(\mathfrak{gl}(m|n)\oplus\mathfrak{gl}(p))$.
\end{lem}
\begin{proof}
Letting $\tau_{23}$ be the map that permutes the middle two of four tensor factors, we have that
\[ R_{\C^{m|n}_q\otimes \C^p_q}=\tau_{23}\circ (R_{\C_q^{m|n}}\otimes R_{\C^p_q})\circ \tau_{23} \]
is the $R$-matrix on the standard module $\C^{m|n}_q\otimes \C^p_q$ of $U_q(\mathfrak{gl}(p)\oplus \mathfrak{gl}(m|n))$, so it follows that
\[ \operatorname{Sq}^2(\C_q^{m|n}\otimes \C^p_q)=\tau_{23}\left( (\operatorname{Sq}^2(\C_q^{m|n})\otimes \operatorname{Sq}^2(\C^p_q) ) \oplus (\bigwedge_q^2(\C_q^{m|n})\otimes \bigwedge_q^2(\C^p_q) )\right). \]
Letting $v_1,\ldots,v_{m+n}$ be the standard basis of $\C^{m|n}_q$ and $x_1,\ldots,x_p$ the standard basis of $\C^p_q$, we have a spanning set
\[ (v_i\otimes v_i)\otimes (x_k\otimes x_k), \quad i\leq m \]
\[ (v_i\otimes v_i)\otimes (x_k\otimes x_l +q x_l\otimes x_k ), \quad i\leq m, k<l \]
\[ (v_i\otimes v_j + (-1)^{\deg(v_i)\deg(v_j)}qv_j\otimes v_i) \otimes (x_k\otimes x_k), \quad i<j \]
\[(v_i\otimes v_j + (-1)^{\deg(v_i)\deg(v_j)}qv_j\otimes v_i) \otimes (x_k\otimes x_l +q x_l\otimes x_k ), \quad i<j, k<l \]
\[ (qv_i\otimes v_j - (-1)^{\deg v_i\deg v_j}v_j\otimes v_i)\otimes (qx_k\otimes x_l-x_l\otimes x_k), \quad i<j, k<l \]
\[ (v_i\otimes v_i)\otimes (qx_k\otimes x_l-x_l\otimes x_k), \quad i>m, k<l
\]
of $\tau_{23}\operatorname{Sq}^2(\C_q^{m|n}\otimes \C^p_q)$.
Let $a=v_i\otimes x_k$, $b=v_i\otimes x_l$, $c=v_j\otimes x_k$, $d=v_j\otimes x_l$ with $\Z/2\Z$-grading determined by the degree of $v_i$ or $v_j$, and we take $i<j$, $k<l$. Then in $\bigwedge_q (\C^{m|n}_q\otimes \C^p_q)$ we have
\[ x^2=\frac 1 2 x^2+(-1)^{1+\deg x}\frac 1 2 x^2, \quad \mathrm{for }\, x\in \{a,b,c,d\} \]
\[ab=(-1)^{\deg a} q^{1-2\deg a}ba, \quad
ac = (-1)^{1+\deg a\deg c}q ca \] \[ bd = (-1)^{1+\deg b\deg d}q bd,\quad ad = (-1)^{1+ \deg a \deg d}da
\]
\[ bc +(-1)^{\deg b \deg c}cb = (q-q^{-1})(1+(-1)^{\deg b \deg c})ad.
\]

Then $\bigwedge_q(\C^{m|n}_q\otimes \C^p_q)$ is generated by $\{v_i\otimes x_j\mid 1\leq i\leq m+n,1\leq j\leq p\}$ subject to the relations above. Let $v_{ij}=v_i\otimes x_j$. Then there is a spanning set given by elements of the form $v_{i_1j_1}\wedge\cdots \wedge v_{i_lj_l}$, with $1\leq i_1 <\cdots < i_l \leq m+n$ and $1\leq j_1<\cdots <j_l\leq p$ for all $l\neq \N$. This is linearly independent, since by setting $q=1$ in the relations, we see that it is linearly independent at $q=1$.

Hence the dimension of the $q=1$ specialisation of $\bigwedge_q (\C^{m|n}_q\otimes \C^p_q)$ is equal to that of $\bigwedge (\C^{m|n}\otimes \C^p)$, and so it follows that the $q=1$ specialisation of $\bigwedge_q(\C^{m|n}_q\otimes \C^p_q)$ is isomorphic to $\bigwedge(\C^{m|n}\otimes \C^p)$ as modules over $U(\mathfrak{gl}(m|n)\oplus \mathfrak{gl}(p))$.
\end{proof}

\begin{lem}\label{lm:hookshaped}
There is an isomorphism
\[ \bigwedge_q(\C^{m|n}_q\otimes \C^p_q) \cong \bigoplus_{\mu\in H} V_{m|n}(\mu^t)\otimes V_p(\mu) \]
where $H$ is the set of dominant $\mathfrak{gl}(p)$ weights with $\mu_{n+1}\leq m$, $\mu^t$ is the reflection of the Young diagram about the diagonal, and $V_{m|n}(\mu^t)$ and $V_p(\mu)$ are highest-weight modules of $U_q(\mathfrak{gl}(m|n))$ and $U_q(\mathfrak{gl}(p))$ respectively.
\end{lem}
\begin{proof}
By \cite[Theorem 3.3]{Cheng2001}, we have
\begin{equation}\label{eqn:classicaldecomp}
\bigwedge(\C^{m|n}\otimes \C^p) \cong \bigoplus_{\mu \in H} V_{m|n}(\mu^t)\otimes V_p(\mu)
\end{equation}  
as modules over $U(\mathfrak{gl}(p)\oplus \mathfrak{gl}(m|n))$ (note this is the classical universal enveloping algebra, not the quantum one).

The module $\bigwedge_q (\C^{m|n}_q\otimes \C^p_q)$ decomposes into a direct sum of irreducible modules over $U_q(\mathfrak{gl}(m|n)\otimes \mathfrak{gl}(p))$, which are of the form $V\otimes W$, with $V$ an irreducible over $U_q(\mathfrak{gl}(m|n))$ and $W$ an irreducible over $U_q(\mathfrak{gl}(p))$, and both irreducibles are highest-weight modules. But by Lemma \ref{lm:flatmodule}, $\bigwedge_q(\C^{m|n}_q\otimes \C^p_q)$ specialises at $q=1$ to $\bigwedge(\C^{m|n}\otimes \C^p)$. Since the highest-weight modules of $U_q(\mathfrak{gl}(m|n))$ specialise to highest-weight modules of $U(\mathfrak{gl}(m|n))$, and the decomposition into irreducible is uniquely determined by the dimension of the weight-spaces, it follows that the highest-weight modules appearing in the decomposition of $\bigwedge_q(\C^{m|n}_q\otimes \C^p_q)$ correspond to the highest-weight modules appearing in the classical decomposition \ref{eqn:classicaldecomp}, so the result follows.
\end{proof}

\begin{thm}[Skew Howe duality]\label{th:skewhoweduality}
The actions of $U_q(\mathfrak{gl}(p))$ and $U_q(\mathfrak{gl}(m|n))$ on $\bigwedge_q(\C_q^{m|n}\otimes \C_q^p)$ generate each others commutant. As $U_q(\mathfrak{gl}(m|n))$ representations, there is an isomorphism
\[ \bigwedge_q(\C_q^{m|n}\otimes \C_q^p)\cong \left(\bigwedge_q \C_q^{m|n}\right)^{\otimes p} \]
and the $(\lambda_1,\ldots,\lambda_p)$ weight space for the action of $U_q(\mathfrak{gl}(p))$ is identified with
\[ \bigwedge^{\lambda_1}_q\C^{m|n}_q \otimes \cdots \otimes \bigwedge^{\lambda_p}_q\C^{m|n}_q. \]
\end{thm}
\begin{proof}
By Lemma \ref{lm:hookshaped}, we have a direct sum decomposition into tensor factors, which shows that the actions of $U_q(\mathfrak{gl}(p))$ and $U_q(\mathfrak{gl}(m|n))$ on $\bigwedge_q(\C_q^{m|n}\otimes \C_q^p)$ generate each others commutant.

To define the isomorphism, we let
\[ \phi_j: \bigwedge_q \C^{m|n}_q \to \bigwedge_q(\C^{m|n}_q\otimes \C^p_q): v_i\mapsto v_i\otimes x_j \]
extended linearly, and then we can define
\[ \phi:\left(\bigwedge_q \C^{m|n}_q \right)^{\otimes p}\to \bigwedge_q(\C^{m|n}_q\otimes \C^p_q) = \phi_1\wedge \phi_2\wedge \cdots \wedge \phi_p. \]
By checking the relations in Lemma \ref{lm:flatmodule}, we can see that this is a well-defined map of $U_q(\mathfrak{gl}(m|n))$ representations, since the wedge product commutes with the $U_q(\mathfrak{gl}(m|n))$-action, and the result is clearly a spanning set with dimension equal to the dimension of $\bigwedge_q(\C^{m|n}_q\otimes \C^p_q)$.

The final statement follows by checking the action of each $L_i$ on the right-hand side of the above isomorphisms.
\end{proof}

\begin{thm}\label{thm:directsumdecomp}
There is an isomorphism
\[ \End_{U_q(\mathfrak{gl}(m|n)} \left(\bigwedge_q(\C^{m|n}_q\otimes \C^p_q)\right)\cong \bigoplus_{\mu \in H} \End_{\C(q)} (V_p(\mu)) \]
where $H$ is the set of $\mu\in \Z^p$ satisfying $\mu_{n+1}\leq m$.
\end{thm}
\begin{proof}
By Lemma \ref{lm:hookshaped}, we have
\[\End_{U_q(\mathfrak{gl}(m|n)} \left(\bigwedge_q(\C^{m|n}_q\otimes \C^p_q)\right) \cong \End_{U_q(\mathfrak{gl}(m|n)}\left( \bigoplus_{\mu \in H} V_{m|n}(\mu^t)\otimes V_p(\mu) \right) \cong \bigoplus_{\mu \in H} \End_{\C(q)} (V_p(\mu)) \]
since the $V_{m|n}(\mu^t)$ are simple modules of $U_q(\mathfrak{gl}(m|n))$ with no non-trivial maps between them.
\end{proof}

\begin{example}Take $p=2$ and $(m,n)\not\in \{(1,0),(0,1),(0,0)\}$, and consider $\bigwedge_q^2 (\C^{m|n}_q\otimes \C^2_q)$. This module is $V_{m|n}((1,1))\otimes V_2((2,0))\oplus V_{m|n}((2,0))\otimes V_{2}((1,1))$. Then as a $U_q(\mathfrak{gl}(2))$ module, this is
\[ \xymatrix{\bigwedge_q^0(\C^{m|n}_q)\otimes \bigwedge_q^2(\C^{m|n}_q) \ar@<-.5ex>[r]& \bigwedge_q^1(\C^{m|n}_q)\otimes \bigwedge_q^1(\C^{m|n}_q) \ar@<-.5ex>[l] \ar@<-.5ex>[r] & \bigwedge_q^2(\C^{m|n}_q)\otimes \bigwedge_q^0(\C^{m|n}_q)\ar@<-.5ex>[l]} \]
which is isomorphic to a direct sum $V_2(2,0)\oplus V_2(1,1)$. The arrows denote generators for the spaces of $\C(q)$-linear maps $\End(V_2(2,0))\oplus \End(V_2(1,1))$, noting that $V_2(1,1)$ is a 1-dimensional representation.
\end{example}

\section{Ladder Diagrams}\label{sec:ladderstop}
In this section we define a graphical calculus for an algebra $\dot U_q(\mathfrak{gl}(p))$ closely related to $U_q(\mathfrak{gl}(p))$. This will translate to a graphical calculus on a representation category of modules over $U_q(\mathfrak{gl}(m|n))$.
\subsection{The algebra \texorpdfstring{$\dot U_q (\mathfrak{gl}(p))$}{Uq (gl(p))}}\label{sec:ladders}
We form an algebra $U'_q(\mathfrak{gl}(p))$ by adjoining to $U_q(\mathfrak{gl}(p))$ elements $1_\lambda$ for each $\lambda\in \Z^p$, with the extra relations
\[ 1_\lambda 1_\lambda'=\delta_{\lambda \lambda'}1_\lambda \]
\[ E_i 1_\lambda = 1_{\lambda+\alpha_i} E_i \]
\[ F_i 1_\lambda = 1_{\lambda-\alpha_i} F_i \]
\[ L_i 1_\lambda = q^{\lambda_i} 1_\lambda \]
where $\alpha_i=(0,\ldots,1,-1,\ldots,0)$ with the $1$ in position $i$, and $\lambda_i$ is the $i$th term of $\lambda$.

\begin{definition}We define the non-unital algebra
\[ \dot U_q (\mathfrak{gl}(p)) = \bigoplus_{\lambda,\mu\in \Z^p} 1_\lambda U'_q(\mathfrak{gl}(p)) 1_\mu.\]
\end{definition}

It is often convenient to treat $\dot U_q(\mathfrak{gl}(p))$ as a category, with objects $1_\lambda$ and morphisms $1_\lambda \to 1_\nu$ given by the elements of $1_\nu \dot U_q(\mathfrak{gl}(p)) 1_\lambda$. We shall freely switch between the two viewpoints.

In order to relate the category $\dot U_q(\mathfrak{gl}(p))$ to morphisms on modules over $U_q(\mathfrak{gl}(m|n))$, we will mostly be interested in the following quotient:

\begin{definition}
We define $\dot U_q^{\infty}(\mathfrak{gl}(p))$ to be the quotient of $\dot U_q(\mathfrak{gl}(p))$ by the two-sided ideal generated by the elements $1_\lambda$ that have $\lambda_i<0$ for some $i$.
\end{definition}

Cautis, Kamnitzer and Morrison \cite{Cautis2012} defined \emph{ladder diagrams} to describe morphisms in the category $\dot U^\infty_q(\mathfrak{gl}(p))$.

\begin{definition}
A \emph{ladder} with $p$ \emph{uprights} is a diagram in $[0,1]\times [0,1]$ with $p$ oriented vertical lines connecting the bottom edge to the top edge with horizontal \emph{rungs} joining adjacent uprights. Each line segment is labelled with a natural number such that the algebraic sum of labels at a trivalent vertex is $0$.
\end{definition}

To relate this to $\dot U_q(\mathfrak{gl}(p))$, we associate the following ladders to morphisms in $\dot U_q(\mathfrak{gl}(p))$:

\[E^{(r)}_i1_{\lambda} \mapsto \begin{tikzpicture}[baseline=-0.65ex]
\draw (-0.75,-1) -- (-0.75,1);
\draw (0.75,-1) -- (0.75,1);
\draw (-0.75,0.25) -- (0.75,-0.25);
\draw (-2,-1) -- (-2,1);
\draw (2,-1) -- (2,1);
\draw (0,0.3) node {$r$};
\draw (-1,-1.25) node {$\lambda_i$};
\draw (-1,1.25) node {$\lambda_i+r$};
\draw (-1.4,0) node {$\cdots$};
\draw (1.4,0) node {$\cdots$};
\draw (1,-1.25) node {$\lambda_{i+1}$};
\draw (1,1.25) node {$\lambda_{i+1}-r$};
\draw (-2,-1.25) node {$\lambda_1$};
\draw (2,-1.25) node {$\lambda_m$};
\draw (-2,1.25) node {$\lambda_1$};
\draw (2,1.25) node {$\lambda_m$};
\end{tikzpicture} \]

\[ F^{(r)}_i 1_{\lambda}\mapsto \begin{tikzpicture}[baseline=-0.65ex]
\draw (-0.75,-1) -- (-0.75,1);
\draw (0.75,-1) -- (0.75,1);
\draw (-0.75,-0.25) -- (0.75,0.25);
\draw (-2,-1) -- (-2,1);
\draw (2,-1) -- (2,1);
\draw (0,0.3) node {$r$};
\draw (-1,-1.25) node {$\lambda_i$};
\draw (-1,1.25) node {$\lambda_i-r$};
\draw (-1.4,0) node {$\cdots$};
\draw (1.4,0) node {$\cdots$};
\draw (1,-1.25) node {$\lambda_{i+1}$};
\draw (1,1.25) node {$\lambda_{i+1}+r$};
\draw (-2,-1.25) node {$\lambda_1$};
\draw (2,-1.25) node {$\lambda_m$};
\draw (-2,1.25) node {$\lambda_1$};
\draw (2,1.25) node {$\lambda_m$};
\end{tikzpicture}  \]

Then we define the following relations on ladders:
\[ \begin{tikzpicture}[baseline=-0.65ex]
\draw (-1,-1) -- (-1,1);
\draw (0,-1) -- (0,1);
\draw (1,-1) -- (1,1);
\draw (-1,-0.6) -- (0,-0.2);
\draw (1,0.2) -- (0,0.6);
\draw (-1,-1.25) node {$\lambda_1$};
\draw (0,-1.25) node {$\lambda_2$};
\draw (1,-1.25) node {$\lambda_3$};
\draw (-1.6,0.75) node {$\lambda_1-r$};
\draw (0,1.25) node {$\lambda_2+r+s$};
\draw (1.6,0.75) node {$\lambda_3-s$};
\draw (0.5,0.2) node {$s$};
\draw (-0.5,-0.2) node {$r$};
\end{tikzpicture} =  \begin{tikzpicture}[baseline=-0.65ex]
\draw (-1,-1) -- (-1,1);
\draw (0,-1) -- (0,1);
\draw (1,-1) -- (1,1);
\draw (-1,0.2) -- (0,0.6);
\draw (1,-0.6) -- (0,-0.2);
\draw (-1,-1.25) node {$\lambda_1$};
\draw (0,-1.25) node {$\lambda_2$};
\draw (1,-1.25) node {$\lambda_3$};
\draw (-1.6,0.75) node {$\lambda_1-r$};
\draw (0,1.25) node {$\lambda_2+r+s$};
\draw (1.6,0.75) node {$\lambda_3-s$};
\draw (0.5,-0.2) node {$s$};
\draw (-0.5,0.2) node {$r$};
\end{tikzpicture} \]

\[ \begin{tikzpicture}[baseline=-0.65ex]
\draw (-1,-1) -- (-1,1);
\draw (0,-1) -- (0,1);
\draw (1,-1) -- (1,1);
\draw (-1,-0.2) -- (0,-0.6);
\draw (1,0.6) -- (0,0.2);
\draw (-1,-1.25) node {$\lambda_1$};
\draw (0,-1.25) node {$\lambda_2$};
\draw (1,-1.25) node {$\lambda_3$};
\draw (-1.6,0.75) node {$\lambda_1+r$};
\draw (0,1.25) node {$\lambda_2-r-s$};
\draw (1.6,0.75) node {$\lambda_3+s$};
\draw (0.5,0.2) node {$s$};
\draw (-0.5,-0.2) node {$r$};
\end{tikzpicture} =  \begin{tikzpicture}[baseline=-0.65ex]
\draw (-1,-1) -- (-1,1);
\draw (0,-1) -- (0,1);
\draw (1,-1) -- (1,1);
\draw (-1,0.6) -- (0,0.2);
\draw (1,-0.2) -- (0,-0.6);
\draw (-1,-1.25) node {$\lambda_1$};
\draw (0,-1.25) node {$\lambda_2$};
\draw (1,-1.25) node {$\lambda_3$};
\draw (-1.6,0.75) node {$\lambda_1+r$};
\draw (0,1.25) node {$\lambda_2-r-s$};
\draw (1.6,0.75) node {$\lambda_3+s$};
\draw (0.5,-0.2) node {$s$};
\draw (-0.5,0.2) node {$r$};
\end{tikzpicture} \]

\[ \begin{tikzpicture}[baseline=-0.65ex]
\draw (-0.5,-1) -- (-0.5,1);
\draw (0.5,-1) -- (0.5,1);
\draw (-0.5,-0.6) -- (0.5,-0.2);
\draw (-0.5,0.2) -- (0.5,0.6);
\draw (-0.5,-1.25) node {$\lambda_1$};
\draw (0.5,-1.25) node {$\lambda_2$};
\draw (-1,1.25) node {$\lambda_1-r-s$};
\draw (1,1.25) node {$\lambda_2+r+s$};
\draw (0,-0.6) node {$r$};
\draw (0,0.6) node {$s$};
\end{tikzpicture} = \chuse{r+s}{s}
\begin{tikzpicture}[baseline=-0.65ex]
\draw (-0.5,-1) -- (-0.5,1);
\draw (0.5,-1) -- (0.5,1);
\draw (-0.5,-0.2) -- (0.5,0.2);
\draw (-0.5,-1.25) node {$\lambda_1$};
\draw (0.5,-1.25) node {$\lambda_2$};
\draw (-1,1.25) node {$\lambda_1-r-s$};
\draw (1,1.25) node {$\lambda_2+r+s$};
\draw (0,-0.3) node {$r+s$};
\end{tikzpicture}
 \]
 
\[ \begin{tikzpicture}[baseline=-0.65ex]
\draw (-0.5,-1) -- (-0.5,1);
\draw (0.5,-1) -- (0.5,1);
\draw (-0.5,-0.6) -- (0.5,-0.2);
\draw (0.5,0.2) -- (-0.5,0.6);
\draw (-0.5,-1.25) node {$\lambda_1$};
\draw (0.5,-1.25) node {$\lambda_2$};
\draw (-1,1.25) node {$\lambda_1-s+r$};
\draw (1,1.25) node {$\lambda_2+s-r$};
\draw (0,-0.6) node {$s$};
\draw (0,0.6) node {$r$};
\draw (-1.1,0) node {$\lambda_1-s$};
\draw (1.1,0) node {$\lambda_2+s$};
\end{tikzpicture} = \sum_t \chuse{\lambda_1-\lambda_2+r-s}{t}
\begin{tikzpicture}[baseline=-0.65ex]
\draw (-0.5,-1) -- (-0.5,1);
\draw (0.5,-1) -- (0.5,1);
\draw (-0.5,-0.2) -- (0.5,-0.6);
\draw (0.5,0.6) -- (-0.5,0.2);
\draw (-0.5,-1.25) node {$\lambda_1$};
\draw (0.5,-1.25) node {$\lambda_2$};
\draw (-1,1.25) node {$\lambda_1-r+s$};
\draw (1,1.25) node {$\lambda_2+r-s$};
\draw (0,-0.7) node {$r-t$};
\draw (0,0.7) node {$s-t$};
\draw (-1.4,0) node {$\lambda_1+r-t$};
\draw (1.4,0) node {$\lambda_2-r+t$};
\end{tikzpicture}
\]

\[ \begin{tikzpicture}[baseline=-0.65ex]
\draw (-1,-1.4) -- (-1,1.4);
\draw (0,-1.4) -- (0,1.4);
\draw (1,-1.4) -- (1,1.4);
\draw (-1,-1) -- (0,-0.6);
\draw (-1,-0.2) -- (0,0.2);
\draw (0,0.6) -- (1,1);
\draw (-1,-1.65) node {$\lambda_1$};
\draw (0,-1.65) node {$\lambda_2$};
\draw (1,-1.65) node {$\lambda_3$};
\draw (-0.5,-1) node {$1$};
\draw (-0.5,0.2) node {$1$};
\draw (0.5,1) node {$1$};
\end{tikzpicture}
- [2] 
\begin{tikzpicture}[baseline=-0.65ex]
\draw (-1,-1.4) -- (-1,1.4);
\draw (0,-1.4) -- (0,1.4);
\draw (1,-1.4) -- (1,1.4);
\draw (-1,-1) -- (0,-0.6);
\draw (0,-0.2) -- (1,0.2);
\draw (-1,0.6) -- (0,1);
\draw (-1,-1.65) node {$\lambda_1$};
\draw (0,-1.65) node {$\lambda_2$};
\draw (1,-1.65) node {$\lambda_3$};
\draw (-0.5,-1) node {$1$};
\draw (0.5,0.2) node {$1$};
\draw (-0.5,1) node {$1$};
\end{tikzpicture} + 
\begin{tikzpicture}[baseline=-0.65ex]
\draw (-1,-1.4) -- (-1,1.4);
\draw (0,-1.4) -- (0,1.4);
\draw (1,-1.4) -- (1,1.4);
\draw (0,-1) -- (1,-0.6);
\draw (-1,-0.2) -- (0,0.2);
\draw (-1,0.6) -- (0,1);
\draw (-1,-1.65) node {$\lambda_1$};
\draw (0,-1.65) node {$\lambda_2$};
\draw (1,-1.65) node {$\lambda_3$};
\draw (0.5,-1) node {$1$};
\draw (-0.5,0.2) node {$1$};
\draw (-0.5,1) node {$1$};
\end{tikzpicture} = 0
\]
\[\begin{tikzpicture}[baseline=-0.65ex]
\draw (-0.5,-1) -- (-0.5,1);
\draw (0.5,-1) -- (0.5,1);
\draw (-0.5,-0.6) -- (0.5,-0.6);
\draw (1.5,-1) -- (1.5,1);
\draw (2.5,-1) -- (2.5,1);
\draw (1.5,0.6) -- (2.5,0.6);
\draw (-0.5,-1.25) node {$k_1$};
\draw (0.5,-1.25) node {$k_2$};
\draw (1.5,-1.25) node {$k_3$};
\draw (2.5,-1.25) node {$k_4$};
\draw (0,-0.6) node [label=above:$r$] {};
\draw (2,0.6) node [label=above:$s$] {};
\draw (-1,1.25) node {};
\draw (1,1.25) node {};
\draw (-1.1,0) node {};
\draw (1.1,0) node {};
\draw (1,0) node {$\cdots$};
\end{tikzpicture} =
\begin{tikzpicture}[baseline=-0.65ex]
\draw (-0.5,-1) -- (-0.5,1);
\draw (0.5,-1) -- (0.5,1);
\draw (-0.5,0.6) -- (0.5,0.6);
\draw (1.5,-1) -- (1.5,1);
\draw (2.5,-1) -- (2.5,1);
\draw (1.5,-0.6) -- (2.5,-0.6);
\draw (-0.5,-1.25) node {$k_1$};
\draw (0.5,-1.25) node {$k_2$};
\draw (1.5,-1.25) node {$k_3$};
\draw (2.5,-1.25) node {$k_4$};
\draw (0,0.6) node [label=above:$r$] {};
\draw (2,-0.6) node [label=above:$s$] {};
\draw (1,1.25) node {};
\draw (1.1,0) node {};
\draw (1,0) node {$\cdots$};
\end{tikzpicture}
\]
with either orientation on each of the rungs in the last relation as long as the two $r$-coloured rungs have the same orientation, and similarly for the two $s$-coloured rungs. We also take mirror images of the third and fifth relations, and include all relations with arbitrarily many uprights on each side.

Since these relations were imposed to match the relations on morphisms in $\dot U_q(\mathfrak{gl}(p))$, we get the following:
\begin{lem}\label{lem:ladders}
The category $\dot U_q^\infty(\mathfrak{gl}(p))$ is equivalent to the category of ladders on $p$ uprights.
\end{lem}

\subsection{Relationship with modules of \texorpdfstring{$U_q(\mathfrak{gl}(m|n))$}{Uq(gl(m|n))}}
We want a full description of the following category of $U_q(\mathfrak{gl}(m|n))$-modules:

\begin{definition}
We let $\operatorname{Rep}(\mathfrak{gl}(m|n))$ be the category monoidally generated by $\bigwedge_q^{k}(\C^{m|n}_q)$ for all $k$. That is, objects in the category are direct sums of
\[ \bigwedge_q^{k_1}(\C^{m|n}_q)\otimes \cdots \otimes\bigwedge_q^{k_p}(\C^{m|n}_q) \]
for all $(k_1,\ldots,k_p)\in \N^p$, and all $p\in \N$. Morphisms in the category are all $U_q(\mathfrak{gl}(m|n))$-module morphisms.
\end{definition}

As a major corollary to skew Howe duality (Theorem \ref{th:skewhoweduality}) we have the following:

\begin{thm}\label{thm:fullfunctor}
There is a full functor
\[ \dot U_q^\infty(\mathfrak{gl}(p)) \to \operatorname{Rep}(\mathfrak{gl}(m|n)) \]\end{thm}
\begin{proof}
The functor is defined by sending $1_\lambda$ to $\bigwedge_q^{\lambda_1}(\C^{m|n}_q)\otimes \bigwedge_q^{\lambda_2}(\C^{m|n}_q)\otimes\cdots\otimes \bigwedge_q^{\lambda_p}(\C^{m|n}_q)$ and sending
\[ 1_\nu \dot U_q(\mathfrak{gl}(p)) 1_\lambda \to \operatorname{Hom}_{U_q(\mathfrak{gl}(m|n))}\left( \bigwedge_q^{\lambda_1}(\C^{m|n}_q)\otimes \cdots\otimes \bigwedge_q^{\lambda_p}(\C^{m|n}_q), \bigwedge_q^{\nu_1}(\C^{m|n}_q)\otimes \cdots\otimes \bigwedge_q^{\nu_p}(\C^{m|n}_q)\right) \]
by the action of $U_q(\mathfrak{gl}(p))$ in Theorem \ref{th:skewhoweduality}. Since this action generates the commutant of the action of $U_q(\mathfrak{gl}(m|n))$, it follows that this functor is full.
\end{proof}

\subsection{Extra \texorpdfstring{$\mathfrak{gl}(m|n)$}{gl(m|n))} relations on ladder diagrams}
The functor in Theorem \ref{thm:fullfunctor} is not faithful for all $p$. However, in the special case $n=0$ (corresponding to ordinary Lie algebras $\mathfrak{gl}(m)$, the kernel is easy to describe: it is generated by morphisms that factor through a weight $1_\lambda$ with $\lambda_i>m$ for some $i$. This is proved in \cite{Cautis2012}, and gives a complete description of the categories $\operatorname{Rep}(\mathfrak{gl}(m))$ and $ \operatorname{Rep}(\mathfrak{sl}(m))$, since all morphisms can be simply given uniquely by ladders with colours bounded above by $m$.

In \cite{Grant2014}, we also give a description of the kernel in the case $m=n=1$, which turns out to be a little more complicated. For general $m,n$ it seems very difficult to give closed formulas for the extra relations (in terms of ladder diagrams), but we can describe how they arise.

\begin{lem}\label{algiso}
Let $\lambda$ be a dominant weight, and let $L(\lambda)$ be the set of all dominant weights dominated by $\lambda$. Then there is an isomorphism of algebras
\[ \dot{U}_q(\mathfrak{gl}(p))/I_{\lambda} \to \bigoplus_{l\in L(\lambda)} \operatorname{End}_{\C(q)}(V_p(l)). \]
\end{lem}
\begin{proof}
This is Lemma 4.4.2 in \cite{Cautis2012}.
\end{proof}

Note that we have $\dot U_q^\infty (\mathfrak{gl}(p))\cong \bigoplus_{K\in \N}\dot{U}_q(\mathfrak{gl}(p))/I_{(K,0,\ldots,0)}$. Hence Lemma \ref{algiso} implies that there exists a system of orthogonal central idempotents $e_l\in \dot{U}^\infty_q(\mathfrak{gl}(p))$ corresponding to the direct summation on the right hand side.

By Theorem \ref{thm:directsumdecomp}, we have
\[ \operatorname{End}_{U_q(\mathfrak{gl}(m|n))}\left( \bigwedge_q(\C^{m|n}_q\otimes \C^p_q)\right) \cong \bigoplus_{\mu\in H} \operatorname{End}_{\C(q)}(V_p(\mu)). \]
Thus there is a map
\[ \dot{U}^\infty_q(\mathfrak{gl}(p)) \to \operatorname{End}_{U_q(\mathfrak{gl}(m|n))}\left( \bigwedge_q(\C^{m|n}_q\otimes \C^p_q)\right) \]
which simply corresponds to projection by $\sum_{l\in H} e_l$.

By Theorem \ref{algiso}, we have a functor
\[ \dot U_q^\infty(\mathfrak{gl}(p))\to \operatorname{Rep}(\mathfrak{gl}(m|n)) \]
which factors through
\[ \dot U_q^\infty (\mathfrak{gl}(p))\sum_{\mu\in H}e_\mu \]
such that the induced functor
\[ \dot U_q^\infty (\mathfrak{gl}(p))\sum_{\mu\in H}e_\mu \to \operatorname{Rep}(\mathfrak{gl}(m|n)) \]
is full and faithful. This gives us our desired full description of the relations on 
$\operatorname{Rep}(\mathfrak{gl}(m|n))$,
as any morphism in $\operatorname{Rep}(\mathfrak{gl}(m|n))$ can be described by ladder diagrams. This description is unique up to the ladder relations in Section \ref{sec:ladders}, and the relation that $\sum_{\mu\in H}e_\mu$ is the identity on $\operatorname{Rep}(\mathfrak{gl}(m|n))$.

\begin{definition}
We let $\dot U_q^{(m|n)}(\mathfrak{gl}(p))$ be the category $\dot U_q^\infty (\mathfrak{gl}(p))\sum_{\mu\in H}e_\mu$.
\end{definition}

\begin{thm}\label{thm:descriptionOfRep}
The functor
\[ \dot U_q^{(m|n)}(\mathfrak{gl}(p)) \to \operatorname{Rep}(\mathfrak{gl}(m|n)) \]
is full and faithful, and the induced functor
\[ \bigoplus_{p=2}^\infty \dot U_q^{(m|n)}(\mathfrak{gl}(p)) \to \operatorname{Rep}(\mathfrak{gl}(m|n)) \]
is an equivalence of categories.
\end{thm}
\begin{proof}
The first part is discussed above. For the second part, simply note that the functor is essentially surjective and each summand is full and faithful.
\end{proof}

The first non-trivial example of this is $\dot{U}_q^{(1,0)}(\mathfrak{gl}(2))$ and $\dot{U}_q^{(0,1)}(\mathfrak{gl}(2))$.
\begin{example} Consider $\dot{U}_q(\mathfrak{gl}(2))/I_{(2,0)}$, which has a basis
\[ 1_{(2,0)},\, 1_{(1,1)}, \, 1_{(0,2)}, \, F1_{(2,0)},\, F^{(2)}1_{(2,0)}, \, E1_{(1,1)}, \, F1_{(1,1)}, \, E1_{(0,2)}, \, E^{(2)}1_{(0,2)}, \, EF1_{(1,1)}. \]
Now $\dot{U}_q(\mathfrak{gl}(2))/I_{(2,0)} \cong \End_{\C(q)}(V_2(2,0))\oplus \End_{\C(q)}(V_2(1,1))$, so there must exist orthogonal idempotents corresponding to this decomposition. One notes that $EF1_{(1,1)}\cdot V_2(2,0)=[2]1_{(1,1)}\cdot V_2(2,0)$, while $EF1_{(1,1)}\cdot V_2(1,1)=0$ as this representation is 1-dimensional.

Hence, we find that $1_{(2,0)}+\frac{1}{[2]}EF1_{(1,1)}+1_{(0,2)}$ is an idempotent projecting to $\End_{\C(q)}(V_2(2,0))$ and $1_{(1,1)}-\frac{1}{[2]}EF1_{(1,1)}$ is an idempotent projecting to $\End_{\C(q)}(V_2(1,1))$.

Hence the algebra $1_{(1,1)}\dot U_q^{(1|0)}(\mathfrak{gl}(2))1_{(1,1)}$ is defined by \[ 1_{(1,1)}\dot U_q^{\infty}(\mathfrak{gl}(2))(1_{(1,1)}-\frac{1}{[2]}EF1_{(1,1)})\]
and $1_{(1,1)}\dot U_q^{(0|1)}(\mathfrak{gl}(2))1_{(1,1)}$ is defined by
\[1_{(1,1)}\dot U_q^{\infty}(\mathfrak{gl}(2))(\frac{1}{[2]}EF1_{(1,1)}).\]
\end{example}

\subsection{Branching rules and locality of relations}\label{sec:branchingrules}
Here we wish to verify that there is a well-defined inclusion $\dot U_q^{(m|n)}(\mathfrak{gl}(p))\hookrightarrow \dot U_q^{(m|n)}(\mathfrak{gl}(p+1))$ induced by the inclusion $\dot U_q(\mathfrak{gl}(p))\to \dot U_q(\mathfrak{gl}(p+1))$ and establish that the extra relations this imposes on $\operatorname{Rep}(\mathfrak{gl}(m|n))$ are indeed local, and do not depend on addition of strands on either side.

\begin{definition}
For each $j\geq 0$ define the inclusion
\[ \iota_j:\dot U_q(\mathfrak{gl}(p)) \to \dot U_q(\mathfrak{gl}(p+1)) \]
on objects by $1_{(\lambda_1,\ldots,\lambda_p)} \mapsto 1_{(\lambda_1,\ldots,\lambda_p,j)}$, and on morphisms by $E_i1_{(\lambda_1,\ldots,\lambda_p)} \mapsto E_i1_{(\lambda_1,\ldots,\lambda_p,j)}$ and $F_i1_{(\lambda_1,\ldots,\lambda_p)} \mapsto F_i1_{(\lambda_1,\ldots,\lambda_p,j)}$. This is well-defined since relations are mapped to relations.
\end{definition}

It is clear that this inclusion descends to an  inclusion $\dot U_q^\infty(\mathfrak{gl}(p)) \to \dot U_q^\infty(\mathfrak{gl}(p+1))$ since weights with non-negative entries are carried to the same. This inclusion gives rise to a restriction functor taking modules over $\dot U_q(\mathfrak{gl}(p+1))$ to modules over $\dot U_q(\mathfrak{gl}(p))$, by simply forgetting the action of the $E_p,F_p,L_{p+1}$. We denote the restriction of a $\dot U_q(\mathfrak{gl}(p+1))$-module $M$ to a $\dot U_q(\mathfrak{gl}(p))$-module by $M|_{p}$.

We have the following well-known theorem:

\begin{thm}\label{thm:branchingrule}
There is an isomorphism of $\dot U_q(\mathfrak{gl}(p))$-modules
\[ V_{p+1}(\lambda)|_{p} \cong \bigoplus_\nu V_p(\nu) \]
where the sum is over all dominant weights $\nu$ satisfying $\lambda_{i+1}\leq \nu_i\leq \lambda_{i}$.
\end{thm}
\begin{proof}
The classical case is well-known (see, for example, \cite{Itzykson1966,Vaz2013}). The quantum case then follows since irreducible highest-weight modules specialise at $q=1$ to irreducible highest-weight modules.
\end{proof}

Hence we have 
\begin{equation}\label{eqn:branching}
\operatorname{End}_{\C(q)}(V_{p+1}(\lambda)|_{p}) \cong \bigoplus_\nu \operatorname{End}_{\C(q)}(V_p(\nu)).\end{equation} 
 
We can now state the following:

\begin{thm}\label{thm:inclusion}
The inclusion $\iota_j:\dot U^\infty_q(\mathfrak{gl}(p)) \to \dot U^\infty_q(\mathfrak{gl}(p+1))$ descends to a well-defined inclusion
\[\dot U_q^{(m|n)}(\mathfrak{gl}(p))\hookrightarrow \dot U_q^{(m|n)}(\mathfrak{gl}(p+1)). \]
Thus the additional ladder relation in $\operatorname{Rep}(\mathfrak{gl}(m|n))$ remains true if strands are added to the right of the ladder diagrams.
\end{thm}
\begin{proof}
If $\lambda=(\lambda_1,\ldots,\lambda_{p+1})$ is such that $\lambda_{n+1}\leq m$, then $\nu_{n+1}\leq m$ for all $\nu$ in the direct sum in Theorem \ref{thm:branchingrule}. So $\lambda\in H$ implies $\nu\in H$ for such $\nu$. Hence if $\nu\not\in H$ in the decomposition in equation \ref{eqn:branching}, then $\lambda\not\in H$ also, so the inclusion of the quotient is well-defined.
\end{proof}

Note that there is also an inclusion $\dot U_q^\infty(\mathfrak{gl}(p)) \to \dot U_q^\infty(\mathfrak{gl}(p+1))$ where $1_{\lambda}\mapsto 1_{(j,\lambda)}$ and $E_i1_{\lambda}\mapsto E_{i+1}1_{(j,\lambda)}$, $F_i1_\lambda\mapsto F_{i+1}1_{(j,\lambda)}$. This inclusion also has restriction functors, and has the same branching rule as in Theorem \ref{thm:branchingrule}. Hence we have:

\begin{cor}
The additional relation in $\operatorname{Rep}(\mathfrak{gl}(m|n))$ remains true if strands are added to the left or to the right of the ladder diagrams.
\end{cor}

Hence we can deduce that the extra relation on $\operatorname{Rep}(\mathfrak{gl}(m|n))$ is a local one, and its complexity is governed by $n$.

\begin{thm}\label{thm:localrelation}
The relation on $\operatorname{Rep}(\mathfrak{gl}(m|n))$ is generated as a local relation by the identity $\sum_{\mu\in H}e_\mu=\id$ in $\dot U_q^\infty(\mathfrak{gl}(n+1))$.
\end{thm}
\begin{proof}
The identity on $V_{p}(\lambda)$ can be written as a sum of ladders starting from the different weight-spaces that all factor through the identity on the highest-weight vector. If $\lambda_{n+1}>m$, then $\dot U_q^{(m|n)}(\mathfrak{gl}(n+1))$ acts as $0$ on the restriction of the highest-weight vector, since it is carried to the highest-weight module $V_{n+1}(\lambda|_{n+1})$, which has its $(n+1)$-th term $>m$. So in particular the identity on the restriction of the highest-weight vector is equal to an element in $\dot U_q^{(m|n)}(\mathfrak{gl}(n+1))$ acting on $V_{n+1}(\lambda|_{n+1})$ as $0$. The identity on the restriction of the highest-weight vector is the restriction of the identity on the highest-weight vector, so in particular, $\dot U_q^{(m|n)}(\mathfrak{gl}(p))$ acts as $0$ on $V_{p}(\lambda)$. Since $\dot U_q^{(m|n)}(\mathfrak{gl}(p))$ is characterised by acting as $0$ on $V_{p}(\lambda)$ when $\lambda_{n+1}>m$, and as $\dot U_q(\mathfrak{gl}(p))$ when $\lambda_{n+1}\leq m$, it follows that $\dot U_q^{(m|n)}(\mathfrak{gl}(p))$ is determined by the relation that $\sum_{\mu\in H}e_\mu=\id$ with any number of strands on either side.
\end{proof}

\subsection{The special case \texorpdfstring{$p=2$}{p=2}}
The case where $n=0$ is the simplest, as shown by Cautis, Kamnitzer and Morrison \cite{Cautis2012}, as the only extra relation is killing ladders that involve a colour higher than $m$. The case $n=1$ is the next easiest, where we can explicitly compute the local relation for all $m\geq 1$, since the relation is contained in $\dot U_q(\mathfrak{gl}(2))$ by Theorem \ref{thm:localrelation}.
\begin{thm}\label{thm:generalprojector}
The projection to $\End_{\C(q)}(V_2(k+l-m,m))\cdot 1_{(k,l)}$ in $\dot{U}_q(\mathfrak{gl}(2))/I_{(k+l,0)}$ is
\[e^{k,l}_m= \sum_{t=0}^m (-1)^t\frac{\chuse{l-m+t}{t}}{\chuse{k+l-2m+t}{l-m+t}}\frac{[k+l-2m+1]}{[k+l-2m+1+t]} F^{(l-m+t)}E^{(l-m+t)}1_{(k,l)} \]
\end{thm}
\begin{proof}
We work by induction on $m$, over all $k$ and $l$. First observe that when $m=0$, the expression becomes
\[ \frac{1}{\chuse{k+l}{l}}F^{(l)}E^{(l)}1_{(k,l)} \]
which is clearly idempotent by the fact that
\[ E^{(l)}F^{(l)}1_{(k+l,0)}=\chuse{k+l}{l}1_{(k+l,0)}.\]
Its image is $\End_{\C(q)} (V_2(k+l,0))\cdot 1_{(k,l)}$ because it acts as the identity on the highest-weight vector in $V_2(k+l,0)$ but as $0$ on any highest weight module $V_2(\mu)$ with $\mu <(k+l,0)$.

Now suppose we have constructed the idempotent projecting to $\End_{\C(q)} (V_2(k+l-m-1,m+1))\cdot 1_{(k,l)}$ for all $k,l$. We wish to construct the idempotent $e$ for $\End_{\C(q)} (V_2(k+l-m,m))\cdot 1_{(k,l)}$. Now,
\[ \End_{\C(q)} (V_2(k+l-m,m))\cong \dot{U}_q(\mathfrak{gl}(2))[\geq (k+l-m,m)]/\dot{U}_q(\mathfrak{gl}(2))[>(k+l-m,m)] \]
hence there is a surjective map
\[ \dot{U}_q(\mathfrak{gl}(2))[\geq (k+l-m,m)] \to \End_{\C(q)}(V_2(k+l-m,m)). \]
Therefore elements of $\End_{\C(q)}(V_2(k+l-m,m))$ are represented by linear combinations of canonical basis elements acting only on $V_2(\mu)$ with $\mu\geq (k+l-m,m)$, and hence
\[e =\sum_{t=0}^m c_t F^{(l-m+t)}E^{(l-m+t)}1_{(k,l)} \]
for some coefficients $c_t$. To determine $c_t$, we appeal to the fact that $e^2=e$ and the induction hypothesis.

We calculated the idempotent $e_1$ for $\End_{\C(q)} (V_2(k+l-2-m-1,m+1))\cdot 1_{(k-1,l-1)}$ by induction, with coefficients as in the statement of the theorem. To calculate $e^2$, one takes terms
\[ F^{(l-m+t)}E^{(l-m+t)}F^{(l-m+s)}E^{(l-m+s)}1_{(k,l)} \]
and simplifies by writing $E^{(l-m+t)}F^{(l-m+s)}1_{(k+l-m+s,m-s)}$ in terms of $F^{(a)}E^{(b)}1_{(k+l-m+s,m-s)}$ and then combining the powers of $F$ and $E$ by using
\[ F^{(a)}F^{(b)}=\chuse{a+b}{b}F^{(a+b)} \]
and similarly for $E$.

The observation is that that in $e_1$, we have terms $F^{(l-m+t)}E^{(l-m+t)}1_{(k-1,l-1)}$ since the $+1$'s and $-1$'s cancel in the powers of $F$ and $E$, so the commutator for $F$ and $E$ is used on
\[ E^{(l-m+t)}F^{(l-m+s)}1_{(k+l-1-m+s,m-s-1)} \]
and since the coefficients depend only on the difference of the components (ie. the number $k+l-1-m+s - (m-s-1) = k+l-2m+2s$) the coefficients that occur here are the same for $k,l,m$ as they are for $k-1,l-1,m-1$.

Thus, since we must have $e^2=e$, the first $m-1$ coefficients in $e$ must be exactly those appearing in $e_1$, which agrees with the stated formula.

To determine the final $t=m$ term in $e$, we can exploit the relationship between $\dot{U}_q(\mathfrak{gl}(2))/I_{(k+l,0)}$ and the representation theory of $U_q(\mathfrak{gl}(1|1))$ already proved in \cite{Grant2014}. Since $e$ projects to $\End_{\C(q)}(V_2(k+l-m,m))$, then $e$ acts as $0$ when considered as a morphism of representations in $U_q(\mathfrak{gl}(1|1))$ unless $m=1$ or $m=0$. Since we already know the case $m=0$, we check $m>0$.

Applying the morphism to the element $w^k\otimes w^l\in \bigwedge^k_q\C^{1|1}_q\otimes \bigwedge^l_q\C^{1|1}_q$ always gives $0$, since the $m=0$ idempotent already acts as the identity on $w^k\otimes w^l$. Hence we can calculate the final coefficient $c_m$ by knowing the first $m-1$ coefficients and by knowing that the morphism applied to $w^k\otimes w^l$ gives $0$.

We know from \cite{Grant2014} that $F^{(l-m+t)}E^{(l-m+t)}w^k\otimes w^l=\chuse{l}{l-m+t}\chuse{k+l-m+t}{k}w^k\otimes w^l$, hence we deduce that
\[c_m= -\frac{1}{\chuse{k+l}{k}}\sum_{t=0}^{m-1} (-1)^t\frac{\chuse{l-m+t}{t}}{\chuse{k+l-2m+t}{l-m+t}}\frac{[k+l-2m+1]}{[k+l-2m+1+t]}\chuse{l}{l-m+t}\chuse{k+l-m+t}{k}. \]
Use of identities of quantum binomials then yields
\[ c_m=(-1)^m \frac{1}{\chuse{k+l-m}{l}}\chuse{l}{m}\frac{[k+l-2m+1]}{[k+l-m+1]} \]
as required.

To see it has the correct image, simply note that all terms except the $t=0$ term lie in $\dot{U}_q(\mathfrak{gl}(2))[>(k+l-m,m)]$, hence act as $0$ on $V_2(k+l-m,m)$, and the $t=0$ term acts as the identity on the $(k,l)$ weight space of $V_2(k+l-m,m)$. 
\end{proof}

For the case of $U_q(\mathfrak{gl}(m|1))$, the requirement on $\mu\in H$ is $\mu_2\leq m$. This gives a direct sum
\[ \operatorname{End}_{U_q(\mathfrak{gl}(m|n))}\left( \bigwedge_q(\C^{m|n}_q\otimes \C^2_q)\right) \cong \bigoplus_{\mu\in H} \operatorname{End}_{\C(q)}(V_2(\mu)) \]
and by Theorem \ref{th:skewhoweduality}, the left-hand side is identified with morphisms in $\operatorname{Rep}(\mathfrak{gl}(m|n))$ between the tensor products of two simple objects $\bigwedge^k(\C^{m|n}_q)$.

Hence the projection $\dot U^\infty_q(\mathfrak{gl}(2))$ is multiplication by
\[ \sum_{k,l\in \N}\sum_{i=1}^m e^{k,l}_i \]
where the $e_k$ are as in Theorem \ref{thm:generalprojector}. This of course acts as a finite sum on any element in $\dot U_q^{\infty}(\mathfrak{gl}(2))$. Thus one can think of the `additional relation' on $\operatorname{Rep}(\mathfrak{gl}(m|1))$ as
\[ \sum_{i=1}^m e^{k,l}_i = 1_{k,l} \]
for all $k,l$. That is, both sides act identically on $\operatorname{Rep}(\mathfrak{gl}(m|1))$.

In the case of $U_q(\mathfrak{gl}(1|1))$, the idempotent takes on the form
\[ \sum_{k,l\in \N}\left( \frac{1}{\chuse{k+l-2}{l-1}}\begin{tikzpicture}[baseline=-0.65ex]
\draw (-0.5,-1) -- (-0.5,1);
\draw (0.5,-1) -- (0.5,1);
\draw (-0.5,-0.2) -- (0.5,-0.6);
\draw (0.5,0.6) -- (-0.5,0.2);
\draw (-0.5,-1.25) node {$k$};
\draw (0.5,-1.25) node {$l$};
\draw (-0.5,1.25) node {$k$};
\draw (0.5,1.25) node {$l$};
\draw (0,-0.7) node {$l-1$};
\draw (0,0.7) node {$l-1$};
\draw (-0.75,0) node {};
\draw (0.75,0) node {};
\end{tikzpicture}- \frac{[l][k+l-1]}{[k]\chuse{k+l}{k}} \begin{tikzpicture}[baseline=-0.65ex]
\draw (-0.5,-1) -- (-0.5,1);
\draw (0.5,-1) -- (0.5,-0.6);
\draw (0.5,0.6) -- (0.5,1);
\draw (-0.5,-0.2) -- (0.5,-0.6);
\draw (0.5,0.6) -- (-0.5,0.2);
\draw (-0.5,-1.25) node {$k$};
\draw (0.5,-1.25) node {$l$};
\draw (-0.5,1.25) node {$k$};
\draw (0.5,1.25) node {$l$};
\draw (0,-0.7) node {};
\draw (0,0.7) node {};
\draw (-0.75,0) node {};
\draw (0.75,0) node {};
\end{tikzpicture}\right) + \left( \frac{1}{\chuse{k+l}{k}}\begin{tikzpicture}[baseline=-0.65ex]
\draw (-0.5,-1) -- (-0.5,1);
\draw (0.5,-1) -- (0.5,-0.6);
\draw (0.5,0.6) -- (0.5,1);
\draw (-0.5,-0.2) -- (0.5,-0.6);
\draw (0.5,0.6) -- (-0.5,0.2);
\draw (-0.5,-1.25) node {$k$};
\draw (0.5,-1.25) node {$l$};
\draw (-0.5,1.25) node {$k$};
\draw (0.5,1.25) node {$l$};
\draw (0,-0.7) node {};
\draw (0,0.7) node {};
\draw (-0.75,0) node {};
\draw (0.75,0) node {};
\end{tikzpicture}\right) \]
\[ = \sum_{k,l\in \N}\left( \frac{1}{\chuse{k+l-2}{l-1}}\begin{tikzpicture}[baseline=-0.65ex]
\draw (-0.5,-1) -- (-0.5,1);
\draw (0.5,-1) -- (0.5,1);
\draw (-0.5,-0.2) -- (0.5,-0.6);
\draw (0.5,0.6) -- (-0.5,0.2);
\draw (-0.5,-1.25) node {$k$};
\draw (0.5,-1.25) node {$l$};
\draw (-0.5,1.25) node {$k$};
\draw (0.5,1.25) node {$l$};
\draw (0,-0.7) node {$l-1$};
\draw (0,0.7) node {$l-1$};
\draw (-0.75,0) node {};
\draw (0.75,0) node {};
\end{tikzpicture}- \frac{[l-1]}{\chuse{k+l-1}{k-1}} \begin{tikzpicture}[baseline=-0.65ex]
\draw (-0.5,-1) -- (-0.5,1);
\draw (0.5,-1) -- (0.5,-0.6);
\draw (0.5,0.6) -- (0.5,1);
\draw (-0.5,-0.2) -- (0.5,-0.6);
\draw (0.5,0.6) -- (-0.5,0.2);
\draw (-0.5,-1.25) node {$k$};
\draw (0.5,-1.25) node {$l$};
\draw (-0.5,1.25) node {$k$};
\draw (0.5,1.25) node {$l$};
\draw (0,-0.7) node {};
\draw (0,0.7) node {};
\draw (-0.75,0) node {};
\draw (0.75,0) node {};
\end{tikzpicture}\right)
\]
This relation is equivalent to the one derived in \cite[Equation 4.3]{Grant2014}, which can be seen by repeatedly applying the relations to the middle weights $(k+1,l-1)$ and $(k+2,l-2)$ until the diagrams factor through $(k+l-1,1)$ and $(k+l,0)$.

For other $\dot U_q(\mathfrak{gl}(p))$, these idempotents seem very difficult to compute explicitly, but by Theorem \ref{thm:localrelation} we at least know that the relation on $\operatorname{Rep}(\mathfrak{gl}(m|1))$ is generated locally by the above relation.

\subsection{Direct Limit of \texorpdfstring{$\dot U_q(\mathfrak{gl}(p))$}{Uq(gl(p))}}\label{sec:glinfinity}
In Theorem \ref{thm:descriptionOfRep}, we use a direct sum of $\dot U_q(\mathfrak{gl}(p))$ to describe an equivalence of categories. However, note that all we really needed was to ensure large enough tensor products of exterior powers were reached by the functor. There is a lot of duplication in the functor, since the object $\bigwedge^2(\C^{m|n}_q)$ of $\operatorname{Rep}(\mathfrak{gl}(m|n))$ is reached by $1_{(2,0)}$, $1_{(2,0,0)}$, and so on. This is fine as far as an equivalence goes, but a slightly neater idea is afforded by the following:

\begin{definition}
Using the inclusion
\[ \iota_0:\dot U_q(\mathfrak{gl}(p)) \to \dot U_q(\mathfrak{gl}(p+1)) \] from Section \ref{sec:branchingrules}, we define $\dot U_q(\mathfrak{gl}(\infty))$ as the direct limit of the system
\[ \dot U_q(\mathfrak{gl}(\infty)) = \lim_{\rightarrow} \left(\xymatrix{\cdots \ar[r] & \dot U_q(\mathfrak{gl}(p)) \ar[r]& \dot U_q(\mathfrak{gl}(p+1)) \ar[r] &\cdots} \right). \]
\end{definition}

Objects in this category are elements $1_{\lambda}$ where $\lambda$ is a sequence of integers with $\lambda_i=0$ for all but finitely many $i$.

It is easy to see (cf. \cite[Theorem 26.3.1]{Lusztig1993}) that each inclusion carries the canonical basis into the canonical basis, and therefore $\dot U_q(\mathfrak{gl}(\infty))$ inherits a canonical basis from the canonical basis of $\dot U_q(\mathfrak{gl}(p))$ for each $p$.

As before the inclusion takes a weight with a negative entry to a weight with a negative entry, so it also descends to a map $\dot U_q^\infty (\mathfrak{gl}(p)) \to \dot U_q^\infty (\mathfrak{gl}(p+1))$ and by Theorem \ref{thm:inclusion} the inclusion descends to a map
\[ \dot U_q^{(m|n)}(\mathfrak{gl}(p)) \to \dot U_q^{(m|n)}(\mathfrak{gl}(p+1)).\]

\begin{definition}
We define $\dot U_q^\infty(\mathfrak{gl}(\infty))$ and $\dot U_q^{(m|n)}(\mathfrak{gl}(\infty))$ to be the direct limits of these inclusions.
\end{definition}

\begin{thm}\label{thm:descriptionOfRep2}
There is an equivalence of categories
\[ \dot U_q^{(m|n)}(\mathfrak{gl}(\infty))\to \operatorname{Rep}(\mathfrak{gl}(m|n)). \]
\end{thm}

We can define a highest-weight module $V_\infty (\lambda)$ over $\dot U_q(\mathfrak{gl}(\infty))$ in the usual way: consider a vector $v_\lambda$ with weight $\lambda$, where $\lambda$ is a sequence with $\lambda_1\geq \lambda_2\geq \cdots$ with $\lambda_i=0$ for all but finitely many $i$. Now define $V_\infty(\lambda)$ to be generated by $F_i^{(k)}v_\lambda$, subject to $E_iv_\lambda=0$ for all $i\in \N$, and $F_i^{(\lambda_i-\lambda_{i+1}+1)}v_\lambda=0$. This module will be infinite dimensional in general.

These modules behave very much like their finite-dimensional counterparts, in the sense that one can define an appropriate notion of the BGG category $\mathcal{O}$ for $\mathfrak{gl}(\infty)$ and the modules $V_\infty(\lambda)$ classify all irreducible modules in $\mathcal{O}$, as proved by Du and Fu \cite{Du2009}.

Then the canonical inclusion $\dot U_q(\mathfrak{gl}(p))\to \dot U_q(\mathfrak{gl}(\infty))$ induces a restriction functor $\operatorname{Res}_p$ giving $V_\infty(\lambda)$ the structure of a $\dot U_q(\mathfrak{gl}(p))$-module. There is then a canonical inclusion of $\dot U_q(\mathfrak{gl}(p))$-modules $V_p(\lambda|_p)\to \operatorname{Res}_p(V_\infty(\lambda))$, where $\lambda|_p$ denotes the first $p$ terms of $\lambda$. Hence we have the following:

\begin{lem}The algebra $\dot U_q^{(m|n)}(\mathfrak{gl}(\infty))$ acts as $0$ on $V_\infty(\mu)$ where $\mu_{n+1}>m$.
\end{lem}

\begin{lem}\label{lem:decomp}
There is an isomorphism of algebras
\[ \dot{U}_q^{(m|n)}(\mathfrak{gl}(\infty))\cong \sum_{\mu\in H} \operatorname{End}_{fr}(V_\infty(\mu)) \]
where $\operatorname{End}_{fr}(V_\infty(\mu))$ is the non-unital algebra of endomorphisms over $\C(q)$ of finite rank, the sum is the coproduct of algebras, and $H$ is the set of all dominant weights with $\mu_{n+1}\leq m$.
\end{lem}

It seems the equivalence in Theorem \ref{thm:descriptionOfRep2} is the most natural way to think of the action of skew Howe duality, particularly from the point of view of categorification, as we shall see.

\subsection{Braiding}\label{sec:braiding}
As in the previous work of Cautis, Kamnitzer and Morrison \cite{Cautis2012} and the author \cite{Grant2014}, the functor $\dot U_q^{\infty}(\mathfrak{gl}(\infty))\to \operatorname{Rep}(\mathfrak{gl}(m|n))$ takes a braiding on $\dot U_q^{(m|n)}(\mathfrak{gl}(\infty))$ to a braiding on $\operatorname{Rep}(\mathfrak{gl}(m|n))$.

A braided monoidal category is a monoidal category equipped with a natural isomorphism from the bifunctor $-\otimes -$ to the bifunctor $-\otimes^{op}-$, satisfying the two equations
\[ \beta_{U\otimes V,W}=(\beta_{U,W}\otimes \id_V)\circ (\id_U\otimes \beta_{V,W}) \]
\[ \beta_{U,V\otimes W}= (\id_V\otimes \beta_{U,W})\circ(\beta_{U,V}\otimes \id_W) \]
for any objects $U,V,W$. These equations are called the \emph{hexagon equations}.

As mentioned in Section \ref{sec:definitions}, the category $\operatorname{Rep}(\mathfrak{gl}(m|n))$ is braided by the $R$-matrix.

There is also an action of the infinite braid group on $\dot U_q^{\infty}(\mathfrak{gl}(\infty))$ defined by
\[ 1_{s_i(\lambda)}T_i1_{\lambda}= q^{(m-n)\lambda_i\lambda_{i+1}-\lambda_i}\sum_{s=0}^\infty(-q)^{s} F_i^{(\lambda_{i}-\lambda_{i+1}-s)}E_i^{(s)}1_{\lambda}, \quad \lambda_i-\lambda_{i+1}\geq 0 \]
\[ 1_{s_i(\lambda)}T_i1_{\lambda}= q^{(m-n)\lambda_i\lambda_{i+1}-\lambda_i}\sum_{s=0}^\infty(-q)^{s} E_i^{(\lambda_{i+1}-\lambda_{i}+s)}F_i^{(s)}1_{\lambda}, \quad \lambda_i-\lambda_{i+1}\leq 0 \]
where $\lambda=(\lambda_1,\lambda_2,\ldots)$ and $s_i(\lambda)$ is $\lambda$ with the $i$th and $(i+1)$th entries swapped. These sums are finite due to the nilpotence of the $E_i$ and $F_i$ in $\dot U_q^\infty(\mathfrak{gl}(\infty))$.

Thus the image of $T_i1_\lambda$ under $\dot U_q^{\infty}(\mathfrak{gl}(\infty))\to \operatorname{Rep}(\mathfrak{gl}(m|n))$ can be taken as the definition of the braiding on $\operatorname{Rep}(\mathfrak{gl}(m|n))$.

\section{Categorified Quantum \texorpdfstring{$\mathfrak{gl}(p)$}{gl(p)}}\label{sec:catquantumgroup}
We give a categorification of $\dot U_q(\mathfrak{gl}(p))$, by an easy modification of Khovanov and Lauda's categorification \cite{Khovanov2010a} of $\dot U_q(\mathfrak{sl}(p))$. Rouquier \cite{Rouquier2008} gave a categorification of $\dot U_q(\mathfrak{sl}(p))$ in a slightly different way, but this was shown to be equivalent to the Khovanov and Lauda categorification by Brundan \cite{Brundan2015}.

\begin{definition} Let $\kk$ be a commutative unital ring. Given a choice of scalars $Q=\{t_{ij}\in \kk^\times\}$ where $t_{ij}\in \kk^\times$ such that $t_{ij}=t_{ji}$ when $j\neq i\pm 1$ and $t_{ii}=1$ for all $i$, the 2-category ${\mathscr{U}}_Q(\mathfrak{gl}(p))$ is defined with
\begin{itemize}
\item Objects: $1_{\lambda}$ for each $\lambda\in \Z^m$.
\item 1-morphisms: formal direct sums of $q^kE_i1_\lambda$ and $q^kF_i1_\lambda$ for $1\leq i\leq p-1$, with $k\in \Z$.
\item 2-morphisms: $\kk$-linear combinations of compositions of the following,

\begin{tabular}{cc}
\begin{tikzpicture}[baseline=-0.65ex]
\node (middle)[label=left:$\lambda+\alpha_i$,label=right:$\lambda$] {$\bullet$};
\node [below of=middle] (bottom) {};
\node [above of=middle] (top) {$i$};
\draw [->] (bottom) -- (top);
\end{tikzpicture}:$E_i1_\lambda\to q^2E_i1_{\lambda}$ & \begin{tikzpicture}[baseline=-0.65ex]
\node (middle)[label=left:$\lambda-\alpha_i$,label=right:$\lambda$] {$\bullet$};
\node [below of=middle] (bottom) {};
\node [above of=middle] (top) {$i$};
\draw [<-] (bottom) -- (top);
\end{tikzpicture}:$F_i1_\lambda\to q^2F_i1_{\lambda}$\\
\begin{tikzpicture}[baseline=-0.65ex]
\node (middle) {};
\node [above of=middle] (tl) {$i$};
\node [right of=tl] (tr) {$j$};
\node [below of=middle] (bl) {$j$};
\node [right of=bl] (br) {$i$};
\node [right of=middle, label=right:$\lambda$] (r) {};
\draw [->] (bl) to [out=90,in=270] (tr);
\draw [->] (br) to [out=90,in=270] (tl);
\end{tikzpicture}:$E_jE_i1_\lambda \to q^{-\alpha_i\cdot \alpha_j}E_iE_j1_\lambda$ &\begin{tikzpicture}[baseline=-0.65ex]
\node (middle) {};
\node [above of=middle] (tl) {$i$};
\node [right of=tl] (tr) {$j$};
\node [below of=middle] (bl) {$j$};
\node [right of=bl] (br) {$i$};
\node [right of=middle, label=right:$\lambda$] (r) {};
\draw [<-] (bl) to [out=90,in=270] (tr);
\draw [<-] (br) to [out=90,in=270] (tl);
\end{tikzpicture}:$F_jF_i1_\lambda \to q^{-\alpha_i\cdot \alpha_j}F_iF_j1_\lambda$ \\
\begin{tikzpicture}[baseline=-0.5cm]
\node (mid) {$i$};
\node [right of=mid] (r) {$i$};
\draw [<-] (mid) to [bend right=90] (r);
\end{tikzpicture}:$1_\lambda \to q^{1+\lambda_i-\lambda_{i+1}}E_iF_i1_\lambda$ & \begin{tikzpicture}[baseline=-0.5cm]
\node (mid) {$i$};
\node [right of=mid] (r) {$i$};
\draw [->] (mid) to [bend right=90] (r);
\end{tikzpicture}:$1_\lambda \to q^{1-\lambda_i+\lambda_{i+1}}F_iE_i1_\lambda$ \\
\begin{tikzpicture}[baseline=+0.25cm]
\node (mid) {$i$};
\node [right of=mid] (r) {$i$};
\draw [->] (mid) to [bend left=90] (r);
\end{tikzpicture}:$E_iF_i1_\lambda \to q^{1+\lambda_i-\lambda_{i+1}}1_\lambda$ & \begin{tikzpicture}[baseline=+0.25cm]
\node (mid) {$i$};
\node [right of=mid] (r) {$i$};
\draw [<-] (mid) to [bend left=90] (r);
\end{tikzpicture}:$F_iE_i1_\lambda \to q^{1-\lambda_i+\lambda_{i+1}}1_\lambda$
\end{tabular}

satisfying the relations as below.
\end{itemize}
\end{definition}

\subsubsection{Biadjointness of $E_i1_{\lambda}$ and $F_i1_\lambda$}
\[
\begin{tikzpicture}[baseline=-0.65ex,node distance=0.75cm]
\coordinate (mid) {};
\node [below of=mid] (bl) {$i$};
\node [above of=mid] (tl)  {$\lambda+\alpha_i$};
\coordinate[right of=bl] (bm) {};
\node [right of=bm] (br) {$\lambda$};
\coordinate [right of=tl] (tm) {};
\node [right of=tm] (tr) {$i$};
\coordinate [right of=mid] (mm) {};
\coordinate [right of=mm] (mr) {};
\draw [->] (bl) -- (mid) to [bend left=90] (mm) to [bend right=90] (mr) -- (tr);
\end{tikzpicture}=\begin{tikzpicture}[baseline=-0.65ex,node distance=0.75cm]
\coordinate (mid) {};
\node [below of=mid] (bl) {$\lambda+\alpha_i$};
\node [above of=mid] (tl)  {$i$};
\coordinate[right of=bl] (bm) {};
\node [right of=bm] (br) {$i$};
\coordinate [right of=tl] (tm) {};
\node [right of=tm] (tr) {$\lambda$};
\coordinate [right of=mid] (mm) {};
\coordinate [right of=mm] (mr) {};
\draw [->] (br) -- (mr) to [bend right=90] (mm) to [bend left=90] (mid) -- (tl);
\end{tikzpicture}= \begin{tikzpicture}[baseline=-0.65ex,node distance=0.75cm]
\node (middle)[label=left:$\lambda+\alpha_i$,label=right:$\lambda$] {};
\node [below of=middle] (bottom) {};
\node [above of=middle] (top) {$i$};
\draw [->] (bottom) -- (top);
\end{tikzpicture}, \quad\begin{tikzpicture}[baseline=-0.65ex,node distance=0.75cm]
\coordinate (mid) {};
\node [below of=mid] (bl) {$i$};
\node [above of=mid] (tl)  {$\lambda-\alpha_i$};
\coordinate[right of=bl] (bm) {};
\node [right of=bm] (br) {$\lambda$};
\coordinate [right of=tl] (tm) {};
\node [right of=tm] (tr) {$i$};
\coordinate [right of=mid] (mm) {};
\coordinate [right of=mm] (mr) {};
\draw [<-] (bl) -- (mid) to [bend left=90] (mm) to [bend right=90] (mr) -- (tr);
\end{tikzpicture}=\begin{tikzpicture}[baseline=-0.65ex,node distance=0.75cm]
\coordinate (mid) {};
\node [below of=mid] (bl) {$\lambda-\alpha_i$};
\node [above of=mid] (tl)  {$i$};
\coordinate[right of=bl] (bm) {};
\node [right of=bm] (br) {$i$};
\coordinate [right of=tl] (tm) {};
\node [right of=tm] (tr) {$\lambda$};
\coordinate [right of=mid] (mm) {};
\coordinate [right of=mm] (mr) {};
\draw [<-] (br) -- (mr) to [bend right=90] (mm) to [bend left=90] (mid) -- (tl);
\end{tikzpicture}= \begin{tikzpicture}[baseline=-0.65ex,node distance=0.75cm]
\node (middle)[label=left:$\lambda-\alpha_i$,label=right:$\lambda$] {};
\node [below of=middle] (bottom) {};
\node [above of=middle] (top) {$i$};
\draw [<-] (bottom) -- (top);
\end{tikzpicture} \]

\subsubsection{Cyclicity of 2-morphisms with respect to the biadjoint structure}
\[\begin{tikzpicture}[baseline=-0.65ex,node distance=0.75cm]
\coordinate (mid) {};
\node [below of=mid] (bl) {$i$};
\node [above of=mid] (tl)  {$\lambda+\alpha_i$};
\coordinate[right of=bl] (bm) {};
\node [right of=bm] (br) {$\lambda$};
\coordinate [right of=tl] (tm) {};
\node [right of=tm] (tr) {$i$};
\coordinate [right of=mid] (mm) {};
\coordinate [right of=mm] (mr) {};
\draw [->] (bl) -- (mid) to [bend left=90] (mm) to [bend right=90] (mr) -- (tr);
\filldraw (mm) circle (2pt);
\end{tikzpicture}=\begin{tikzpicture}[baseline=-0.65ex,node distance=0.75cm]
\coordinate (mid) {};
\node [below of=mid] (bl) {$\lambda+\alpha_i$};
\node [above of=mid] (tl)  {$i$};
\coordinate[right of=bl] (bm) {};
\node [right of=bm] (br) {$i$};
\coordinate [right of=tl] (tm) {};
\node [right of=tm] (tr) {$\lambda$};
\coordinate [right of=mid] (mm) {};
\coordinate [right of=mm] (mr) {};
\draw [->] (br) -- (mr) to [bend right=90] (mm) to [bend left=90] (mid) -- (tl);
\filldraw (mm) circle (2pt);
\end{tikzpicture}= \begin{tikzpicture}[baseline=-0.65ex,node distance=0.75cm]
\node (middle)[label=left:$\lambda+\alpha_i$,label=right:$\lambda$] {$\bullet$};
\node [below of=middle] (bottom) {};
\node [above of=middle] (top) {$i$};
\draw [->] (bottom) -- (top);
\end{tikzpicture}, \quad \begin{tikzpicture}[baseline=-0.65ex,node distance=0.75cm]
\coordinate (mid) {};
\node [below of=mid] (bl) {$i$};
\node [above of=mid] (tl)  {$\lambda-\alpha_i$};
\coordinate[right of=bl] (bm) {};
\node [right of=bm] (br) {$\lambda$};
\coordinate [right of=tl] (tm) {};
\node [right of=tm] (tr) {$i$};
\coordinate [right of=mid] (mm) {};
\coordinate [right of=mm] (mr) {};
\draw [<-] (bl) -- (mid) to [bend left=90] (mm) to [bend right=90] (mr) -- (tr);
\filldraw (mm) circle (2pt);
\end{tikzpicture}=\begin{tikzpicture}[baseline=-0.65ex,node distance=0.75cm]
\coordinate (mid) {};
\node [below of=mid] (bl) {$\lambda-\alpha_i$};
\node [above of=mid] (tl)  {$i$};
\coordinate[right of=bl] (bm) {};
\node [right of=bm] (br) {$i$};
\coordinate [right of=tl] (tm) {};
\node [right of=tm] (tr) {$\lambda$};
\coordinate [right of=mid] (mm) {};
\coordinate [right of=mm] (mr) {};
\draw [<-] (br) -- (mr) to [bend right=90] (mm) to [bend left=90] (mid) -- (tl);
\filldraw (mm) circle (2pt);
\end{tikzpicture}= \begin{tikzpicture}[baseline=-0.65ex,node distance=0.75cm]
\node (middle)[label=left:$\lambda-\alpha_i$,label=right:$\lambda$] {$\bullet$};
\node [below of=middle] (bottom) {};
\node [above of=middle] (top) {$i$};
\draw [<-] (bottom) -- (top);
\end{tikzpicture}\]

\[ \begin{tikzpicture}[baseline=-0.75cm,node distance=0.5cm]
\node (ttrrr) {$j$};
\coordinate [below of=ttrrr] (trrr);
\coordinate [below of=trrr] (rrr);
\node [left of=ttrrr] (ttrr) {$i$};
\coordinate [below of=ttrr,label={[label distance=0.2cm]left:$\lambda$}] (trr);
\coordinate [below of=trr] (rr);
\coordinate [left of=rr] (r);
\coordinate [left of=trr] (tr);
\coordinate [left of=r] (l);
\coordinate [left of=tr] (tl);
\coordinate [left of =tl] (tll);
\coordinate [left of=tll] (tlll);
\coordinate [below of=tll] (ll);
\node [below of=ll] (bll) {$i$};
\coordinate [below of=tlll] (lll);
\node [below of=lll] (blll) {$j$};
\draw [->] (blll) -- (tlll) to [out=90,in=90] (tr) to [out=270,in=90] (l) to [out=270,in=270] (rrr) -- (ttrrr);
\draw [->] (bll) -- (tll) to [out=90,in=90] (tl) to [out=270,in=90] (r) to [out=270,in=270] (rr) -- (ttrr);
\end{tikzpicture}
=
\begin{tikzpicture}[baseline=-0.5cm,node distance=0.5cm]
\node (tl) {$i$};
\node [right of=tl] (tr) {$j$};
\node [below of=tr] (mr) {};
\node [below of=tl] (ml) {};
\node [below of=mr] (br) {$i$};
\node [left of=br] (bl) {$j$};
\draw[<-] (tl) to [out=270,in=90] node [name=middle]{} (br);
\draw[<-] (tr) to [out=270,in=90] (bl);
\node [left of=middle] {$\lambda$};
\node [right of=middle] {};
\end{tikzpicture}
=
\begin{tikzpicture}[baseline=-0.75cm,node distance=0.5cm]
\node (ttlll) {$i$};
\coordinate [below of=ttlll] (tlll);
\coordinate [below of=tlll] (lll);
\node [right of=ttlll] (ttll) {$j$};
\coordinate [below of=ttll] (tll);
\coordinate [below of=tll] (ll);
\coordinate [right of=ll] (l);
\coordinate [right of=tll] (tl);
\coordinate [right of=l] (r);
\coordinate [right of=tl] (tr);
\coordinate [right of =tr] (trr);
\coordinate [right of=trr] (trrr);
\coordinate [below of=trr,label={[label distance=0.2cm]left:$\lambda$}] (rr);
\node [below of=rr] (brr) {$j$};
\coordinate [below of=trrr] (rrr);
\node [below of=rrr] (brrr) {$i$};
\draw [->] (brrr) -- (trrr) to [out=90,in=90] (tl) to [out=270,in=90] (r) to [out=270,in=270] (lll) -- (ttlll);
\draw [->] (brr) -- (trr) to [out=90,in=90] (tr) to [out=270,in=90] (l) to [out=270,in=270] (ll) -- (ttll);
\end{tikzpicture}, \quad \begin{tikzpicture}[baseline=-0.75cm,node distance=0.5cm]
\node (ttrrr) {$j$};
\coordinate [below of=ttrrr] (trrr);
\coordinate [below of=trrr] (rrr);
\node [left of=ttrrr] (ttrr) {$i$};
\coordinate [below of=ttrr,label={[label distance=0.2cm]left:$\lambda$}] (trr);
\coordinate [below of=trr] (rr);
\coordinate [left of=rr] (r);
\coordinate [left of=trr] (tr);
\coordinate [left of=r] (l);
\coordinate [left of=tr] (tl);
\coordinate [left of =tl] (tll);
\coordinate [left of=tll] (tlll);
\coordinate [below of=tll] (ll);
\node [below of=ll] (bll) {$i$};
\coordinate [below of=tlll] (lll);
\node [below of=lll] (blll) {$j$};
\draw [<-] (blll) -- (tlll) to [out=90,in=90] (tr) to [out=270,in=90] (l) to [out=270,in=270] (rrr) -- (ttrrr);
\draw [<-] (bll) -- (tll) to [out=90,in=90] (tl) to [out=270,in=90] (r) to [out=270,in=270] (rr) -- (ttrr);
\end{tikzpicture}
=
\begin{tikzpicture}[baseline=-0.5cm,node distance=0.5cm]
\node (tl) {$i$};
\node [right of=tl] (tr) {$j$};
\node [below of=tr] (mr) {};
\node [below of=tl] (ml) {};
\node [below of=mr] (br) {$i$};
\node [left of=br] (bl) {$j$};
\draw[->] (tl) to [out=270,in=90] node [name=middle]{} (br);
\draw[->] (tr) to [out=270,in=90] (bl);
\node [left of=middle] {$\lambda$};
\node [right of=middle] {};
\end{tikzpicture}
=
\begin{tikzpicture}[baseline=-0.75cm,node distance=0.5cm]
\node (ttlll) {$i$};
\coordinate [below of=ttlll] (tlll);
\coordinate [below of=tlll] (lll);
\node [right of=ttlll] (ttll) {$j$};
\coordinate [below of=ttll] (tll);
\coordinate [below of=tll] (ll);
\coordinate [right of=ll] (l);
\coordinate [right of=tll] (tl);
\coordinate [right of=l] (r);
\coordinate [right of=tl] (tr);
\coordinate [right of =tr] (trr);
\coordinate [right of=trr] (trrr);
\coordinate [below of=trr,label={[label distance=0.2cm]left:$\lambda$}] (rr);
\node [below of=rr] (brr) {$j$};
\coordinate [below of=trrr] (rrr);
\node [below of=rrr] (brrr) {$i$};
\draw [<-] (brrr) -- (trrr) to [out=90,in=90] (tl) to [out=270,in=90] (r) to [out=270,in=270] (lll) -- (ttlll);
\draw [<-] (brr) -- (trr) to [out=90,in=90] (tr) to [out=270,in=90] (l) to [out=270,in=270] (ll) -- (ttll);
\end{tikzpicture}\]

\[
\begin{tikzpicture}[baseline=-0.65ex,node distance=0.5cm]
\coordinate (mid) {};
\coordinate [above left of=mid] (tl);
\coordinate [above right of=mid] (tr);
\coordinate [below left of=mid] (bl);
\coordinate [below right of=mid,label=below left:$\lambda$] (br);
\coordinate [left of=tl] (tll);
\coordinate [left of=bl] (bll);
\node [below of=bll] (bbll) {$j$};
\node [above of=tr] (ttr) {$i$};
\node [right of=ttr] (ttrr) {$j$};
\coordinate [right of=br] (brr);
\node [right of=bbll] (bbl) {$i$};
\draw [<-] (bbll) -- (tll) to [bend left=90] (tl) to [out=270,in=135] (mid) to [out=-45,in=90] (br) to [bend right=90] (brr) -- (ttrr);
\draw [->] (bbl) -- (bl) to [out=90,in=225] (mid) to [out=45,in=270] (tr) -- (ttr);
\end{tikzpicture}= t_{ji} \begin{tikzpicture}[baseline=-0.65ex,node distance=0.5cm]
\coordinate (mid) {};
\coordinate [above left of=mid] (tl);
\coordinate [above right of=mid,label=above left:$\lambda$] (tr);
\coordinate [below left of=mid] (bl);
\coordinate [below right of=mid] (br);
\coordinate [left of=tl] (tll);
\coordinate [left of=bl] (bll);
\coordinate [below of=bll] (bbll);
\coordinate [above of=tr] (ttr);
\coordinate [right of=tr] (trr);
\coordinate [right of=ttr] (ttrr);
\coordinate [right of=br] (brr);
\coordinate [right of=bbll] (bbl);
\node [below of=br] (bbr) {$j$};
\node [right of=bbr] (bbrr) {$i$};
\node [above of=tl] (ttl) {$j$};
\node [left of=ttl] (ttll) {$i$};
\draw [->] (bbrr) -- (trr) to [bend right=90] (tr) to [out=270,in=45] (mid) to [out=225,in=90] (bl) to [bend left=90] (bll) -- (ttll);
\draw [<-] (bbr) -- (br) to [out=90,in=-45] (mid) to [out=135,in=270] (tl) -- (ttl);
\end{tikzpicture} 
=
\begin{tikzpicture}[baseline=-0.65ex,node distance=0.5cm]
\node (middle) {};
\node [above of=middle] (tl) {$i$};
\node [right of=tl] (tr) {$j$};
\node [below of=middle] (bl) {$j$};
\node [right of=bl] (br) {$i$};
\node [right of=middle, label=right:$\lambda$] (r) {};
\draw [<-] (bl) to [out=90,in=270] (tr);
\draw [->] (br) to [out=90,in=270] (tl);
\end{tikzpicture},\quad  \begin{tikzpicture}[baseline=-0.65ex,node distance=0.5cm]
\coordinate (mid) {};
\coordinate [above left of=mid] (tl);
\coordinate [above right of=mid,label=above left:$\lambda$] (tr);
\coordinate [below left of=mid] (bl);
\coordinate [below right of=mid] (br);
\coordinate [left of=tl] (tll);
\coordinate [left of=bl] (bll);
\coordinate [below of=bll] (bbll);
\coordinate [above of=tr] (ttr);
\coordinate [right of=tr] (trr);
\coordinate [right of=ttr] (ttrr);
\coordinate [right of=br] (brr);
\coordinate [right of=bbll] (bbl);
\node [below of=br] (bbr) {$j$};
\node [right of=bbr] (bbrr) {$i$};
\node [above of=tl] (ttl) {$j$};
\node [left of=ttl] (ttll) {$i$};
\draw [<-] (bbrr) -- (trr) to [bend right=90] (tr) to [out=270,in=45] (mid) to [out=225,in=90] (bl) to [bend left=90] (bll) -- (ttll);
\draw [->] (bbr) -- (br) to [out=90,in=-45] (mid) to [out=135,in=270] (tl) -- (ttl);
\end{tikzpicture}=t_{ij}
\begin{tikzpicture}[baseline=-0.65ex,node distance=0.5cm]
\coordinate (mid) {};
\coordinate [above left of=mid] (tl);
\coordinate [above right of=mid] (tr);
\coordinate [below left of=mid] (bl);
\coordinate [below right of=mid,label=below left:$\lambda$] (br);
\coordinate [left of=tl] (tll);
\coordinate [left of=bl] (bll);
\node [below of=bll] (bbll) {$j$};
\node [above of=tr] (ttr) {$i$};
\node [right of=ttr] (ttrr) {$j$};
\coordinate [right of=br] (brr);
\node [right of=bbll] (bbl) {$i$};
\draw [->] (bbll) -- (tll) to [bend left=90] (tl) to [out=270,in=135] (mid) to [out=-45,in=90] (br) to [bend right=90] (brr) -- (ttrr);
\draw [<-] (bbl) -- (bl) to [out=90,in=225] (mid) to [out=45,in=270] (tr) -- (ttr);
\end{tikzpicture} = \begin{tikzpicture}[baseline=-0.65ex,node distance=0.5cm]
\node (middle) {};
\node [above of=middle] (tl) {$i$};
\node [right of=tl] (tr) {$j$};
\node [below of=middle] (bl) {$j$};
\node [right of=bl] (br) {$i$};
\node [right of=middle, label=right:$\lambda$] (r) {};
\draw [->] (bl) to [out=90,in=270] (tr);
\draw [<-] (br) to [out=90,in=270] (tl);
\end{tikzpicture}\]

\subsubsection{Action of KLR algebra on $E_i1_\lambda$}
\[ \begin{tikzpicture}[baseline=-0.65ex]
\node (tl) {};
\node [above of=tl] (ttl) {$i$};
\node [right of=ttl] (ttr) {$j$};
\node [below of=tl] (ml) {$i$};
\node [below of=ttr,label=right:$\lambda$] (tr) {};
\node [below of=tr] (mr) {$j$};
\draw [<-] (ttl) to [out=270,in=90] (tr) to [out=270,in=90] (ml);
\draw [<-] (ttr) to [out=270,in=90] (tl) to [out=270,in=90] (mr);
\end{tikzpicture}
=
\left\lbrace
\begin{array}{lr}
0 & \mathrm{if }\, i=j \\
t_{ij}\begin{tikzpicture}[baseline=-0.65ex]
\node (ml) {};
\node [above of=ml] (tl) {$i$};
\node [right of=tl] (tr) {$j$};
\node [below of=ml] (bl) {};
\node [below of=tr,label=right:$\lambda$] (mr) {};
\node [below of=mr] (br) {};
\draw [<-] (tr) -- (br);
\draw [<-] (tl) -- (bl);
\end{tikzpicture} & \mathrm{if }\, |i-j|>1 \\
t_{ij}\begin{tikzpicture}[baseline=-0.65ex]
\node (ml) {$\bullet$};
\node [above of=ml] (tl) {$i$};
\node [right of=tl] (tr) {$j$};
\node [below of=ml] (bl) {};
\node [below of=tr,label=right:$\lambda$] (mr) {};
\node [below of=mr] (br) {};
\draw [<-] (tr) -- (br);
\draw [<-] (tl) -- (bl);
\end{tikzpicture} + t_{ji}\begin{tikzpicture}[baseline=-0.65ex]
\node (ml) {};
\node [above of=ml] (tl) {$i$};
\node [right of=tl] (tr) {$j$};
\node [below of=ml] (bl) {};
\node [below of=tr,label=right:$\lambda$] (mr) {$\bullet$};
\node [below of=mr] (br) {};
\draw [<-] (tr) -- (br);
\draw [<-] (tl) -- (bl);
\end{tikzpicture}
& \mathrm{if }\, |i-j|=1
\end{array}
\right.
\]

\[
\begin{tikzpicture}[baseline=-0.65ex,node distance=0.5cm]
\node (ml) {};
\node [above of=ml] (tl) {$j$};
\node [right of=tl] (tr) {$i$};
\node [below of=tr,label=right:$\lambda$] (mr) {};
\node [below of=mr] (br) {$j$};
\node [left of=br] (bl) {$i$};
\draw [<-] (tl) to [out=270,in=90] node [name=middle,label=above left:$\bullet$]{} (br);
\draw [<-] (tr) to [out=270,in=90] (bl);
\node [right of=middle] {};
\node [left of=middle] {};
\end{tikzpicture} = \begin{tikzpicture}[baseline=-0.65ex,node distance=0.5cm]
\node (ml) {};
\node [above of=ml] (tl) {$j$};
\node [right of=tl] (tr) {$i$};
\node [below of=tr,label=right:$\lambda$] (mr) {};
\node [below of=mr] (br) {$j$};
\node [left of=br] (bl) {$i$};
\draw [<-] (tl) to [out=270,in=90] node [name=middle,label=below right:$\bullet$]{} (br);
\draw [<-] (tr) to [out=270,in=90] (bl);
\node [right of=middle] {};
\node [left of=middle] {};
\end{tikzpicture}
\quad \quad
\begin{tikzpicture}[baseline=-0.65ex,node distance=0.5cm]
\node (ml) {};
\node [above of=ml] (tl) {$j$};
\node [right of=tl] (tr) {$i$};
\node [below of=tr,label=right:$\lambda$] (mr) {};
\node [below of=mr] (br) {$j$};
\node [left of=br] (bl) {$i$};
\draw [<-] (tl) to [out=270,in=90] node [name=middle,label=below left:$\bullet$]{} (br);
\draw [<-] (tr) to [out=270,in=90] (bl);
\node [right of=middle] {};
\node [left of=middle] {};
\end{tikzpicture} = \begin{tikzpicture}[baseline=-0.65ex,node distance=0.5cm]
\node (ml) {};
\node [above of=ml] (tl) {$j$};
\node [right of=tl] (tr) {$i$};
\node [below of=tr,label=right:$\lambda$] (mr) {};
\node [below of=mr] (br) {$j$};
\node [left of=br] (bl) {$i$};
\draw [<-] (tl) to [out=270,in=90] node [name=middle,label=above right:$\bullet$]{} (br);
\draw [<-] (tr) to [out=270,in=90] (bl);
\node [right of=middle] {};
\node [left of=middle] {};
\end{tikzpicture} \quad \quad \mathrm{i\neq j}
\]

\[ 
\begin{tikzpicture}[baseline=-0.65ex,node distance=0.5cm]
\node (ml) {};
\node [above of=ml] (tl) {$i$};
\node [right of=tl] (tr) {$i$};
\node [below of=tr,label=right:$\lambda$] (mr) {};
\node [below of=mr] (br) {$i$};
\node [left of=br] (bl) {$i$};
\draw [<-] (tl) to [out=270,in=90] node [name=middle,label=below left:$\bullet$]{} (br);
\draw [<-] (tr) to [out=270,in=90] (bl);
\node [right of=middle] {};
\node [left of=middle] {};
\end{tikzpicture} - \begin{tikzpicture}[baseline=-0.65ex,node distance=0.5cm]
\node (ml) {};
\node [above of=ml] (tl) {$i$};
\node [right of=tl] (tr) {$i$};
\node [below of=tr,label=right:$\lambda$] (mr) {};
\node [below of=mr] (br) {$i$};
\node [left of=br] (bl) {$i$};
\draw [<-] (tl) to [out=270,in=90] node [name=middle,label=above right:$\bullet$]{} (br);
\draw [<-] (tr) to [out=270,in=90] (bl);
\node [right of=middle] {};
\node [left of=middle] {};
\end{tikzpicture} =
\begin{tikzpicture}[baseline=-0.65ex,node distance=0.5cm]
\node (ml) {};
\node [above of=ml] (tl) {$i$};
\node [right of=tl] (tr) {$i$};
\node [below of=tr,label=right:$\lambda$] (mr) {};
\node [below of=mr] (br) {$i$};
\node [left of=br] (bl) {$i$};
\draw [<-] (tl) to [out=270,in=90] node [name=middle,label=above left:$\bullet$]{} (br);
\draw [<-] (tr) to [out=270,in=90] (bl);
\node [right of=middle] {};
\node [left of=middle] {};
\end{tikzpicture} - \begin{tikzpicture}[baseline=-0.65ex,node distance=0.5cm]
\node (ml) {};
\node [above of=ml] (tl) {$i$};
\node [right of=tl] (tr) {$i$};
\node [below of=tr,label=right:$\lambda$] (mr) {};
\node [below of=mr] (br) {$i$};
\node [left of=br] (bl) {$i$};
\draw [<-] (tl) to [out=270,in=90] node [name=middle,label=below right:$\bullet$]{} (br);
\draw [<-] (tr) to [out=270,in=90] (bl);
\node [right of=middle] {};
\node [left of=middle] {};
\end{tikzpicture} = \begin{tikzpicture}[baseline=-0.65ex,node distance=0.5cm]
\node (ml) {};
\node [above of=ml] (tl) {$i$};
\node [right of=tl] (tr) {$i$};
\node [below of=ml] (bl) {};
\node [below of=tr,label=right:$\lambda$] (mr) {};
\node [below of=mr] (br) {};
\draw [<-] (tr) -- (br);
\draw [<-] (tl) -- (bl);
\end{tikzpicture}
\]

\[\begin{tikzpicture}[baseline=-0.65ex]
\node (ml) {};
\node [right of=ml] (mm) {};
\node [right of=mm,label=right:$\lambda$] (mr) {};
\node [above of=ml] (tl) {$i$};
\node [above of=mm] (tm) {$j$};
\node [above of=mr] (tr) {$i$};
\node [below of=ml] (bl) {$i$};
\node [below of=mm] (bm) {$j$};
\node [below of=mr] (br) {$i$};
\draw [<-] (tl) to [out=270,in=90] (br);
\draw [<-] (tr) to [out=270,in=90] (bl);
\draw [<-] (tm) to [out=270,in=90] (ml) to [out=270,in=90] (bm);
\end{tikzpicture}
- \begin{tikzpicture}[baseline=-0.65ex]
\node (ml) {};
\node [right of=ml] (mm) {};
\node [right of=mm,label=right:$\lambda$] (mr) {};
\node [above of=ml] (tl) {$i$};
\node [above of=mm] (tm) {$j$};
\node [above of=mr] (tr) {$i$};
\node [below of=ml] (bl) {$i$};
\node [below of=mm] (bm) {$j$};
\node [below of=mr] (br) {$i$};
\draw [<-] (tl) to [out=270,in=90] (br);
\draw [<-] (tr) to [out=270,in=90] (bl);
\draw [<-] (tm) to [out=270,in=90] (mr) to [out=270,in=90] (bm);
\end{tikzpicture} =
t_{ij}\begin{tikzpicture}[baseline=-0.65ex]
\node (ml) {};
\node [right of=ml] (mm) {};
\node [right of=mm,label=right:$\lambda$] (mr) {};
\node [above of=ml] (tl) {$i$};
\node [above of=mm] (tm) {$j$};
\node [above of=mr] (tr) {$i$};
\node [below of=ml] (bl) {$i$};
\node [below of=mm] (bm) {$j$};
\node [below of=mr] (br) {$i$};
\draw [<-] (tl) -- (bl);
\draw [<-] (tm) -- (bm);
\draw [<-] (tr) -- (br);
\end{tikzpicture} \quad \mathrm{if } \, |i-j|=1 
\] 

\[ \begin{tikzpicture}[baseline=-0.65ex]
\node (ml) {};
\node [right of=ml] (mm) {};
\node [right of=mm,label=right:$\lambda$] (mr) {};
\node [above of=ml] (tl) {$i$};
\node [above of=mm] (tm) {$j$};
\node [above of=mr] (tr) {$k$};
\node [below of=ml] (bl) {$k$};
\node [below of=mm] (bm) {$j$};
\node [below of=mr] (br) {$i$};
\draw [<-] (tl) to [out=270,in=90] (br);
\draw [<-] (tr) to [out=270,in=90] (bl);
\draw [<-] (tm) to [out=270,in=90] (ml) to [out=270,in=90] (bm);
\end{tikzpicture}
=  \begin{tikzpicture}[baseline=-0.65ex]
\node (ml) {};
\node [right of=ml] (mm) {};
\node [right of=mm,label=right:$\lambda$] (mr) {};
\node [above of=ml] (tl) {$i$};
\node [above of=mm] (tm) {$j$};
\node [above of=mr] (tr) {$k$};
\node [below of=ml] (bl) {$k$};
\node [below of=mm] (bm) {$j$};
\node [below of=mr] (br) {$i$};
\draw [<-] (tl) to [out=270,in=90] (br);
\draw [<-] (tr) to [out=270,in=90] (bl);
\draw [<-] (tm) to [out=270,in=90] (mr) to [out=270,in=90] (bm);
\end{tikzpicture} \quad \mathrm{if } i\neq k\, \mathrm{or} |i-j|\neq 1 \]

\subsubsection{Mixed relations between $E_iF_j1_\lambda$ and $F_jE_i1_\lambda$ for $i\neq j$}
\[ \begin{tikzpicture}[baseline=-0.65ex]
\node (tl) {};
\node [above of=tl] (ttl) {$i$};
\node [right of=ttl] (ttr) {$j$};
\node [below of=tl] (ml) {$i$};
\node [below of=ttr,label=right:$\lambda$] (tr) {};
\node [below of=tr] (mr) {$j$};
\draw [<-] (ttl) to [out=270,in=90] (tr) to [out=270,in=90] (ml);
\draw [->] (ttr) to [out=270,in=90] (tl) to [out=270,in=90] (mr);
\end{tikzpicture}= t_{ji} \begin{tikzpicture}[baseline=-0.65ex]
\node (ml) {};
\node [above of=ml] (tl) {$i$};
\node [right of=tl] (tr) {$j$};
\node [below of=ml] (bl) {};
\node [below of=tr,label=right:$\lambda$] (mr) {};
\node [below of=mr] (br) {};
\draw [->] (tr) -- (br);
\draw [<-] (tl) -- (bl);
\end{tikzpicture} 
\quad 
\begin{tikzpicture}[baseline=-0.65ex]
\node (tl) {};
\node [above of=tl] (ttl) {$i$};
\node [right of=ttl] (ttr) {$j$};
\node [below of=tl] (ml) {$i$};
\node [below of=ttr,label=right:$\lambda$] (tr) {};
\node [below of=tr] (mr) {$j$};
\draw [->] (ttl) to [out=270,in=90] (tr) to [out=270,in=90] (ml);
\draw [<-] (ttr) to [out=270,in=90] (tl) to [out=270,in=90] (mr);
\end{tikzpicture}= t_{ij} \begin{tikzpicture}[baseline=-0.65ex]
\node (ml) {};
\node [above of=ml] (tl) {$i$};
\node [right of=tl] (tr) {$j$};
\node [below of=ml] (bl) {};
\node [below of=tr,label=right:$\lambda$] (mr) {};
\node [below of=mr] (br) {};
\draw [<-] (tr) -- (br);
\draw [->] (tl) -- (bl);
\end{tikzpicture} \]

\subsubsection{Bubble relations}
\[ \begin{tikzpicture}[node distance=0.5cm,baseline=-0.65ex]
\node (m) {};
\draw[decoration={markings,mark=at position 0.5 with {\arrow{>}}},postaction={decorate}] (m) circle [radius=0.5cm];
\node [below right of=m,label={[label distance=-0.3cm]below right:$\alpha$}] (dot) {$\bullet$};
\node [left of=m,label=left:$i$] (colour) {};
\node [above right of=m, label={[label distance=0.1cm]above right:$\lambda$}] (weight) {};
\end{tikzpicture} = 0 \, \mathrm{if }\, \alpha<-\bar\lambda-1 \quad \begin{tikzpicture}[node distance=0.5cm,baseline=-0.65ex]
\node (m) {};
\draw[decoration={markings,mark=at position 0.5 with {\arrow{<}}},postaction={decorate}] (m) circle [radius=0.5cm];
\node [below right of=m,label={[label distance=-0.3cm]below right:$\alpha$}] (dot) {$\bullet$};
\node [left of=m,label=left:$i$] (colour) {};
\node [above right of=m, label={[label distance=0.1cm]above right:$\lambda$}] (weight) {};
\end{tikzpicture} = 0 \,\mathrm{if }\, \alpha<\bar\lambda-1 \]

\[ \begin{tikzpicture}[node distance=0.5cm,baseline=-0.65ex]
\node (m) {};
\draw[decoration={markings,mark=at position 0.5 with {\arrow{>}}},postaction={decorate}] (m) circle [radius=0.5cm];
\node [below right of=m,label={[label distance=-0.3cm]below right:$-\bar\lambda_i-1$}] (dot) {$\bullet$};
\node [left of=m,label=left:$i$] (colour) {};
\node [above right of=m, label={[label distance=0.1cm]above right:$\lambda$}] (weight) {};
\end{tikzpicture} = \id_{1_{\lambda}}, \quad \begin{tikzpicture}[node distance=0.5cm,baseline=-0.65ex]
\node (m) {};
\draw[decoration={markings,mark=at position 0.5 with {\arrow{<}}},postaction={decorate}] (m) circle [radius=0.5cm];
\node [below right of=m,label={[label distance=-0.3cm]below right:$\bar\lambda_i-1$}] (dot) {$\bullet$};
\node [left of=m,label=left:$i$] (colour) {};
\node [above right of=m, label={[label distance=0.1cm]above right:$\lambda$}] (weight) {};
\end{tikzpicture} = \id_{1_\lambda} \]

\[
\left(\sum_{r=0}^\infty \begin{tikzpicture}[node distance=0.5cm,baseline=-0.65ex]
\node (m) {};
\draw[decoration={markings,mark=at position 0.5 with {\arrow{>}}},postaction={decorate}] (m) circle [radius=0.5cm];
\node [below right of=m,label={[label distance=-0.3cm]below right:$-\bar{\lambda}_i+r-1$}] (dot) {$\bullet$};
\node [left of=m,label=left:$i$] (colour) {};
\node [above right of=m, label={[label distance=0.1cm]above right:$\lambda$}] (weight) {};
\end{tikzpicture}t^{r}\right)
\left(\sum_{s=0}^\infty \begin{tikzpicture}[node distance=0.5cm,baseline=-0.65ex]
\node (m) {};
\draw[decoration={markings,mark=at position 0.5 with {\arrow{<}}},postaction={decorate}] (m) circle [radius=0.5cm];
\node [below right of=m,label={[label distance=-0.3cm]below right:$\bar{\lambda}_i+s-1$}] (dot) {$\bullet$};
\node [left of=m,label=left:$i$] (colour) {};
\node [above right of=m, label={[label distance=0.1cm]above right:$\lambda$}] (weight) {};
\end{tikzpicture}t^s \right) =\id_{1_\lambda}
\]

\subsubsection{Extended $\mathfrak{sl}_2$ relations with $\lambda_i-\lambda_{i+1}>0$}
\[ \begin{tikzpicture}[baseline=-0.65ex,node distance=0.5cm]
\coordinate (m);
\coordinate [below right of=m] (br);
\coordinate [below left of=m] (bl);
\coordinate [above left of=m] (tl);
\coordinate [above right of=m] (tr);
\coordinate [right of=tr,label=below right:$\lambda$] (trr);
\coordinate [right of=br] (brr);
\node [above of=tl] (ttl) {$i$};
\node [below of=bl] (bbl) {$i$};
\draw [->] (bbl) -- (bl) to [out=90,in=225] (tr);
\draw [->] (tr) to [out=45,in=90] (trr) to [bend left=10] (brr) to [out=270,in=315] (br) to [out=135,in=270] (tl) -- (ttl);
\end{tikzpicture}=0, \quad \begin{tikzpicture}[baseline=-0.65ex]
\node (tl) {};
\node [above of=tl] (ttl) {$i$};
\node [right of=ttl] (ttr) {$i$};
\node [below of=tl] (ml) {$i$};
\node [below of=ttr,label=right:$\lambda$] (tr) {};
\node [below of=tr] (mr) {$i$};
\draw [->] (ttl) to [out=270,in=90] (tr) to [out=270,in=90] (ml);
\draw [<-] (ttr) to [out=270,in=90] (tl) to [out=270,in=90] (mr);
\end{tikzpicture}= - \begin{tikzpicture}[baseline=-0.65ex]
\node (ml) {};
\node [above of=ml] (tl) {$i$};
\node [right of=tl] (tr) {$i$};
\node [below of=ml] (bl) {};
\node [below of=tr,label=right:$\lambda$] (mr) {};
\node [below of=mr] (br) {};
\draw [<-] (tr) -- (br);
\draw [->] (tl) -- (bl);
\end{tikzpicture} \]

\[ \begin{tikzpicture}[baseline=-0.65ex]
\node (tl) {};
\node [above of=tl] (ttl) {$i$};
\node [right of=ttl] (ttr) {$i$};
\node [below of=tl] (ml) {$i$};
\node [below of=ttr,label=right:$\lambda$] (tr) {};
\node [below of=tr] (mr) {$i$};
\draw [<-] (ttl) to [out=270,in=90] (tr) to [out=270,in=90] (ml);
\draw [->] (ttr) to [out=270,in=90] (tl) to [out=270,in=90] (mr);
\end{tikzpicture}= - \begin{tikzpicture}[baseline=-0.65ex]
\node (ml) {};
\node [above of=ml] (tl) {$i$};
\node [right of=tl] (tr) {$i$};
\node [below of=ml] (bl) {};
\node [below of=tr,label=right:$\lambda$] (mr) {};
\node [below of=mr] (br) {};
\draw [->] (tr) -- (br);
\draw [<-] (tl) -- (bl);
\end{tikzpicture} + \sum_{f_1+f_2+f_3=\bar\lambda_i-1} \begin{tikzpicture}[baseline=-0.65ex]
\node [label=left:$\lambda$] (ml) {};
\node [above of=ml] (tl) {$i$};
\node [right of=tl] (tr) {$i$};
\node [below of=ml] (bl) {$i$};
\node [below of=tr] (mr) {};
\node [below of=mr] (br) {$i$};
\draw [->] (bl) to [bend left=90] node [pos=0.3,label=above left:$f_3$] {$\bullet$} (br);
\draw [->] (tr) to [bend left=90] node [pos=0.7,label=below left:$f_1$] {$\bullet$} (tl);
\draw[decoration={markings,mark=at position 0.5 with {\arrow{<}}},postaction={decorate}] (mr) circle [radius=0.5cm];
\node [below right = 0.07 of mr,label={[label distance=-0.3cm]below right:$-\bar{\lambda}_i+f_2-1$}] (dot) {$\bullet$};
\end{tikzpicture}
\]
where $\bar{\lambda}_i=\lambda_i-\lambda_{i+1}$.

\subsubsection{Extended $\mathfrak{sl}_2$ relations with $\lambda_i-\lambda_{i+1}<0$}
\[ \begin{tikzpicture}[baseline=-0.65ex,node distance=0.5cm]
\coordinate (m);
\coordinate [below right of=m] (br);
\coordinate [below left of=m] (bl);
\coordinate [above left of=m] (tl);
\coordinate [above right of=m,label=below right:$\lambda$] (tr);
\coordinate [right of=tr] (trr);
\coordinate [right of=br] (brr);
\node [above of=tr] (ttr) {$i$};
\node [below of=br] (bbr) {$i$};
\draw [->] (bbr) -- (br) to [out=90,in=315] (tl);
\draw [->] (tl) to [out=135,in=90] (tll) to [bend right=10] (bll) to [out=270,in=225] (bl) to [out=45,in=270] (tr) -- (ttr);
\end{tikzpicture}=0, \quad \begin{tikzpicture}[baseline=-0.65ex]
\node (tl) {};
\node [above of=tl] (ttl) {$i$};
\node [right of=ttl] (ttr) {$i$};
\node [below of=tl] (ml) {$i$};
\node [below of=ttr,label=right:$\lambda$] (tr) {};
\node [below of=tr] (mr) {$i$};
\draw [<-] (ttl) to [out=270,in=90] (tr) to [out=270,in=90] (ml);
\draw [->] (ttr) to [out=270,in=90] (tl) to [out=270,in=90] (mr);
\end{tikzpicture}= - \begin{tikzpicture}[baseline=-0.65ex]
\node (ml) {};
\node [above of=ml] (tl) {$i$};
\node [right of=tl] (tr) {$i$};
\node [below of=ml] (bl) {};
\node [below of=tr,label=right:$\lambda$] (mr) {};
\node [below of=mr] (br) {};
\draw [->] (tr) -- (br);
\draw [<-] (tl) -- (bl);
\end{tikzpicture} \]

\[ \begin{tikzpicture}[baseline=-0.65ex]
\node (tl) {};
\node [above of=tl] (ttl) {$i$};
\node [right of=ttl] (ttr) {$i$};
\node [below of=tl] (ml) {$i$};
\node [below of=ttr,label=right:$\lambda$] (tr) {};
\node [below of=tr] (mr) {$i$};
\draw [->] (ttl) to [out=270,in=90] (tr) to [out=270,in=90] (ml);
\draw [<-] (ttr) to [out=270,in=90] (tl) to [out=270,in=90] (mr);
\end{tikzpicture}= - \begin{tikzpicture}[baseline=-0.65ex]
\node (ml) {};
\node [above of=ml] (tl) {$i$};
\node [right of=tl] (tr) {$i$};
\node [below of=ml] (bl) {};
\node [below of=tr,label=right:$\lambda$] (mr) {};
\node [below of=mr] (br) {};
\draw [<-] (tr) -- (br);
\draw [->] (tl) -- (bl);
\end{tikzpicture} + \sum_{f_1+f_2+f_3=-\bar\lambda_i-1} \begin{tikzpicture}[baseline=-0.65ex]
\node [label=left:$\lambda$] (ml) {};
\node [above of=ml] (tl) {$i$};
\node [right of=tl] (tr) {$i$};
\node [below of=ml] (bl) {$i$};
\node [below of=tr] (mr) {};
\node [below of=mr] (br) {$i$};
\draw [<-] (bl) to [bend left=90] node [pos=0.3,label=above left:$f_3$] {$\bullet$} (br);
\draw [<-] (tr) to [bend left=90] node [pos=0.7,label=below left:$f_1$] {$\bullet$} (tl);
\draw[decoration={markings,mark=at position 0.5 with {\arrow{>}}},postaction={decorate}] (mr) circle [radius=0.5cm];
\node [below right = 0.07 of mr,label={[label distance=-0.3cm]below right:$\bar{\lambda}_i+f_2-1$}] (dot) {$\bullet$};
\end{tikzpicture}
\]
where $\bar{\lambda}_i=\lambda_i-\lambda_{i+1}$.

\subsubsection{Extended $\mathfrak{sl}_2$ relations with $\lambda_i-\lambda_{i+1}=0$}
\[ \begin{tikzpicture}[baseline=-0.65ex]
\node (tl) {};
\node [above of=tl] (ttl) {$i$};
\node [right of=ttl] (ttr) {$i$};
\node [below of=tl] (ml) {$i$};
\node [below of=ttr,label=right:$\lambda$] (tr) {};
\node [below of=tr] (mr) {$i$};
\draw [->] (ttl) to [out=270,in=90] (tr) to [out=270,in=90] (ml);
\draw [<-] (ttr) to [out=270,in=90] (tl) to [out=270,in=90] (mr);
\end{tikzpicture}= - \begin{tikzpicture}[baseline=-0.65ex]
\node (ml) {};
\node [above of=ml] (tl) {$i$};
\node [right of=tl] (tr) {$i$};
\node [below of=ml] (bl) {};
\node [below of=tr,label=right:$\lambda$] (mr) {};
\node [below of=mr] (br) {};
\draw [<-] (tr) -- (br);
\draw [->] (tl) -- (bl);
\end{tikzpicture} \quad \begin{tikzpicture}[baseline=-0.65ex]
\node (tl) {};
\node [above of=tl] (ttl) {$i$};
\node [right of=ttl] (ttr) {$i$};
\node [below of=tl] (ml) {$i$};
\node [below of=ttr,label=right:$\lambda$] (tr) {};
\node [below of=tr] (mr) {$i$};
\draw [<-] (ttl) to [out=270,in=90] (tr) to [out=270,in=90] (ml);
\draw [->] (ttr) to [out=270,in=90] (tl) to [out=270,in=90] (mr);
\end{tikzpicture}= - \begin{tikzpicture}[baseline=-0.65ex]
\node (ml) {};
\node [above of=ml] (tl) {$i$};
\node [right of=tl] (tr) {$i$};
\node [below of=ml] (bl) {};
\node [below of=tr,label=right:$\lambda$] (mr) {};
\node [below of=mr] (br) {};
\draw [->] (tr) -- (br);
\draw [<-] (tl) -- (bl);
\end{tikzpicture}
\]

\[\begin{tikzpicture}[baseline=-0.65ex,node distance=0.5cm]
\coordinate (m);
\coordinate [below right of=m] (br);
\coordinate [below left of=m] (bl);
\coordinate [above left of=m] (tl);
\coordinate [above right of=m,label=below right:$\lambda$] (tr);
\coordinate [right of=tr] (trr);
\coordinate [right of=br] (brr);
\node [above of=tr] (ttr) {$i$};
\node [below of=br] (bbr) {$i$};
\draw [->] (bbr) -- (br) to [out=90,in=315] (tl);
\draw [->] (tl) to [out=135,in=90] (tll) to [bend right=10] (bll) to [out=270,in=225] (bl) to [out=45,in=270] (tr) -- (ttr);
\end{tikzpicture}= \begin{tikzpicture}[baseline=-0.65ex]
\coordinate [label=right:$\lambda$] (m);
\node [above of=m] (t) {$i$};
\node [below of=m] (b) {$i$};
\draw [->] (b) -- (t);
\end{tikzpicture} = -\begin{tikzpicture}[baseline=-0.65ex,node distance=0.5cm]
\coordinate (m);
\coordinate [below right of=m] (br);
\coordinate [below left of=m] (bl);
\coordinate [above left of=m] (tl);
\coordinate [above right of=m] (tr);
\coordinate [right of=tr,label=below right:$\lambda$] (trr);
\coordinate [right of=br] (brr);
\node [above of=tl] (ttl) {$i$};
\node [below of=bl] (bbl) {$i$};
\draw [->] (bbl) -- (bl) to [out=90,in=225] (tr);
\draw [->] (tr) to [out=45,in=90] (trr) to [bend left=10] (brr) to [out=270,in=315] (br) to [out=135,in=270] (tl) -- (ttl);
\end{tikzpicture}
\]

\begin{definition}
The \emph{Karoubi envelope} $\operatorname{Kar}(\mathscr{C})$ of a 2-category $\mathscr{C}$ is the 2-category with the same objects as $\mathscr{C}$ and whose Hom-categories are the Karoubi envelopes of the Hom-categories of $\mathscr{C}$.
\end{definition}
\begin{definition}
We define $\dot{\mathscr{U}}_Q(\mathfrak{gl}(p))$ to be $\operatorname{Kar}({\mathscr{U}}_Q(\mathfrak{gl}(p)))$.
\end{definition}

We take the split Grothendieck group $K_0(\mathscr{C})$ of a 2-category $\mathscr{C}$ to be the direct sum of the split Grothendieck groups of each Hom-category. If a category $\mathscr{D}$ is graded, then we can give $K_0(\mathscr{D})$ the structure of a $\Z[q,q^{-1}]$-module by letting $q$ act as the grading shift functor.

\begin{thm}[Khovanov-Lauda \cite{Khovanov2010a}] The 2-category $\dot{\mathscr{U}}_Q(\mathfrak{gl}(p))$ satisfies
\[ K_0(\dot{\mathscr{U}}_Q(\mathfrak{gl}(p)))\otimes_{\Z[q,q^{-1}]} \C(q)\cong \dot{U}_q(\mathfrak{gl}(p)). \]
\end{thm}

Without tensoring with $\C(q)$, the algebra $K_0(\dot{\mathscr{U}}_Q(\mathfrak{gl}(p)))$ is isomorphic to the integral form of $\dot{U}_q(\mathfrak{gl}(p))$, generated over $\Z[q,q^{-1}]$ by $E_i^{(r)}$ and $F_i^{(r)}$ for all $i$.

We can categorify $\dot U^\infty_q(\mathfrak{gl}(p))$ in a straight-forward way.

\begin{definition}
We define $\dot{\mathscr{U}}^\infty_Q(\mathfrak{gl}(p))$ to be the quotient of $\dot{\mathscr{U}}_Q(\mathfrak{gl}(p)))$ by the identity 2-morphisms of the identity 1-morphisms of the objects $1_\lambda$ with $\lambda_i<0$ for some $i$.
\end{definition}

In terms of the planar diagrams above, we can interpret this as meaning a 2-morphism is $0$ if it contains a region coloured with a weight $\lambda$ with $\lambda_i<0$ for some $i$. Note that this is a direct sum of the categorified $q$-Schur algebras defined by Mackaay, Sto\v{s}i\`{c} and Vaz \cite{Mackaay2010}.

We can also categorify $\dot U^\infty_q(\mathfrak{gl}(\infty))$.

\begin{lem}
The inclusion 
\[ \iota_0:\dot U^\infty_q(\mathfrak{gl}(p)) \to \dot U^\infty_q(\mathfrak{gl}(p+1)) \]
from Section \ref{sec:branchingrules} lifts to a 2-functor
\[ \dot{\mathscr{U}}^\infty_Q(\mathfrak{gl}(p)) \to \dot{\mathscr{U}}^\infty_Q(\mathfrak{gl}(p+1)).\]
\end{lem}
\begin{proof}
The 2-functor carries $1_{\lambda}$ to $1_{(\lambda,0)}$ and $E_i1_{\lambda}$ to $E_i1_{(\lambda,0)}$ and similarly for $F_i1_\lambda$. The 2-functor acts on 2-morphisms only by changing the colouring of a region from $\lambda$ to $(\lambda,0)$. Since all relations involving an $i$ coloured strand depend only on the value of $\lambda_i-\lambda_{i+1}$, relations on 2-morphisms in $\dot{\mathscr{U}}^\infty_Q(\mathfrak{gl}(p))$ are mapped to relations on 2-morphisms in $\dot{\mathscr{U}}^\infty_Q(\mathfrak{gl}(p+1))$, so the 2-functor is well-defined.
\end{proof}

Indecomposable objects in $\dot{\mathscr{U}}_Q(\mathfrak{gl}(p))$ are carried to indecomposable objects in $\dot{\mathscr{U}}_Q(\mathfrak{gl}(p+1))$. Due to Webster \cite{Webster2012}, the indecomposable objects correspond to the elements of Lusztig's canonical basis.

\begin{definition}
We let $\dot{\mathscr{U}}^\infty_Q(\mathfrak{gl}(\infty))$ be the direct limit of the system of 2-categories
\[ \dot{\mathscr{U}}^\infty_Q(\mathfrak{gl}(\infty)) = \lim_{\rightarrow} \left(\xymatrix{\cdots \ar[r] & \dot{\mathscr{U}}^\infty_Q(\mathfrak{gl}(p)) \ar[r]& \dot{\mathscr{U}}^\infty_Q(\mathfrak{gl}(p+1)) \ar[r] &\cdots} \right). \]
\end{definition}

By construction, this satisfies
\[ K_0(\dot{\mathscr{U}}^\infty_Q(\mathfrak{gl}(\infty)))\otimes_{\Z[q,q^{-1}]} \C(q) \cong \dot U^\infty_q(\mathfrak{gl}(\infty)). \]

Categorification of $\dot U_q^{(m|n)}(\mathfrak{gl}(\infty))$ is more difficult, and relies on a notion of categorification of the modules $V_\infty(\lambda)$. This will be left to Section \ref{sec:catirreducibleweightmodules}.

\subsection{KLR algebras}\label{sec:KLRalgebras}
The definition of $\dot{\mathscr{U}}_Q(\mathfrak{gl}(p))$ made use of relations coming from the KLR algebra. This algebra appears in other related settings, and in particular we will make use of it to categorify highest-weight representations of $U_q(\mathfrak{gl}(p))$.

Let $I=\{1,\ldots,p-1\}$. Given an element $\nu=(\nu_1,\ldots,\nu_n)\in I^n$, we let $s_l(\nu)=(\nu_1,\ldots,\nu_{l+1},\nu_l,\ldots,\nu_n)$ interchanging only the $l$ and $(l+1)$th entries.

As before, let $\kk$ be a commutative ring with unit, and for all $i,j\in I$ let $t_{ij}\in \kk^\times$ such that $t_{ij}=t_{ji}$ when $j\neq i\pm 1$, and $t_{ii}=1$ for all $i$, with $Q=\{t_{ij}\mid i,j\in I\}$.

\begin{definition}
The KLR algebra $R(n)$ of degree $n$ of type $A_{p-1}$ with the choice of scalars $Q$ is defined to be the unital algebra over $\kk$ defined by generators $e(\nu)$ ($\nu\in \{1,\ldots,p-1\}^n$), $x_k$ ($1\leq k\leq n$) and $\tau_l$ ($1\leq l \leq n-1$) with defining relations:

\[e(\nu)e(\nu')=\delta_{\nu,\nu'}e(\nu), \quad \sum_{\nu\in I^n} = 1 \]
\[ x_kx_l=x_lx_k, \quad x_ke(\nu)=e(\nu)x_k \]
\[ \tau_l e(\nu) = e(s_l(\nu))\tau_l, \quad \tau_k\tau_l=\tau_l\tau_k \, \text{if } |k-l|>1 \]
\[ \tau^2_ke(\nu)=\left\lbrace \begin{array}{ll}
0 & \text{if } \nu_k=\nu_{k+1} \\
t_{\nu_k\nu_{k+1}}e(\nu) & \text{if } |\nu_k-\nu_{k+1}|>1 \\
(t_{\nu_k\nu_{k+1}}x_k+t_{\nu_{k+1}\nu_k}x_{k+1})e(\nu) & \text{if } \nu_k=\nu_{k+1}\pm 1
\end{array}\right. \]
\[ (\tau_kx_l-x_{s_k(l)}\tau_k)e(\nu)=\left\lbrace \begin{array}{ll}
-e(\nu) & \text{if } l=k, \nu_k=\nu_{k+1} \\
e(\nu) & \text{if } l=k+1, \nu_k=\nu_{k+1} \\
0 & \text{otherwise}
\end{array}\right. \]
\[(\tau_{k+1}\tau_k\tau_{k+1}-\tau_k\tau_{k+1}\tau_k)e(\nu)=
\left\lbrace \begin{array}{ll}
t_{\nu_k\nu_{k+1}}e(\nu)& \text{if } \nu_k=\nu_{k+2} \text{ and } \nu_{k+1}=\nu_{k}\pm 1\\
0 & \text{otherwise}
\end{array}\right. \]
The algebra $R(n)$ is graded, with $\deg(e(\nu))=0$, $\deg(x_k)=2$, $\deg(\tau_le(\nu))=-\alpha_{\nu_l}\cdot\alpha_{\nu_{l+1}}$.
\end{definition}

In the following, we let $E_{\mathbf{i}}1_\lambda=E_{i_n}\cdots E_{i_1}1_\lambda$ for any $\mathbf{i}\in I^n$.
\begin{lem}\label{lem:KLRbimodule} In the $2$-category ${\mathscr{U}}_Q(\mathfrak{gl}(p))$, the space
\[\bigoplus_{\mathbf{j}\in I^n}\operatorname{Hom}(E_{\mathbf{i}}1_\lambda,E_{\mathbf{j}}1_\lambda)\]
is a right module over $R(n)$ for each $\mathbf{i}\in I^n$.
\end{lem}
\begin{proof}
The right-module structure is given by $e(\mathbf{i})$ acting as projection to \[ \operatorname{Hom}(E_{\mathbf{i}}1_\lambda,E_{\mathbf{j}}1_\lambda),\] $x_k$ acting by adding a dot to the top of the $k$th strand, and $\tau_k$ adding a $4$-valent crossing between the top of the $k$ and $(k+1)$st strands.
\end{proof}

In fact, if we denote by $\Pi_\lambda$ the graded commutative $\kk$-algebra freely generated by all bubbles
\[  \begin{tikzpicture}[node distance=0.5cm,baseline=-0.65ex]
\node (m) {};
\draw[decoration={markings,mark=at position 0.5 with {\arrow{>}}},postaction={decorate}] (m) circle [radius=0.5cm];
\node [below right of=m,label={[label distance=-0.3cm]below right:$-\lambda_i+\lambda_{i+1}-1+\alpha$}] (dot) {$\bullet$};
\node [left of=m,label=left:$i$] (colour) {};
\node [above right of=m, label={[label distance=0.1cm]above right:$\lambda$}] (weight) {};
\end{tikzpicture} \]
if $\lambda_i-\lambda_{i+1}<0$, or else by
\[ \begin{tikzpicture}[node distance=0.5cm,baseline=-0.65ex]
\node (m) {};
\draw[decoration={markings,mark=at position 0.5 with {\arrow{<}}},postaction={decorate}] (m) circle [radius=0.5cm];
\node [below right of=m,label={[label distance=-0.3cm]below right:$\lambda_i-\lambda_{i+1}-1+\alpha$}] (dot) {$\bullet$};
\node [left of=m,label=left:$i$] (colour) {};
\node [above right of=m, label={[label distance=0.1cm]above right:$\lambda$}] (weight) {};
\end{tikzpicture}\]
if $\lambda_i-\lambda_{i+1}\geq 0$, then we have the following stronger statement. We denote by $\operatorname{HOM}(x,y)$ the space of morphisms $x\to q^k y$ for any $k\in \Z$. That is, $\operatorname{HOM}(x,y)= \bigoplus_k \operatorname{Hom}(x,q^ky)$.

\begin{thm}[Khovanov-Lauda \cite{Khovanov2010a}]
The map
\[ e(\mathbf{j})R(n)e(\mathbf{i})\otimes_{\kk} \Pi_\lambda \to \operatorname{HOM}(E_{\mathbf{i}}1_\lambda,E_{\mathbf{j}}1_\lambda) \]
is an isomorphism.
\end{thm}

Thus the KLR algebra $R=\bigoplus_n R(n)$ essentially contains all information about mappings between the $E_\mathbf{i}$'s. Due to the biadjunction between $E_i$ and $F_i$, we could also say that the KLR algebra contains all information about the mappings between the $F_\mathbf{i}$'s.

Let $U^{-}_q(\mathfrak{gl}(p))$ be the subalgebra of $U_q(\mathfrak{gl}(p))$ generated by the $F_i$.

\begin{thm}
There is an isomorphism \[U^{-}_q(\mathfrak{gl}(p))\cong K_0(\operatorname{p-mod}R)\otimes_{\Z[q,q^{-1}]}\C(q) \]
where $\operatorname{p-mod}R$ is the category of finitely generated graded projective left $R$-modules.
\end{thm}

\subsection{Alternative definition}
It will be useful to have the following definition of $\mathscr{U}_Q(\mathfrak{gl}(p))$ due to Rouquier \cite{Rouquier2008}.

\begin{definition}
The strict additive $\kk$-linear 2-category ${\mathscr{U}}_q(\mathfrak{gl}(p))$ is defined as follows:
\begin{itemize}
\item Objects: $1_\lambda$ for each $\lambda\in\Z^p$.
\item 1-morphisms: direct sums of concatenations of $q^kE_i1_\lambda:1_\lambda \to 1_{\lambda+\alpha_i}$ and $q^kF_i1_\lambda:1_\lambda\to 1_{\lambda-\alpha_i}$ where $\alpha_i=(0,\ldots,1,-1,\ldots,0)$ and $k\in \Z$.
\item 2-morphisms: generated by $x_i:E_i1_\lambda\to q^2E_i1_\lambda$, $\tau_{ij}:E_iE_j1_\lambda\to q^{-\alpha_i\cdot\alpha_j}E_jE_i1_\lambda$, $\epsilon_i:E_iF_i1_\lambda\to q^{1-\lambda\cdot \alpha_i}1_\lambda$, $\eta_i:1_\lambda\to q^{1+\lambda\cdot \alpha_i}F_iE_i$ subject to:
\begin{enumerate}
\item $\epsilon_i\id_{E_i1_\lambda}\circ \id_{E_i1_\lambda}\eta_i=\id_{E_i1_\lambda}$ and $\id_{F_i1_\lambda}\epsilon_i \circ \eta_i\id_{F_i1_\lambda}=\id_{F_i1_\lambda}$.
\item The 2-morphisms $x_i$ and $\tau_{ij}$ obey the KLR relations.
\item Let $\sigma_{ij}=\id_{F_jE_i1_{\lambda}}\epsilon_j\circ \id_{F_j1_{\lambda+\alpha_i}}\tau_{ji} \id_{F_j1_\lambda}\circ \eta_j\id_{E_iF_j1_\lambda}:E_iF_j1_\lambda\to F_jE_i1_\lambda$. Then the 2-morphisms:
\begin{itemize}
\item $\sigma_{ij}$ for $i\neq j$.
\item If $\lambda_i-\lambda_{i+1}\geq 0$, \[ \rho_i=\sigma_{ii}\oplus \bigoplus_{k=0}^{\lambda_i-\lambda_{i+1}-1} \epsilon_i\circ x^k_i\id_{F_i1_{\lambda\alpha_i}}:E_iF_i1_\lambda \to F_iE_i1_\lambda\oplus\bigoplus_{k=0}^ {\lambda_i-\lambda_{i+1}-1} q^{-\lambda_i+\lambda_{i+1}+2k+1}1_\lambda\]
\item If $\lambda_i-\lambda_{i+1}\leq 0$, \[ \rho_i=\sigma_{ii}\oplus \bigoplus_{k=0}^{-\lambda_i+\lambda_{i+1}-1} \id_{F_i1_{\lambda+\alpha_i}}x^k_i \circ \eta_i:E_iF_i1_\lambda\bigoplus_{k=0} ^{-\lambda_i+\lambda_{i+1}-1}q^ {-\lambda_i+\lambda_{i+1}-2k-1}1_\lambda \to F_iE_i1_\lambda.\]
\end{itemize}
are invertible.
\end{enumerate}
\end{itemize}
\end{definition}

\begin{thm}[Brundan \cite{Brundan2015}]\label{thm:same}
The two definitions of ${\mathscr{U}}_q(\mathfrak{gl}(p))$ agree.
\end{thm}

Thus the entire structure of the 2-category is determined by a relatively small number of axioms. Essentially one takes the KLR algebra from Section \ref{sec:KLRalgebras} interpreted as 2-morphisms between the $E_i$'s, then formally adjoins a right-adjoint $F_i$ to each $E_i$, and then localises at maps between compositions of $E_i$'s and $F_j$'s. It is a remarkable fact that this makes $F_i$ also into a left-adjoint for $E_i$, so that $E_i$ and $F_i$ are both left and right adjoint to each other.

As before we obtain $\dot{\mathscr{U}}_q(\mathfrak{gl}(p))$ by taking the Karoubi envelope of ${\mathscr{U}}_q(\mathfrak{gl}(p))$.

\subsection{2-representations}
In this section we define an appropriate `higher' analogue of a representation of $\mathfrak{gl}(p)$. Rather than studying the action on a vector space, here we study the action of $\dot{\mathscr{U}}_Q(\mathfrak{gl}(p))$ on a category. Hence we define a mapping of $\dot{\mathscr{U}}_Q(\mathfrak{gl}(p))$ into 2-category, that can be thought of as a category of functors and natural transformations acting on another category. 

There are a few competing definitions of what a 2-representation should be. This is partly due to the definitions being made in the absence of Theorem \ref{thm:same}, so that it was not clear what conditions a 2-representation would need to satisfy.

A definition of a $Q$-strong 2-representation was given by Cautis and Lauda \cite{Cautis2011}, which is based on the Khovanov-Lauda definition of the 2-category but gives a relatively small list of conditions to check, proving that much of the structure is given `for free', as in Rouquier's definition of the 2-category.

Because of Theorem \ref{thm:same} due to Brundan \cite{Brundan2015}, we can define a 2-representation according to Rouquier \cite{Rouquier2008}, which gives a very small list of conditions a 2-representation needs to satisfy.

\begin{definition}[Rouquier \cite{Rouquier2008}]
A 2-representation of $U_q(\mathfrak{gl}(p))$ is a graded additive $\kk$-linear 2-category $\mathscr{C}$, and a strict 2-functor $\mathscr{U}_Q(\mathfrak{gl}(p))\to \mathscr{C}$. This is equivalent to the following:
\begin{itemize}
\item There exists a family of objects $(V_\lambda)_{\lambda\in\Z^p}$ of $\mathscr{C}$.
\item There exist $1$-morphisms $E_iV_\lambda:V_\lambda\to V_{\lambda+\alpha_i}$ and $F_iV_\lambda:V_{\lambda}\to V_{\lambda-\alpha_i}$ in $\mathscr{C}$.
\item For all $\lambda$, there are 2-morphisms $x_i:E_i1_\lambda\to E_i1_\lambda$ and $\tau_{ij}:E_{j}E_i1_\lambda \to E_iE_j1_\lambda$ satisfying the KLR relations.
\item $E_iV_\lambda$ is left-adjoint to $F_iV_\lambda$.
\item The maps $\rho_i$ and $\sigma_{ij}$ for $i\neq j$ map to isomorphisms.
\end{itemize}
\end{definition}

\begin{definition}
A 2-representation of $U_q(\mathfrak{gl}(p))$ is said to be integrable if for each $\lambda\in\Z^p$ and each object $M\in V_\lambda$ there exists $n$ such that for all $i$ we have
\[ E_i^n(M)=F_i^n(M)=0. \]
\end{definition}

The following conditions for the existence of an integrable 2-representation were given by Cautis and Lauda \cite{Cautis2011}.

\begin{thm}\label{thm:integrable2rep}
Let $\kk$ be a field, and $\mathscr{C}$ be a graded additive $\kk$-linear idempotent-complete 2-category. The following data gives rise to an integrable 2-representation $\mathscr{C}$ of $U_q(\mathfrak{gl}(p))$:
\begin{itemize}
\item There exists a family of objects $(V_\lambda)_{\lambda\in\Z^p}$ of $\mathscr{C}$.
\item There exist 1-morphisms $E_iV_\lambda:V_\lambda\to V_{\lambda+\alpha_i}$ and $F_iV_\lambda:V_{\lambda}\to V_{\lambda-\alpha_i}$ in $\mathscr{C}$.
\item $V_{\lambda+r\alpha_i}$ is isomorphic to $0$ for all but finitely many $r\in \Z$.
\item $E_iV_\lambda$ has a left adjoint and a right adjoint, and its right adjoint is $F_iV_{\lambda+\alpha_i}$.
\item $\operatorname{Hom}_\mathscr{C} (\id_{V_\lambda},q^l\id_{V_\lambda})$ is $0$ if $l<0$ and generated by the identity $2$-morphism if $l=0$. All hom-spaces between $1$-morphisms are finite dimensional.
\item In $\mathscr{C}$,
\[ F_iE_iV_\lambda \cong E_iF_iV_\lambda \bigoplus_{k=0}^{-\lambda_i+\lambda_{i+1}-1} q^{-\lambda_i+\lambda_{i+1}+2k+1} V_{\lambda}, \text{if } \lambda_i-\lambda_{i+1}\leq 0 \]
\[ E_iF_iV_\lambda \cong F_iE_iV_\lambda \bigoplus_{k=0}^{\lambda_i-\lambda_{i+1}-1} q^{\lambda_i-\lambda_{i+1}-2k-1}V_\lambda, \text{if } \lambda_i-\lambda_{i+1}\geq 0
\]
where multiplication by $q$ denotes the grading shift functor on $\mathscr{C}$.
\item There are 2-morphisms $x_i:E_iV_\lambda\to E_iV_\lambda$ and $\tau_{ij}:E_{j}E_iV_\lambda \to E_iE_jV_\lambda$ satisfying the KLR relations.
\item If $i\neq j$, then $F_jE_iV_\lambda \cong E_iF_jV_\lambda$.
\end{itemize}
\end{thm}

Note that it follows that $F_iV_{\lambda+\alpha_i}$ is also left-adjoint to $E_iV_\lambda$.

Since $\mathscr{U}_Q(\mathfrak{gl}(\infty))$ looks `locally' like $\mathscr{U}_Q(\mathfrak{gl}(p))$, we can extend the notion of $2$-representation to $\mathscr{U}_Q(\mathfrak{gl}(\infty))$.

\begin{definition}
A 2-representation of $U_q(\mathfrak{gl}(\infty))$ is a graded additive $\kk$-linear 2-category $\mathscr{C}$, and a strict 2-functor $\mathscr{U}_Q(\mathfrak{gl}(\infty))\to \mathscr{C}$. This is equivalent to the following:
\begin{itemize}
\item There exists a family of objects $(V_\lambda)_{\lambda\in\Z^\infty}$ of $\mathscr{C}$.
\item There exist $1$-morphisms $E_iV_\lambda:V_\lambda\to V_{\lambda+\alpha_i}$ and $F_iV_\lambda:V_{\lambda}\to V_{\lambda-\alpha_i}$ in $\mathscr{C}$.
\item For all $\lambda$, there are 2-morphisms $x_i:E_i1_\lambda\to E_i1_\lambda$ and $\tau_{ij}:E_{j}E_i1_\lambda \to E_iE_j1_\lambda$ satisfying the KLR relations.
\item $E_iV_\lambda$ is left-adjoint to $F_iV_\lambda$.
\item The maps $\rho_i$ and $\sigma_{ij}$ for $i\neq j$ map to isomorphisms.
\end{itemize}
\end{definition}

\begin{definition}
A 2-representation of $U_q(\mathfrak{gl}(\infty))$ is said to be integrable if for each $\lambda\in\Z^\infty$ and each object $M\in V_\lambda$ there exists $n$ such that for all $i$ we have
\[ E_i^n(M)=F_i^n(M)=0. \]
\end{definition}

Since Theorem \ref{thm:integrable2rep} holds for all $p\geq 2$, it is easy to see that it also holds for $p=\infty$. The notion of an integrable 2-representation of $U_q(\mathfrak{gl}(\infty))$ is then a categorification of the notion of an integrable weight-module over $U_q(\mathfrak{gl}(\infty))$, in the sense of Du and Fu \cite{Du2009}.

\section{Foams}\label{sec:foams}
Lemma \ref{lem:ladders} tells us that $\dot U^\infty_q(\mathfrak{gl}(p))$ is equivalent to the category of ladders on $p$ uprights. It follows that $\dot U^\infty_q(\mathfrak{gl}(\infty))$ is equivalent to the category of ladders with any finite number of uprights.

In this section we relate the categorification $\dot{\mathscr{U}}^\infty_Q(\mathfrak{gl}(\infty))$ of $\dot U^\infty_q(\mathfrak{gl}(\infty))$ to a categorification of the category of ladders. As before we let $\kk$ be a commutative unital ring.

\begin{definition}The 2-category $\operatorname{Foam}$ is defined as follows:
\begin{itemize}
\item Objects are sequences $(\lambda_1,\lambda_2,\ldots)$ with $\lambda_i=0$ for all but finitely many $\lambda_i$.
\item 1-morphisms $\lambda \to \mu$ are ladder diagrams with uprights at the bottom coloured by $\lambda$ and uprights at the top coloured by $\mu$.
\item 2-morphisms are $\kk$-linear combinations of labelled, decorated singular surfaces with oriented seams whose generic slices are ladder diagrams, generated by the following:

\begin{figure}[H]
$\id =$ \raisebox{-50pt}{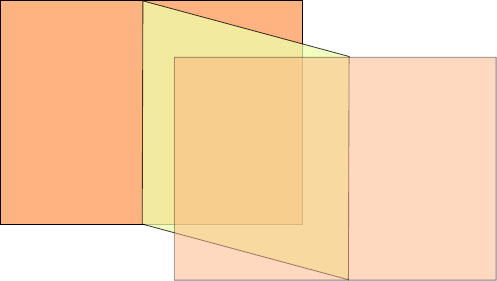} $x_i=$\raisebox{-50pt}{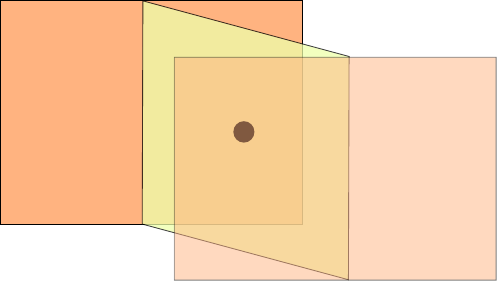}
\end{figure}
\begin{figure}[H]
$\tau_{ii}=$\raisebox{-50pt}{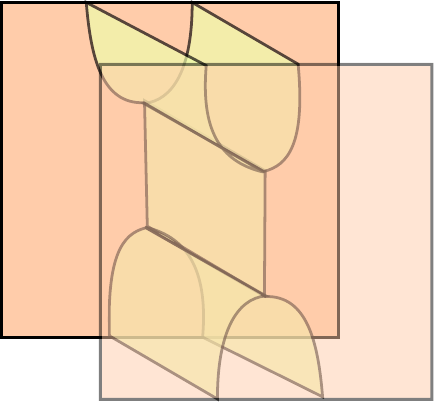}
\end{figure}
\begin{figure}[H]
$\tau_{i,i+1}=$\raisebox{-50pt}{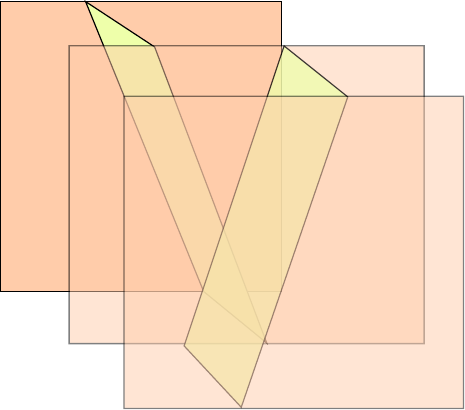} $\tau_{i+1,i}=$\raisebox{-50pt}{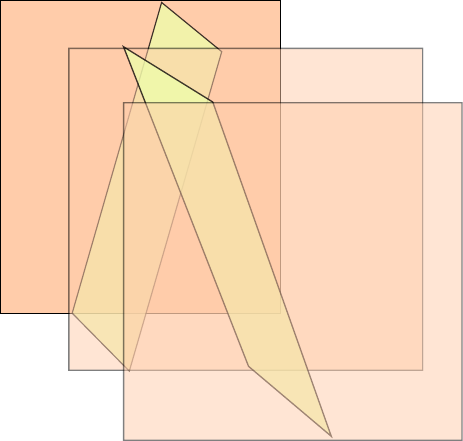}
\end{figure}

\begin{figure}[H]
$\tau_{i,j}=$\raisebox{-50pt}{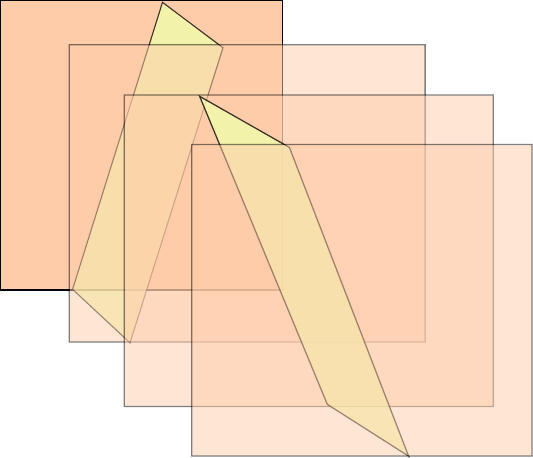}
$\tau_{j,i}=$\raisebox{-50pt}{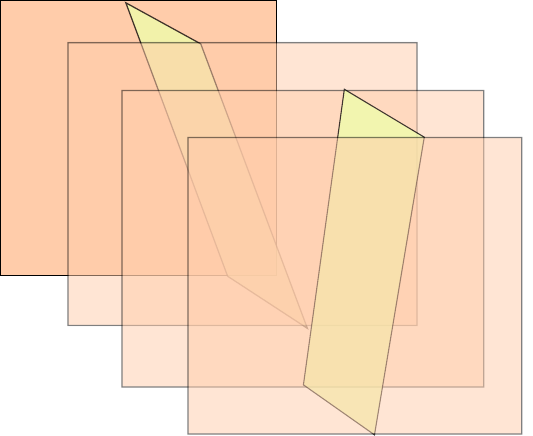}
\end{figure}

\begin{figure}[H]
$\eta_i=$\raisebox{-50pt}{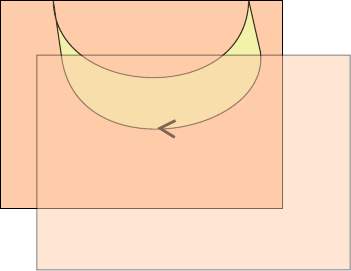} $\epsilon_i=$\raisebox{-50pt}{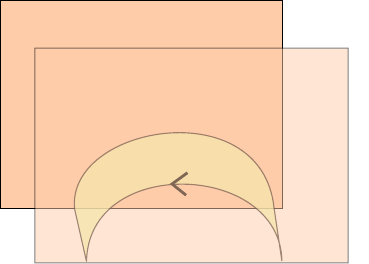}
\end{figure}

\begin{figure}[H]
$\eta'_i=$\raisebox{-50pt}{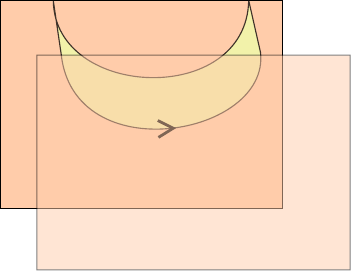}
$\epsilon'_i=$\raisebox{-50pt}{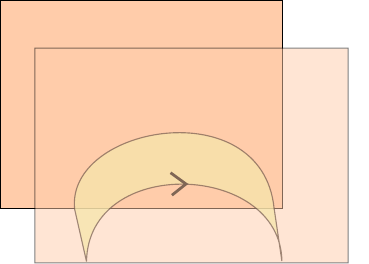}
\end{figure}
such that the 2-functor $\mathscr{U}^\infty_Q(\mathfrak{gl}(\infty))\to \operatorname{Foam}$ defined by
\[E_i1_{\lambda} \mapsto \begin{tikzpicture}[baseline=-0.65ex]
\draw (-0.75,-1) -- (-0.75,1);
\draw (0.75,-1) -- (0.75,1);
\draw (-0.75,0.25) -- (0.75,-0.25);
\draw (-2,-1) -- (-2,1);
\draw (2,-1) -- (2,1);
\draw (0,0.3) node {$1$};
\draw (-1,-1.25) node {$\lambda_i$};
\draw (-1,1.25) node {$\lambda_i+1$};
\draw (-1.4,0) node {$\cdots$};
\draw (1.4,0) node {$\cdots$};
\draw (1,-1.25) node {$\lambda_{i+1}$};
\draw (1,1.25) node {$\lambda_{i+1}-1$};
\draw (-2,-1.25) node {$\lambda_1$};
\draw (2,-1.25) node {$\lambda_m$};
\draw (-2,1.25) node {$\lambda_1$};
\draw (2,1.25) node {$\lambda_m$};
\end{tikzpicture} \]

\[ F_i 1_{\lambda}\mapsto \begin{tikzpicture}[baseline=-0.65ex]
\draw (-0.75,-1) -- (-0.75,1);
\draw (0.75,-1) -- (0.75,1);
\draw (-0.75,-0.25) -- (0.75,0.25);
\draw (-2,-1) -- (-2,1);
\draw (2,-1) -- (2,1);
\draw (0,0.3) node {$1$};
\draw (-1,-1.25) node {$\lambda_i$};
\draw (-1,1.25) node {$\lambda_i-1$};
\draw (-1.4,0) node {$\cdots$};
\draw (1.4,0) node {$\cdots$};
\draw (1,-1.25) node {$\lambda_{i+1}$};
\draw (1,1.25) node {$\lambda_{i+1}+1$};
\draw (-2,-1.25) node {$\lambda_1$};
\draw (2,-1.25) node {$\lambda_m$};
\draw (-2,1.25) node {$\lambda_1$};
\draw (2,1.25) node {$\lambda_m$};
\end{tikzpicture}  \]

\[ \begin{tikzpicture}[baseline=-0.65ex]
\node (middle)[label=left:$\lambda+\alpha_i$,label=right:$\lambda$] {$\bullet$};
\node [below of=middle] (bottom) {};
\node [above of=middle] (top) {$i$};
\draw [->] (bottom) -- (top);
\end{tikzpicture} \mapsto x_i, \quad  \begin{tikzpicture}[baseline=-0.65ex]
\node (middle) {};
\node [above of=middle] (tl) {$i$};
\node [right of=tl] (tr) {$j$};
\node [below of=middle] (bl) {$j$};
\node [right of=bl] (br) {$i$};
\node [right of=middle, label=right:$\lambda$] (r) {};
\draw [->] (bl) to [out=90,in=270] (tr);
\draw [->] (br) to [out=90,in=270] (tl);
\end{tikzpicture} \mapsto \tau_{ij} \]
\[ \begin{tikzpicture}[baseline=-0.5cm]
\node (mid) {$i$};
\node [right of=mid] (r) {$i$};
\draw [->] (mid) to [bend right=90] (r);
\end{tikzpicture} \mapsto \eta_i, \quad \begin{tikzpicture}[baseline=+0.25cm]
\node (mid) {$i$};
\node [right of=mid] (r) {$i$};
\draw [->] (mid) to [bend left=90] (r);
\end{tikzpicture}\mapsto \epsilon_i \]
\[ \begin{tikzpicture}[baseline=-0.5cm]
\node (mid) {$i$};
\node [right of=mid] (r) {$i$};
\draw [<-] (mid) to [bend right=90] (r);
\end{tikzpicture} \mapsto \eta'_i, \quad \begin{tikzpicture}[baseline=+0.25cm]
\node (mid) {$i$};
\node [right of=mid] (r) {$i$};
\draw [<-] (mid) to [bend left=90] (r);
\end{tikzpicture} \mapsto \epsilon'_i \]
is an equivalence of 2-categories.
\end{itemize} 

\end{definition}
\begin{remark}
The definition of the foam 2-category in \cite{Queffelec} and \cite{Lauda2012} is less tautological, and in particular avoids the more rigid setting of ladder diagrams, using instead web diagrams. However, the resulting 2-categories are equivalent as noted in \cite[Proposition 3.22]{Queffelec}. The use of ladder-based foams will be unavoidable in what follows, hence we use this definition.
\end{remark}
\begin{remark}
Note that dots can only be placed on the yellow facets, which are traces of rungs of the ladder diagrams obtained by slicing horizontally through each foam diagram. In \cite{Queffelec} and \cite{Lauda2012} (and indeed \cite{Mackaay2007}) there are algebraic relations allowing dots to migrate to adjacent facets. However, this uses the fact that they are working specifically with $\mathfrak{sl}(n)$ foams, and in particular uses the rule that $n$ dots on a 1-labelled facet gives the $0$ foam. In general, dot migration is more complicated, as we shall see in Section \ref{sec:non-local}.
\end{remark}

\section{Categorification of Irreducible Highest Weight Modules}\label{sec:catirreducibleweightmodules}
In this section, we review the categorification of highest-weight $U_q(\mathfrak{gl}(p))$-modules using cyclotomic KLR algebras.

From now on, we choose the base ring $\kk$ to be a field.

\begin{definition}
Let $\lambda=(\lambda_1,\ldots,\lambda_p)$ be a dominant weight. We define the cyclotomic KLR algebra $R^\lambda(n)$ of degree $n$ of type $A_{p-1}$ to be the quotient of the KLR algebra $R(n)$ by the 2-sided ideal generated by $\sum_{\nu\in I^n}x_1^{\lambda_{\nu_1}-\lambda_{\nu_1+1}}e(\nu)$. We let $R^\lambda=\bigoplus_n R^\lambda(n)$.
\end{definition}

The importance of the cyclotomic KLR algebras is that they categorify highest-weight modules of $U_q(\mathfrak{sl}_{p})$ (or equivalently, $U_q(\mathfrak{gl}_{p})$), as conjectured in \cite{Khovanov2009} and proven (for general quantum groups $U_q(\mathfrak{g})$) by Kang and Kashiwara \cite{Kang2012}. An alternative proof has been given by Webster \cite{Webster2013}, and in the case of type $A_p$ a proof was also given by Vaz \cite{Vaz2013}.

\begin{thm}\label{thm:catoffiniterep}
The category of projective modules $\operatorname{p-mod}R^\lambda$ is a 2-representation of $U_q(\mathfrak{gl}(p))$ and satisfies
\[ K_0 (\operatorname{p-mod}R^\lambda)\otimes_{\Z[q,q^{-1}]} \C(q) \cong V_{p+1}(\lambda) \]
where $V_{p+1}(\lambda)$ is the irreducible module of $U_q(\mathfrak{gl}(p+1))$ of highest-weight $\lambda$, and $K_0(\mathcal{C})$ denotes the split Grothendieck group of the category $\mathcal{C}$.
\end{thm}

In fact, $\operatorname{p-mod}R^\lambda\cong \bigoplus_{\nu\in \N [I]} \operatorname{p-mod}R^\lambda (\nu)$ where $R^\lambda (\nu)$ is the subalgebra generated by $|\nu|$ strands, with $\nu_i$ strands labelled $i$. This gives a decomposition of the categorified highest-weight module into categorified weight spaces. Note that $\operatorname{p-mod}R^\lambda (0)$ corresponds to the highest weight space.

\begin{remark}
Without the tensor product with $\C(q)$, $K_0(\operatorname{p-mod} R^\lambda)$ is isomorphic to the integral form of the representation $V_{p+1}(\lambda)$, which is a module over $\Z[q,q^{-1}]$.
\end{remark}

Kang and Kashiwara \cite{Kang2012} also give an action of the categorified quantum group $\mathcal U_q(\mathfrak{gl}(p+1))$ as follows:

The functors \[E_i:\operatorname{p-mod}R^\lambda (\nu+i) \to \operatorname{p-mod}R^\lambda (\nu)\]
are defined by a graded shift of the restriction of modules by the inclusion $R^\lambda (\nu) \hookrightarrow R^\lambda (\nu+i)e(\nu,i)$ where $e(\nu,i)$ is the sum over all colourings of identity strands except with the last strand coloured $i$. In other words, $N\mapsto e(\nu,i)N$. Also
\[ F_i:\operatorname{p-mod}R^\lambda (\nu) \to \operatorname{p-mod}R^\lambda (\nu+i) \]
is defined by induction
\[ F_i(M) = R^\lambda(\nu+i)e(\nu,i)\otimes_{R^\lambda(\nu)} M. \]
Kang and Kashiwara show that these are exact, projective, and are left and right adjoint to each other.

Note also that these functors carry an action of the KLR algebra as follows: consider a sequence $\mathbf{i}=(i_1,\ldots,i_n)\in \{1,\ldots,p\}^n$, and consider the functor \[E_{\mathbf{i}}=E_{i_n}\cdots E_{i_1}:\operatorname{p-mod}R^\lambda (\nu+\mathbf{i}) \to \operatorname{p-mod}R^\lambda (\nu):N\mapsto e(\nu,\mathbf{i}) N\]
which has the effect of projecting to strands with the left-most coloured by colourings according to $\nu$, and the right-most coloured in the fixed sequence $\mathbf{i}$. Then the action of the KLR algebra on $n$ strands $R(\emph i)$ intertwines the structure as a $R^{\lambda}(\nu)$-module, so defines a natural transformation $E^n_\mathbf{i} \to  E^n_\mathbf{i}$.

There is a similar action for the $F$'s, with `added strands' giving the KLR action.

We have the following:

\begin{thm}[Cautis, Lauda \cite{Cautis2011},Webster \cite{Webster2013}]\label{thm:cycloklrrep}
The category $\operatorname{p-mod} R^\lambda$ forms a 2-repre\-sentation of the categorified quantum group $\mathcal U_q (\mathfrak{gl}(p+1))$.
\end{thm}

\begin{remark}\label{rem:CLapproach}Cautis and Lauda's approach depends upon the additional conjecture that the centre of $R^\lambda(\beta)$ is 1-dimensional in degree $0$ and $0$ in negative degrees.
\end{remark}

\subsection{Categorification of highest-weight modules of \texorpdfstring{$U_q(\mathfrak{gl}(\infty))$}{Uq(gl(infinity))}}
Here we extend the categorification of highest-weight modules of $U_q(\mathfrak{gl}(p))$ to a categorification of the highest-weight modules $V_\infty(\lambda)$ of $U_q(\mathfrak{gl}(\infty))$ defined in Section \ref{sec:glinfinity}.

Given a dominant weight $\lambda\in \Z^p$, there is a well-defined inclusion of non-unital $\kk$-algebras
\[ R^\lambda \hookrightarrow R^{(\lambda,0)} \]
given by taking the projectors $e(\nu)\in R^\lambda$ to $e(\nu)\in R^{(\lambda,0)}$.

\begin{definition}
Let $\lambda\in \Z^\infty$ be a dominant weight with all but finitely many entries equal to zero, and suppose that $\lambda_i=0$ for $i>p$. Then we define $R^\lambda$ to be the direct limit
\[ R^\lambda = \lim_\rightarrow \left( \xymatrix{ R^{(\lambda_1,\ldots,\lambda_p)} \ar[r] & R^{(\lambda_1,\ldots,\lambda_p,0)} \ar[r] & R^{(\lambda_1,\ldots,\lambda_p,0,0)} \ar[r] & \cdots}\right). \]
\end{definition}

For a more concrete description the algebra $R^\lambda(n)$ of degree $n$ is generated by $e(\nu)$ with $\nu\in \N^n$, $x_k$ with $1\leq k\leq n$ and $\tau_l$ with $1\leq l\leq n-1$, subject to the relations in Section \ref{sec:KLRalgebras}, and the cyclotomic relation
\[ x_1^{\lambda_{\nu_1}-\lambda_{\nu_1+1}}e(\nu)=0 \]
for all $\nu\in \N^n$. The algebra $R^\lambda$ is non-unital, since the sum over all orthogonal idempotents is infinite.

We define $\N [\N]$ to be the commutative semi-group consisting of elements
\[ \beta=\sum_{i\in \N} \beta_i\cdot i \]
where $\beta_i\in \N\cup \{0\}$ is $0$ for all but finitely many $i$. We let $|\beta|=\sum_i\beta_i$.

Given $\beta\in \N[\N]$ we define $R^\lambda(\beta)=R^\lambda \sum_\nu e(\nu)$ where the sum is over all $\nu\in \N^{|\beta|}$ such that the entry $i$ appears $\beta_i$ times in the entries of $\nu$. We also let \[ e(\beta,i)=\sum_{\nu\in \N^{|\beta|+1},\nu_{|\beta|+1}=i}e(\nu).\]

Denote by $\operatorname{p-mod}R^\lambda$ the category of finite-dimensional projective left $R^\lambda$-modules. As before, $\operatorname{p-mod}R^\lambda$ breaks into a coproduct
\[ \operatorname{p-mod}R^\lambda \cong \sum_\beta \operatorname{p-mod}R^\lambda(\beta) \]
over all $\beta\in \N[\N]$, corresponding to the coproduct $R^\lambda=\sum_\beta R^\lambda(\beta)$.

\begin{thm}\label{thm:catofinfinityrep}
The category $\operatorname{p-mod}R^\lambda$ is a 2-representation of $U_q(\mathfrak{gl}(\infty))$ which categorifies the representation $V_\infty (\lambda)$.
\end{thm}
\begin{proof}
For each $\beta\in \N[\N]$ and $i\in \N$ there exist functors
\[ E_i:\operatorname{p-mod}R^\lambda(\beta+i)\to \operatorname{p-mod}R^\lambda(\beta):N\mapsto q^{1-\lambda_i+\lambda_{i+1}+\beta_i-\beta_{i+1}}e(\beta,i)R^\lambda(\beta+i)\otimes_{R^\lambda(\beta+i)}N \]
\[ F_i:\operatorname{p-mod}R^\lambda(\beta)\to \operatorname{p-mod}R^\lambda(\beta+i):M\mapsto R^\lambda (\beta+i)e(\beta,i)\otimes_{R^\lambda(\beta)}M \]
which are well-defined since $R^\lambda(\beta+i)e(\beta,i)$ is a projective right $R^\lambda(\beta)$-module and $e(\beta,i)R^\lambda(\beta+i)$ is a projective left $R^\lambda(\beta)$-module by the main theorem of Kang and Kashiwara \cite{Kang2012}.

The 2-morphism $x_i:E_i1_\lambda\to E_i1_\lambda$ is given by left-multiplication by $x_{|\beta|+1}$ on $e(\beta,i)R^\lambda(\beta+i)\otimes_{R^\lambda(\beta+i)}N$.
For the action of $\tau_{ij}$ we note that
\[ E_iE_j(N)\cong e(\beta,j,i)R^{\lambda}(\beta+i+j)\otimes_{R^\lambda(\beta+i+j)}N \]
where $e(\beta,j,i)=\sum_{\nu\in \N^{|\beta|+2},\nu_{|\beta|+2}=i,\nu_{|\beta|+1}=j}e(\nu)$, and so $\tau_{ij}$ is given by left-multiplication by $\tau_{n+1}$.

Each algebra $R^\lambda(\beta)$ is exactly the same as $R^{(\lambda_1,\ldots,\lambda_p)}(\beta)$ where $p$ is such that $\beta_i=0$ for $i>p$, and therefore relations on the natural transformations between these functors are exactly the same as in the case of $U_q(\mathfrak{gl}(p))$. Hence by Theorem \ref{thm:cycloklrrep}, we must have $E_i$ left-adjoint to $F_i$ and the morphisms $\rho_i$ and $\sigma_{ij}$ map to isomorphisms for $i\neq j$.

Hence $\operatorname{p-mod}R^\lambda$ is a 2-representation of $U_q(\mathfrak{gl}(\infty))$. Since this is generated by the category $R^\lambda(0)\cong \operatorname{p-mod}(\kk)=\operatorname{Vect}_\kk$, we conclude that this is a categorification of $V_\infty(\lambda)$.
\end{proof}

\begin{remark}
As in remark \ref{rem:CLapproach}, this could be proven more directly if we knew that the centre of $R^\lambda(\beta)$ contained no negative degree elements, and a 1-dimensional space of degree $0$ elements.
\end{remark}

The categories $\operatorname{p-mod}R^\lambda(\beta)$ categorify the weight-space of $V_\infty(\lambda)$ of weight $\lambda-\sum_i \beta_i\alpha_i$.

\section{Categorification of \texorpdfstring{$\operatorname{Rep}(\mathfrak{gl}(m|n))$}{Rep(gl(m|n))}}\label{sec:catofrep}
Now we return to the programme of categorifying $\operatorname{Rep}(\mathfrak{gl}(m|n))$. The idea in this section is that we can use the decomposition of $\dot{U}_q^\infty(\mathfrak{gl}(\infty))$ into a direct sum indexed by dominant weights satisfying $\mu_{n+1}\leq m$, and categorify each direct summand separately.

Given a dominant weight $\mu$ with $\mu_{n+1}\leq m$, we have, by the results of the previous section, a 2-functor
\[ \Phi_\mu:\dot{\mathscr{U}}^\infty_Q(\mathfrak{gl}(\infty)) \to \operatorname{p-mod}R^\mu. \]

Let $\mathscr{E}(V_\infty(\mu))$ be the quotient of  $\dot{\mathscr{U}}^\infty_Q(\mathfrak{gl}(\infty))$ by the kernel of $\Phi_\mu$ (that is, quotient by the 2-morphisms acting as $0$ on $\operatorname{p-mod}R^\mu$). By construction, there is an integrable 2-repre\-sentation \[ \mathscr{E}(V_\infty(\mu))\to \operatorname{p-mod}R^\mu.\]  By Theorem \ref{thm:catofinfinityrep} we have the following:

\begin{thm}\label{thm:catEnd}
There is a canonical isomorphism
\[ K_0(\mathscr{E}(V_\infty(\mu)))\otimes_{\Z[q,q^{-1}]} \C(q)\cong \operatorname{End}_{fr}(V_\infty(\mu)). \]
\end{thm}
\begin{proof}
The algebra $\operatorname{End}_{fr}(V_\infty(\mu))$ inherits a basis from the canonical basis of $\dot U_q(\mathfrak{gl}(\infty))$. Each non-zero basis element corresponds to an indecomposable object of $\dot{\mathscr{U}}_Q(\mathfrak{gl}(\infty))$ that acts non-trivially on $\operatorname{p-mod}R^\mu$ by Theorem \ref{thm:catofinfinityrep}, so in particular its identity 2-morphism is not killed. This gives rise to an injective map 
\[ \operatorname{End}_{fr}(V_\infty(\mu)) \to K_0(\mathscr{E}(V_\infty(\mu)))\otimes_{\Z[q,q^{-1}]} \C(q) \]
sending the basis element to the indecomposable object. Any indecomposable object not reached by the map acts trivially on $\operatorname{p-mod}R^\mu$ by Theorem \ref{thm:catofinfinityrep}, so the map is also surjective.
\end{proof}

Hence we define the categorification of $\operatorname{Rep}(\mathfrak{gl}(m|n))$ as follows:

\begin{definition}
Let
\[ \mathscr{R}(\mathfrak{gl}(m|n))=\sum_{\mu\in H} \mathscr{E}(V_\infty(\mu)) \]
where the sum indicates the coproduct of 2-categories.
\end{definition}

\begin{thm}
The 2-category $\mathscr{R}(\mathfrak{gl}(m|n))$ satisfies
\[ K_0(\mathscr{R}(\mathfrak{gl}(m|n)))\otimes_{\Z[q,q^{-1}]}\C(q)\cong \operatorname{Rep}(\mathfrak{gl}(m|n)). \]
\end{thm}
\begin{proof}
This follows from Theorem \ref{thm:catEnd} and the decomposition of $\operatorname{Rep}(\mathfrak{gl}(m|n))$ into $\sum_{\mu\in H} \operatorname{End}(V_\infty(\mu))$ from Lemma \ref{lem:decomp} and Theorem \ref{thm:descriptionOfRep2}.
\end{proof}

\subsection{The Foam Description of the 2-category}
One advantage of this categorification is that each 2-category $\mathcal{E}(V_\infty(\lambda))$ admits a description by foams, since we have the equivalence of 2-categories $\operatorname{Foam}\cong \dot{\mathscr{U}}^\infty_Q(\mathfrak{gl}(\infty))$ and $\mathcal{E}(V_\infty(\lambda))$ occurs as a quotient of $\dot{\mathscr{U}}^\infty_Q(\mathfrak{gl}(\infty))$.

The relation
\[ x_1^{\lambda_{\nu_1}-\lambda_{\nu_1+1}}e(\nu)=0 \]
translates to a relation on foams, given by Figure \ref{fig:foamrelation}
\begin{figure}[h!]
$\raisebox{-50pt}{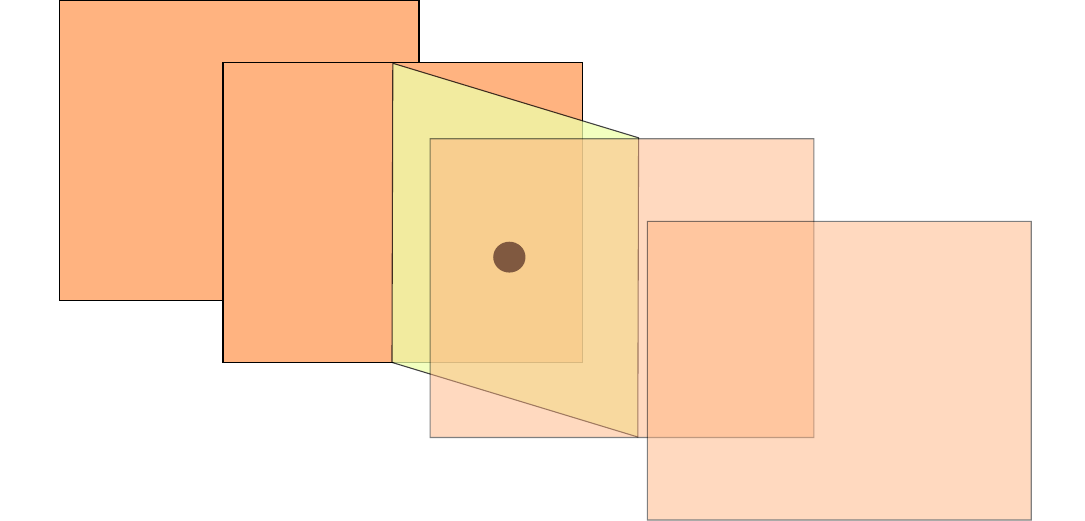}=0$
\caption{}\label{fig:foamrelation}
\end{figure}
where $\lambda=(\lambda_1,\lambda_2,\ldots)$, and $\nu_1=i$. So sufficiently large numbers of dots on horizontal facets give the $0$ foam.

This effectively means that foams provide a diagram calculus for natural transformations of compositions of the endofunctors $F_i1_\lambda$ and $E_i1_\lambda$ in the category of finite-dimensional projective left $R^\lambda$-modules.

\subsection{Braiding}\label{sec:catBraiding}
In Section \ref{sec:braiding}, we defined a braiding on $\operatorname{Rep}(\mathfrak{gl}(m|n))$ by the braiding on $\dot U_q^\infty(\mathfrak{gl}(\infty))$. There is a categorification of this braiding $\mathcal{T}_i1_\lambda$ as a chain complex in $\dot{\mathscr{U}}_Q^\infty(\mathfrak{gl}(\infty))$ by
\[\xymatrix{F_i^{(\lambda_{i}-\lambda_{i+1})} 1_\lambda \ar[r]^-{d_1} & qF_i^{(\lambda_i-\lambda_{i+1}+1)}E_i 1_\lambda\ar[r]^-{d_2} & \cdots \ar[r]^-{d_t}& q^{s}F_i^{(\lambda_i-\lambda_{i+1}+s)}E_i^{(s)}1_\lambda \ar[r]^-{d_{t+1}} &\cdots} \]
if $\lambda_i-\lambda_{i+1}\geq 0$, with grading shifted by $q^{(m-n)\lambda_i\lambda_{i+1}-\lambda_i}$, and by
\[\xymatrix{E_i^{(\lambda_{i+1}-\lambda_i)} 1_\lambda \ar[r]^-{d_1} & qE_i^{(\lambda_{i+1}-\lambda_{i}+1)} F_i1_\lambda\ar[r]^-{d_2} & \cdots \ar[r]^-{d_t}& q^{s}E_i^{(\lambda_{i+1}-\lambda_{i}+s)}F_i^{(s)}1_\lambda \ar[r]^-{d_{t+1}} &\cdots} \]
if $\lambda_{i+1}-\lambda_i\geq 0$, with grading shifted by $q^{(m-n)\lambda_i\lambda_{i+1}-\lambda_i}$. The differentials $d_t$ in the second complex are explicitly defined using `thick calculus' by Lauda, Queffelec and Rose \cite[Section 2.2]{Lauda2012}. Each complex is finite for the same reason the sum defining the braiding in Section \ref{sec:braiding} is finite: $E_i^{(s)}1_\lambda$ and $F_i^{(s)}1_\lambda$ are $0$ for sufficiently large $s$ in the category $\dot U_q^\infty(\mathfrak{gl}(\infty))$. The complex $\mathcal{T}_i1_\lambda$ is invertible up to chain homotopy, so there also exists a chain complex $1_\lambda\mathcal{T}_i^{-1}$ with $T_i^{-1}T_i1_\lambda\sim 1_\lambda$.

The complexes $\mathcal{T}_i1_\lambda$ and $\mathcal{T}_i^{-1}1_\lambda$ descend to a complex over $\mathcal{E}(V_\infty(\mu))$ for each $\mu$. Given a link $L$ coloured with representations $\bigwedge^k_q \C^n_q$ of $U_q(\mathfrak{sl}(n))$, it is therefore possible to define the coloured $\mathfrak{sl}_n$ link homology $H^{i,j}_n(L)$ as follows: write the link $L$ as a ladder diagram, with a positive crossing taken to be the complex $T_i1_\lambda$ and a negative crossing to be $T_i^{-1}1_\lambda$ where $\lambda$ is the colouring of the uprights immediately below the crossing. Project this complex to $\sum\mathcal{E}(V_\infty(n,n,\ldots,n,0,\ldots))$ where the sum is over the number of $n$'s appearing in the entries of the highest-weight. The complex is non-zero in exactly one of the summands. Then by Section \ref{sec:foamssln}, we can apply the tautological functor $\operatorname{Foam}(1_{(n,n,\ldots,n,0,\ldots)},-)$ to each of the objects in the chain complex. The result is now a complex of $\kk$-vector spaces (or abelian groups) and we define $H_n^{i,j}(L)$ to be the homology of this complex. The result depends only on the coloured link $L$, and is independent of other choices. This turns out to be exactly the $\mathfrak{sl}(n)$ homologies defined for colour $\C^n_q$ by Khovanov and Rozansky \cite{Khovanov2004} and in higher colours $\bigwedge^k \C^n_q$ by Wu \cite{Wu2009}. This is the approach to $\mathfrak{sl}(n)$ link homology developed by Lauda, Queffelec and Rose \cite{Lauda2012} for $n=2,3$ and Queffelec and Rose \cite{Queffelec}.

The main problem with making this approach work in the general case of $\mathfrak{gl}(m|n)$ is that it is not possible to write a closed MOY diagram as a ladder diagram, since there is no $i>0$ such that $\bigwedge^i_q\C^{m|n}_q\cong \C_q$ and $(\C^{m|n}_q)^*$ does not occur as $\bigwedge^i_q\C^{m|n}_q$ for any $i$. To define a link homology, one would at least need a way to include duals in the diagram calculus. An approach to this in the decategorified setting has been demonstrated by Queffelec and Sartori \cite{Queffelec2014}, \cite{Queffelec2015}, using a doubled Schur algebra.

\subsection{Examples And Non-local Behaviour}\label{sec:non-local}

To understand this categorification, let us give some examples. 

Consider the identity $I$ on the standard representation $\C_q^{m|n}\to \C_q^{m|n}$. This is represented as a ladder by a single upright labelled $1$, that is, by the element $1_{(1,0,\ldots)}$ in $\dot{U}_q^{(m|n)} (\mathfrak{gl}(\infty))$. So in $\mathscr{R}(\mathfrak{gl}(m|n))$, the space of 2-morphisms that map $I\to I$ are given by the 2-morphisms in $\mathscr{E}(V_\infty(1,0,\ldots))$ as $V_\infty(1,0,\ldots)$ is the only representation with a $(1,0,\ldots)$-weight space. But then $(1,0,\ldots)$ is also the highest-weight, so there are no non-trivial endomorphisms of $1_{(1,0,\ldots)}$ in $\mathscr{E}(V_\infty(1,0,\ldots))$. Therefore there is only a 1-dimensional space of 2-morphisms $I\to I$.

In general, the identity $\bigwedge^k_q \C^{m|n}_q \to \bigwedge^k_q \C^{m|n}_q$ has a 1-dimensional space of 2-morphisms, since it is the space of 2-morphisms $1_{(k,0,\ldots)}\to 1_{(k,0,\ldots)}$ in $\mathscr{E}(V_\infty(k,0,\ldots))$.

However, note that not all ladders with label $1$ have 1-dimensional endomorphism spaces. Consider instead the identity $I:\C^{m|n}_q\otimes \C^{m|n}_q\to \C^{m|n}_q\otimes \C^{m|n}_q$. This is represented by $1_{(1,1,0,\ldots)}$ in $\dot U_q^{(m|n)}(\mathfrak{gl}(\infty))$. In this case there is a $(1,1,0,\ldots)$ weight-space in two representations, namely $V_\infty(2,0,\ldots)$ and $V_\infty(1,1,0,\ldots)$. Then the space of 2-morphisms $1_{(1,1,0,\ldots)}\to 1_{(1,1,0,\ldots)}$ in $\mathscr{E}(V_\infty(2,0,\ldots))$ is 4-dimen\-sional, since each strand can have at most one dot on it. However, the space of $2$-morphisms $1_{(1,1,0,\ldots)}\to 1_{(1,1,0,\ldots)}$ in $\mathscr{E}(V_\infty(1,1,0,\ldots))$ is again $1$-dimensional. So this time $\operatorname{End}(I)$ is $1$-dimensional if $(m,n)=(1,0)$, $4$-dimensional if $(m,n)=(0,1)$, and $5$-dimensional otherwise.

This demonstrates an important property of this categorification: the relations on foams are non-local. In the special case of $\mathcal{E}(V_\infty(n,n,\ldots,n,0,\ldots))$, it turns out every 1-labelled facet can be treated the same way: an $n$-dotted $1$-labelled facet is $0$. However, in a category like $\mathcal{E}(V_\infty(3,1,0,\ldots))$, it is clear that there are no dots allowed on the $1$-facet appearing in the identity foam of $1_{(3,1,0,\ldots)}$ since this is the highest weight space, but applying $F_1F_2$ to get to the $(2,1,1,\ldots)$ weight space, there is an endomorphism which has a dotted facet, since the cyclotomic relation gives $x_2^2e(12)=0$, see Figure \ref{fig:threeoneexample}. There is also an endomorphism corresponding to a $\tau$ foam, which maps $F_1F_21_{(3,1,0,\ldots)}$ to $F_2F_11_{(3,1,0,\ldots)}$. Therefore, these two $1$-labelled facets play different roles.

\begin{figure}[h]\caption{A dot on the other yellow facet results in the $0$ foam.}\label{fig:threeoneexample}
\raisebox{-50pt}{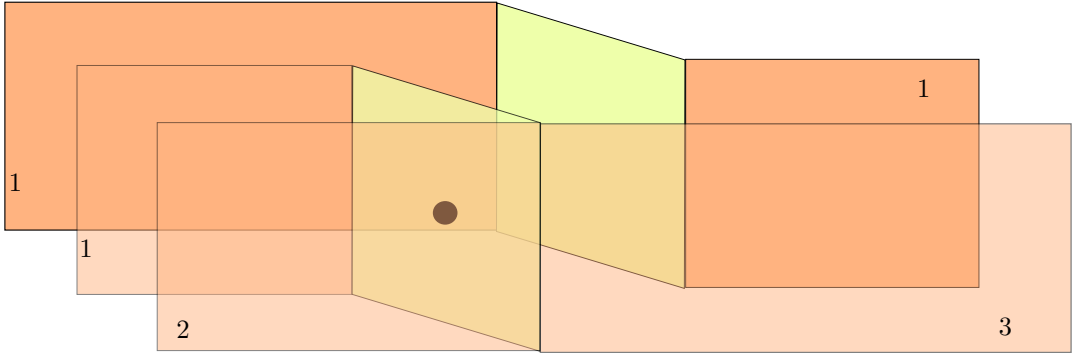}
\end{figure}

Another aspect of the non-locality is that there is a dependence on what has come above or below the ladder in question. For instance, we said there is a $5$-dimensional space of 2-morphisms $I:\C^{m|n}_q\otimes \C^{m|n}_q\to \C^{m|n}_q\otimes \C^{m|n}_q$, but if this map is preceded or followed by a map $\bigwedge^2_q \C^{m|n}_q\to \C^{m|n}_q\otimes \C^{m|n}_q$ or $\C^{m|n}_q\otimes \C^{m|n}_q \to \bigwedge^2_q \C^{m|n}_q$, then it is only $4$-dimensional, since we are then working solely in $\mathcal{E}(V_\infty(2,0,\ldots))$.

\section{Important special cases}\label{sec:specialcases}
\subsection{Relationship to Foams for \texorpdfstring{$\mathfrak{sl}(n)$}{sl(n)} homology}\label{sec:foamssln}
This section serves both as an example and as an explanation of differences from the usual definition of $\mathfrak{sl}(n)$ foams. We have $\operatorname{Rep}(\mathfrak{gl}(n))\cong \operatorname{Rep}(\mathfrak{gl}(n|0))$.

Thus we have
\begin{equation}\label{eqn:decomp}\operatorname{Rep}(\mathfrak{gl}(n))\cong \sum_{\mu\in H} \operatorname{End}_{fr}(V_\infty(\mu))\end{equation}
where $H$ is the set of dominant weights with $\mu_i\leq n$ for all $i$.

We can recover a description of $\operatorname{Rep}(\mathfrak{sl}(n))$ by converting $\mathfrak{gl}(n)$ weights to $\mathfrak{sl}(n)$ weights by the mapping $(\lambda_1,\ldots,\lambda_n)\mapsto (\lambda_1-\lambda_2,\lambda_2-\lambda_3,\ldots,\lambda_{n-1}-\lambda_n)$. The remarkable special case of $\operatorname{Rep}(\mathfrak{sl}(n))$ has the property that it is possible to describe closed MOY diagrams using only ladders that are oriented upwards. This is due to the isomorphism
\[ \left( \bigwedge_q^i \C^n_q\right)^* \cong \bigwedge_q^{n-i} \C^n_q \]
in $\operatorname{Rep}(\mathfrak{sl}(n))$, and also $\bigwedge^n_q \C^n_q \cong \C_q$. This means for instance, that the closed circle can be described as $EF1_{(n,0,0,\ldots)}$, or in ladder diagrams as

\[ \begin{tikzpicture}[baseline=-0.65ex]
\draw (-0.5,-1) -- (-0.5,1);
\draw (0.5,-0.2) -- (0.5,0.2);
\draw (-0.5,-0.6) -- (0.5,-0.2);
\draw (0.5,0.2) -- (-0.5,0.6);
\draw (-0.5,-1.25) node {$n$};
\draw (0.5,-1.25) node {};
\draw (-0.5,1.25) node {$n$};
\draw (1,1.25) node {};
\draw (0,-0.6) node {};
\draw (0,0.6) node {};
\draw (-1.1,0) node {$n-1$};
\draw (1,0) node {$1$};
\end{tikzpicture}. \]

It is then evident from the ladder axioms that this is equal to

\[ [n] \begin{tikzpicture}[baseline=-0.65ex]
\draw (-0.5,-1) -- (-0.5,1);
\draw (-0.5,1.25) node {$n$};
\draw (-0.5,-1.25) node {$n$};
\end{tikzpicture}. \]

This property was used by Cautis, Kamnitzer and Morrison \cite{Cautis2012} to deduce that all MOY relations are consequences of relations in $\dot{U}_q^\infty(\mathfrak{gl}(\infty))$ with the additional relation that a strand is $0$ if it has colour higher than $n$, corresponding to the decomposition in equation \ref{eqn:decomp}.

Using this, Lauda, Queffelec and Rose \cite{Lauda2012}, defined a category of foams $N\operatorname{Foam}_n$ by the relation that a foam is $0$ if it has a facet labelled $1$ with $n$ dots on it. 

In fact, this corresponds to our 2-category $\mathcal{E}(V_\infty(n,n,\ldots,n,0,\ldots))$ where the entries in the highest weight sum to $N$. The reason this 2-category is suitable for studying link homology is that the resolutions of any link written as ladder diagrams will have some number of uprights coloured $n$ at the bottom, and at the top. Hence the ladder represents an element of \[1_{(n,n,\ldots,n,0,\ldots)}\dot U_q^\infty(\mathfrak{gl}(\infty))1_{(n,n,\ldots,n,0,\ldots)}\cong 1_{(n,n,\ldots,n,0,\ldots)}\operatorname{End}_{fr}(V_{\infty}(n,n,\ldots,n,0,\ldots))1_{(n,n,\ldots,n,0,\ldots)}\]
since $V_\infty(n,n,\ldots,n,0,\ldots)$ is the only highest weight space with highest weight in $H$ containing a non-zero $(n,n,\ldots,n,0,\ldots)$ weight space. Thus if we want to think of diagrams that are local pictures of a link diagram, then it makes sense to restrict only to $\mathcal{E}(V_\infty(n,n,\ldots,n,0,\ldots))$, rather than considering the whole of $\mathscr{R}(\mathfrak{gl}(n|0))$.

The highest-weight space of $V_\infty(n,n,\ldots,n,0,\ldots)$ is 1-dimensional, so the space of its endomorphisms is also 1-dimensional, since an endomorphism of the highest-weight space is determined by where the element $1$ gets sent.

Similarly, the category $\operatorname{p-mod} R^{(n,n,\ldots,n,0,\ldots)}(0)$ is equivalent to the category $\operatorname{Vect}_\kk$ of vector spaces over the ground field $\kk$. Therefore a functor $G\in \mathcal{E}(V_\infty(n,n,\ldots,n,0,\ldots))$ on $\operatorname{p-mod} R^{(n,n,\ldots,n,0,\ldots)}(0)$ is determined by where $\kk$ is sent. In fact, it is possible to determine the image of $G$ using foams by the Yoneda lemma, since
\[ G\kk \cong \operatorname{Nat}(\operatorname{Hom}(\kk,-),G) \cong \operatorname{Nat}(\id_{\operatorname{Vect}_\kk},G)\cong \operatorname{Foam}(1_{(n,n,\ldots,n,0,\ldots)},G)\]
where $\operatorname{Nat}$ is the vector space of natural transformations, all isomorphisms are isomorphisms of vector spaces, and $\operatorname{Foam}(1_{(n,n,\ldots,n,0,\ldots)},G)$ is the space of foams from the ladder $1_{(n,n,\ldots,n,0,\ldots)}$ to the ladder representing the 1-morphism $G$. This shows that the tautological functor introduced by Bar-Natan \cite{Bar-Natan2004} is a natural categorification of the evaluation of a MOY diagram. In fact, the same isomorphism works if the ground ring $\kk$ is $\Z$, since then $\operatorname{p-mod}R^{(n,n,\ldots,n,0,\ldots)}(0)$ is the category of free abelian groups.

Note that restricting to $\mathcal{E}(V_\infty(n,n,\ldots,n,0,\ldots))$ has a number of advantages. All facets labelled $1$ can be thought of as having essentially the same `role', and we have the uniform rule that $n$ dots on a 1-labelled facet gives the $0$ foam. A similar rule applies to all higher coloured facets. This remarkable feature makes $\mathfrak{sl}(n)$ foams much nicer to work with, and means we can forget the ladder structure and work with a more flexible 2-category. Dot migration relations can be defined which allow all facets to be treated equally, rather than separated strictly into those that sit below uprights (orange) and those that sit below rungs (yellow). This is described by Queffelec and Rose \cite[Section 3.1]{Queffelec}.

\subsection{Symmetric powers of the standard representation of \texorpdfstring{$\mathfrak{sl}(n)$}{sl(n)}}\label{sec:symmetric}
The special case of $\mathfrak{gl}(n|0)$ is shown above to be particularly interesting. Another important case is $\mathfrak{gl}(0|n)$. In fact, we have $U_q(\mathfrak{gl}(n|0))\cong U_q(\mathfrak{gl}(0|n))$ by an isomorphism that takes $q\mapsto q^{-1}$.

The standard representation $\C^{0|n}_q$ consists of only odd-degree elements, so in fact we have $\bigwedge^i_q \C^{0|n}_q\cong S^i_q \C^n_q$ where $S^i_q$ is the $i$th symmetric power.

We have \[ \operatorname{Rep}(\mathfrak{gl}(0|n))\cong \sum_{\mu\in H} \operatorname{End}_{fr}(V_\infty(\mu))\]
where $H$ consists of dominant weights with at most $n$ parts by Lemma \ref{lem:decomp}. This gives a diagram calculus for the category $\operatorname{Sym}(\mathfrak{sl}(n))$, consisting of ladder diagrams satisfying an additional relation corresponding to the condition on $H$.

In the case of $n=2$, diagrams for symmetric powers of $\mathfrak{sl}(2)$ have been studied by Rose and Tubbenhauer \cite{Rose2015}. They use a slightly different set-up, where they mix the ladder diagrams with the Temperley-Lieb algebra. This means they can close MOY diagrams, and define a link polynomial corresponding symmetric powers of $\mathfrak{sl}(n)$. By closing the diagram, the extra relation disappears and is replaced by a relation that ties together ladder diagrams and the diagram from the Temperley-Lieb algebra, namely
\[\frac{1}{[2]}\begin{tikzpicture}[baseline=-0.65ex]
\draw (-0.5,-1) -- (-0.5,1);
\draw (0.5,-1) -- (0.5,-0.6);
\draw (0.5,0.6) -- (0.5,1);
\draw (-0.5,-0.2) -- (0.5,-0.6);
\draw (0.5,0.6) -- (-0.5,0.2);
\draw (-0.5,-1.25) node {$1$};
\draw (0.5,-1.25) node {$1$};
\draw (-0.5,1.25) node {$1$};
\draw (0.5,1.25) node {$1$};
\draw (0,-0.7) node {};
\draw (0,0.7) node {};
\draw (-0.75,0) node {};
\draw (0.75,0) node {};
\end{tikzpicture}=
\begin{tikzpicture}[baseline=-0.65ex]
\draw (-0.5,-1)-- (-0.5,1);
\draw (0.5,-1) -- (0.5,1);
\draw (-0.5,-1.25) node {$1$};
\draw (0.5,-1.25) node {$1$};
\draw (-0.5,1.25) node {$1$};
\draw (0.5,1.25) node {$1$};
\end{tikzpicture}
+\frac{1}{[2]} \begin{tikzpicture}[baseline=-0.65ex]
\draw (-0.5,-1) to [out=90,in=180] (0,-0.2) to [out=0,in=90] (0.5,-1);
\draw (-0.5,1) to [out=270,in=180] (0,0.2) to [out=0,in=270] (0.5,1);
\draw (-0.5,-1.25) node {$1$};
\draw (0.5,-1.25) node {$1$};
\draw (-0.5,1.25) node {$1$};
\draw (0.5,1.25) node {$1$};
\end{tikzpicture}.
 \]

Note that this relation identifies the Jones-Wenzl projector $\operatorname{JW}_2$ (on the right-hand side) in the Temperley-Lieb algebra with the idempotent in $\dot U_q^\infty(\mathfrak{gl}(\infty))$ projecting to $\operatorname{End}_{fr}(V_\infty(2,0,0,\ldots))$. In general, this means the $k$th Jones-Wenzl projector $\operatorname{JW}_k$ projecting to $S^k_q\C^n_q$ inside $\bigotimes_{i=1}^k \C^n_q$ is identified with the projection to $\operatorname{End}_{fr}(V_\infty(k,0,0,\ldots))$ in $\dot U^\infty_q(\mathfrak{gl}(\infty))$.

Similarly, one can mix exterior and symmetric powers of $\mathfrak{sl}(n)$ in a single graphical calculus. The category which mixes exterior powers and symmetric powers seems a lot easier to describe completely, involving only a single `dumbbell' relation that relates the two types of representation. This was proved by Tubbenhauer, Vaz and Wedrich \cite{Tubbenhauer2015}.

The connection between knot homology for exterior powers and for symmetric powers is particularly interesting because of a conjectural relationship at the level of HOMFLY homology due to Gukov and Sto\v{s}i\'{c} \cite{Gukov2011}:
\begin{conjecture}[Gukov, Sto\v{s}i\'{c}] For a knot $K$, there is an isomorphism between anti-symmetric coloured HOMFLY homology and symmetric coloured HOMFLY homology
\[H^{\Lambda^r}_{i,j,*}(K)\cong H^{S^r}_{i,-j,*}(K).\]
\end{conjecture}
This conjecture is based on computational evidence, although the homology theory $H^{S^r}_{i,j,k}(K)$ is not yet formally defined.
\subsection{Non-local relations in Heegaard Floer homology}\label{sec:heegaardfloer}
Part of our interest in categorifying $\operatorname{Rep}(\mathfrak{gl}(m|n))$ was to understand the case $m=n=1$. The knot polynomial associated to the standard representation of $U_q(\mathfrak{gl}(1|1))$ is the Alexander polynomial (see \cite{Sartori2013}, or \cite{Grant2014}). There are several known link homologies with graded Euler characteristic equal to the Alexander polynomial, such as Heegaard Floer knot homology and Instanton Floer knot homology, but one arising purely from categorified representation theory has not yet been found. Heegaard Floer knot homology instead is usually defined in terms of analytic geometry.

Gilmore \cite{Gilmore2010} gave a description of Heegaard Floer knot homology that more closely resembles some constructions of Khovanov-Rozansky $\mathfrak{sl}(n)$ homology, in the sense that it was based on Ozsv\'{a}th and Szab\'{o}'s cube of resolutions description of Heegaard Floer homology \cite{Ozsvath2009a}, and uses MOY diagrams to prove invariance of the homology under Reidemeister moves. It is interesting that in Gilmore's algebraic setting, it is also necessary to make use of non-local relations. 

We hope that is is possible to use our categorification $\mathscr{R}(\mathfrak{gl}(1|1))$ to define a link homology theory categorifying the Alexander polynomial. It would be especially interesting if this theory was related to Gilmore's construction of Heegaard Floer knot homology.

\printbibliography
\end{document}